\documentclass[10pt]{amsart}
\usepackage{amssymb}
\usepackage{graphicx}
\usepackage{dsfont}
\usepackage{xcolor}
\usepackage[colorlinks,citecolor=darkblue,urlcolor=blue,bookmarks=false,hypertexnames=true]{hyperref}
\definecolor{darkblue}{HTML}{004C93} 
\definecolor{MainRed}{rgb}{.6, .1, .1}

\numberwithin{equation}{section}

\newcommand{\R}{\mathbb{R}}

\newcommand{\Sp}{{\rm Sp}}

\newcommand{\ve}{\varepsilon}
\newcommand{\vp}{\varphi}
\newcommand{\Om}{\Omega}
\newcommand{\bal}{\begin{aligned}}
\newcommand{\eal}{\end{aligned}}
\newcommand{\ben}{\begin{equation}}
\newcommand{\een}{\end{equation}}
\newcommand{\be}{\begin{equation*}}
\newcommand{\ee}{\end{equation*}}

\newcommand{\mkl}{\mu_{k}^{\lambda}}

\newcommand{\mpl}{\mu_{p}^{\lambda}}
\newcommand{\mil}{\mu_{i}^{\lambda}}

\newcommand{\mkle}{\mu_{k,\ve}}

\newcommand{\vkl}{\varphi_{k}^{\lambda}}
\newcommand{\vpl}{\varphi_{p}^{\lambda}}
\newcommand{\vil}{\varphi_{i}^{\lambda}}
\newcommand{\vjl}{\varphi_{j}^{\lambda}}

\newcommand{\Fke}{F_{k,\ve}}

\def \rr {\mathbb{R}}
\def \crit{2^{*}}

\def \bb{\hbox}

\def \xl{x_\lambda}

\def\la{\lambda}

\theoremstyle{plain}

\newtheorem{theo}{Theorem}[section]%[chapter]
\newtheorem{prop}{Proposition}[section]%[chapter]
\newtheorem{defi}{Definition}[section]%[chapter]
\newtheorem{lemme}{Lemma}[section]%[chapter]
\newtheorem{corol}{Corollary}[section]%[chapter]
\theoremstyle{remark}

\title[]{Least-energy solutions of the Br\'ezis-Nirenberg problem in the non-coercive case in dimension $3$}

\author{Hussein Cheikh Ali}

\author{Bruno Premoselli}

\address{Universit\'e Libre de Bruxelles, Service d'analyse, CP 218, Boulevard du Triomphe, B-1050 Bruxelles, Belgique.}

\email{bruno.premoselli@ulb.be}

\address{Laboratoire Paul Painlev\'e, Universit\'e de Lille, Cit\'e Scientifique - B\^atiment M2 59655 Villeneuve d'Ascq Cedex, France}

\email{frederic.cheikh-ali@univ-lille.fr}

\thanks{The second author was supported by the Fonds Th\'elam, an ARC Avanc\'e 2020 grant and an EoS FNRS grant.} 

\date{3 September, 2025}

\begin{document}

\maketitle

	\begin{abstract}
Let $\Om$ be a bounded, smooth domain of $\R^n, n \ge 3$ and $\lambda \ge 0$. We consider in this paper the celebrated Br\'ezis-Nirenberg problem:
\begin{equation}\label{eq:critlambda:abs} \tag{$*$}
	\left\{\begin{aligned}
		-\Delta u -\lambda u & =\left|u\right|^{\crit-2}u &\hbox{ in } \Omega, \\
		u & = 0 \quad \text{ in } \partial \Om, 
			\end{aligned}\right.
\end{equation}
where $2^* = \frac{2n}{n-2}$. When $n=3$ we investigate the existence of \emph{least-energy solutions} for \eqref{eq:critlambda:abs}, that we define as having the lowest $L^{2^*}(\Om)$ norm among all non-zero solutions of \eqref{eq:critlambda:abs}. We prove that least-energy solutions of \eqref{eq:critlambda:abs} exist when $\lambda$ belongs to a left neighbourhood of any eigenvalue of $-\Delta$ that we explicitly characterise by a positive mass assumption. We obtain in particular the first \emph{existence} result for \eqref{eq:critlambda:abs} on a general smooth bounded domain $\Om$ when $n=3$ and $\lambda \ge \Lambda_1$. In order to do this we introduce, for any $\lambda \ge  0$, a new variational problem inspired from spectral-theoretic considerations which is as follows: for any $u \in L^{2^*}(\Om), u>0$ a.e., we consider the principal eigenvalue of $- \Delta-\lambda$ on the weighted  space $L^2(\Om, u^{2^*-2} dx)$, whose value we then minimise over the set of normalised weights $\Vert u \Vert_{2^*} = 1$. When $\lambda \ge \Lambda_1$ this defines a new, non-smooth variational problem for which we develop a variational theory. We prove that its minimisers exist under the aforementioned positive mass assumption and that they yield \emph{least-energy} solutions of \eqref{eq:critlambda:abs}. With this framework we also obtain new results in the higher-dimensional case $n \ge 4$, where we show that the energy function of \eqref{eq:critlambda:abs} is discontinuous exactly at the eigenvalues of $- \Delta$.
\end{abstract}

\tableofcontents

\section{Introduction and statement of the results}

\subsection{Least-energy solutions of the Br\'ezis-Nirenberg problem.} Let $\Omega$ be a smooth bounded domain of  $\rr^n$,  $n\geq 3$. We denote by $H_0^1(\Om)$ the closure of $C_c^{\infty}(\Om)$ for the norm $\Vert u \Vert_{H^1_0}^2 =  \int_{\Om} |\nabla u|^2\, dx$. If $u \in H^1_0(\Om)$, we denote by $\Vert u \Vert_{2^*}$ its $L^{2^*}(\Om)$ norm. We let  $0< \lambda_1 < \lambda_2 \le \cdots$ be the eigenvalues of $-\Delta$ in $H^1_0(\Om)$ counted with multiplicity and $\Lambda_k$, for $k \ge 1$, be the eigenvalues counted without multiplicity: hence $\Lambda_1 < \Lambda_2< \cdots$ and $\Sp(-\Delta) = \{ \Lambda_k\}_{k \ge 1}$. We let $\Lambda_0 = 0$ by convention and we let $K_n$ be the optimal constant for Sobolev's inequality in $\R^n$: 
\begin{equation} \label{defK0}
 \frac{1}{K_n^2} = \inf_{\vp \in C^\infty_c(\R^n) \backslash \{0\}} \frac{\int_{\R^n}|\nabla \vp|^2 dx}{\big(\int_{\R^n} |\vp|^{\frac{2n}{n-2}}dx \big)^{\frac{n-2}{n}}}. 
 \end{equation}
Let $\lambda \ge 0$. We consider in this paper the celebrated Br\'ezis-Nirenberg problem: 
\begin{equation}\label{eq:critlambda}
	\left\{\begin{aligned}
		-\Delta u -\lambda u & =\left|u\right|^{\crit-2}u &\hbox{ in } \Omega, \\
		u & = 0 \quad \text{ in } \partial \Om
	\end{aligned}\right.
\end{equation}
where $\crit:=\frac{2n}{n-2}$ is the critical Sobolev exponent. The investigation of \eqref{eq:critlambda} was initiated in the seminal paper \cite{BN} in the coercive case $\lambda \in [0, \Lambda_1)$, has since then  attracted uninterrupted attention and has generated an abundant literature in the last forty years. In the coercive case existence and multiplicity results of positive and sign-changing solutions are in \cite{BN, CapozziFortunatoPalmieri, CeramiSoliminiStruwe, ChenZou, ClappWeth2, FortunatoJannelli, HebeyVaugonNodal, RoselliWillem, TavaresYouZou}. In the non-coercive cas $\lambda \ge \Lambda_1$ every solution of \eqref{eq:critlambda} changes sign. Non-trivial solutions exist when $\lambda >0$ and $n \ge 5$ \cite{CapozziFortunatoPalmieri} or when $n = 4$ and $\lambda \in \Sp(-\Delta)$ has multiplicity at most $n$ (which includes in particular the case $\lambda = \Lambda_1$), see \cite{ClappWeth2}. Further existence and multiplicity results for solutions in the non-coercive case are  in \cite{ABP2, CeramiFortunatoStruwe, ClappWeth2, DevillanovaSolimini, FortunatoJannelli, SchechterZou, Solimini} and the existence of infinitely many solutions of \eqref{eq:critlambda} was proven when $n \ge 7$ and $\lambda >0$ in \cite{SchechterZou} (previous results were in \cite{DevillanovaSolimini} and, when $\Om$ is the unit ball, in \cite{FortunatoJannelli, Solimini}).

\medskip

In this paper we investigate the existence and the qualitative behavior of \emph{least-energy solutions} of \eqref{eq:critlambda}. We define the energy function of \eqref{eq:critlambda} as follows: for $\lambda \ge 0$ fixed, we let
\begin{equation} \label{minimalenergy}
\mathcal{E}_\lambda(\Om) = \inf_{u \in H^1_0(\Om)\backslash \{0\} \text{ solves } \eqref{eq:critlambda}} \int_{\Om}|u|^{2^*}dx.  
 \end{equation}
If no solutions of \eqref{eq:critlambda} exist we let $\mathcal{E}_{\lambda}(\Om) = - \infty$. We say that $\mathcal{E}_\lambda(\Om)$ is attained if there exists a non-zero solution $u \in H^1_0(\Om)$ of \eqref{eq:critlambda}  such that $\int_\Om |u|^{2^*}dx = \mathcal{E}_\lambda(\Om)$, and any such  $u$ will be called \emph{a least-energy solution of \eqref{eq:critlambda}}\footnote{Let us insist on this point to avoid any possible ambiguity. The notion of \emph{least-energy solution} that we consider in this paper is fairly simple: it is a non-zero solution $u$ of \eqref{eq:critlambda} which has least energy among \emph{all} non-zero solutions of \eqref{eq:critlambda}: equivalently, any other non-zero solution $v$ of \eqref{eq:critlambda} satisfies $\int_{\Om} |u|^{2^*}\, dx \le  \int_{\Om} |v|^{2^*}\, dx$. When $ \lambda \in [0,\Lambda_1)$, for instance, the positive ground states constructed in \cite{BN} are least-energy solutions. When $\lambda \ge \Lambda_1$ and every solution of \eqref{eq:critlambda} changes sign, however, our notion of least-energy solution should not be confused with the notion of a variational minimal nodal solution considered e.g. in \cite{BartschWethWillem}. In particular least-energy solutions in our sense, as we shall see, may have more than two nodal domains.}.  We focus more precisely on least-energy solutions $u$ whose energy satisfies $\int_{\Om} |u|^{2^*} \, dx <K_n^{-n}$. In the coercive case $\lambda \in [0, \Lambda_1)$ the situation is very well-understood: the results of \cite{BN, Druetdim3} show that $\mathcal{E}_\lambda(\Om)$ is attained at a positive least-energy solution with $\mathcal{E}_\lambda(\Om) < K_n^{-n}$ if and only if $\lambda \in (\lambda_{0,*}, \Lambda_1)$, where $\lambda_{0,*} = 0$ if $n \ge 4$ and $\lambda_{0,*} >0$ is given by \eqref{deflambdastar} below if $n=3$ (see \cite{BN} when $n \ge 4$ and \cite{Druetdim3} when $n=3$). In the non-coercive case $\lambda \ge \Lambda_1$ the existence of a non-trivial least-energy solution of \eqref{eq:critlambda}  was proven in \cite{SzulkinWethWillem} when $n \ge 5$ and when [$n=4$ and $\lambda \not \in \Sp(-\Delta)$], and the solutions of \cite{SzulkinWethWillem} also satisfy $\mathcal{E}_\lambda(\Om) < K_n^{-n}$. To the best of our knowledge, when $n=3$ and $\lambda \ge \Lambda_1$, the existence of non-trivial solutions of \eqref{eq:critlambda} -- let alone of \emph{least-energy} solutions -- on a general smooth bounded domain $\Om$ is open. The only existence results in the literature when $n =3$ and $\lambda \ge \Lambda_1$ are for bifurcating solutions in a left neighbourhood of $\Sp(-\Delta)$ of fixed width \cite{CeramiFortunatoStruwe}, for sign-changing solutions when $\Om$ is a rectangle and $\lambda$ is large enough \cite{FortunatoJannelli, SunWeiYang2} and for large energy sign-changing solutions of \eqref{eq:critlambda} in the unit ball (see the very recent \cite{SunWeiYang} concerning Br\'ezis' conjecture). In this paper, among other things, we address the existence problem for \eqref{eq:critlambda} on a general $\Om$ when $n=3$ according to the value of $\lambda \ge \Lambda_1$. 

\medskip

Our main focus is in the three-dimensional case. We introduce some notations: if $n=3$ and  $\lambda \ge 0, \lambda \not \in \Sp(-\Delta)$, $-\Delta - \lambda$ is invertible and we let $G_\lambda$ be its Green's function with Dirichlet boundary conditions. At a fixed $x \in \Om$, $G_\lambda(x,\cdot)$ expands as 
\begin{equation*} 
	G_\lambda(x,y)=\frac{1}{4 \pi |x-y|}+ m_\lambda(x) + O(|x-y|) \quad \text{ as } y \to x.
\end{equation*}
We call $x \in \Om \mapsto m_\lambda(x)$ the \emph{mass function} of $G_\lambda(x,\cdot)$ at $x$. In the coercive case $ \lambda \in (0, \Lambda_1)$ it is well-known (see \cite{BN, Druetdim3}) that (positive) least energy solutions of \eqref{eq:critlambda} of energy less than $K_3^{-3}$ exist if and only if $\lambda \in (\lambda_{0,*}, \Lambda_1)$ for some $\lambda_{0,*} >0$. This threshold value $\lambda_{0,*}$ is determined by the following condition:
\begin{equation} \label{deflambdastar}
\max_{\Om} m_{\lambda} >0 \quad \iff \quad  \lambda \in (\lambda_{0,*}, \Lambda_1).
\end{equation}
By analogy with the coercive case it is natural to wonder, when $\lambda \ge \Lambda_1$, whether the mass function $m_\lambda$ still plays a role in the existence of solutions of \eqref{eq:critlambda} satisfying $\mathcal{E}_\lambda(\Om) < K_3^{-3}$ -- in particular through its sign. When $\lambda \ge \Lambda_1, \lambda \not \in \Sp(-\Delta)$ the mass function $m_\lambda$, although well-defined, has suprisingly never been studied and this question seems to never have been addressed before. We investigate it in detail in this paper. For $i \ge 0$ we define 
 \begin{equation} \label{deflambdaistar}
 \lambda_{i,*} =  \inf \big \{ \lambda \in (\Lambda_i, \Lambda_{i+1}), \, \, \max_{\Om} m_\lambda >0  \big \} 
 \end{equation}
and
\begin{equation} \label{def:lambdai:min}
\bar{\lambda}_i = \inf \big \{ \lambda \in (\Lambda_i, \Lambda_{i+1}),\, \,   -\infty < \mathcal{E}_\lambda(\Om) < K_3^{-2} \big \},
\end{equation}
with the convention that $\bar{\lambda}_i = - \infty$ if no solutions of \eqref{eq:critlambda} exist when $\lambda \in (\Lambda_i, \Lambda_{i+1})$. For $i \ge 0$ we prove in Appendix \ref{annexe:masse} below that $\lambda \mapsto \max_{\Om} m_\lambda$ is increasing in $(\Lambda_i, \Lambda_{i+1})$ and that $\lim_{\lambda \underset{<}{\to}  \Lambda_{i+1}} \max_{\Om} m_\lambda  = + \infty$: thus $ \lambda_{i,*} < \Lambda_{i+1}$ and $\max_{\Om} m_\lambda >0$ when $\lambda \in (\lambda_{i,*}, \Lambda_{i+1})$. When $i=0$, ie when $\lambda \in (0, \Lambda_1)$, it was shown in \cite{BN} that $\bar{\lambda}_0 >0$, and it was later shown in \cite{Druetdim3} that $\bar{\lambda}_0 = \lambda_{0,*}$. In other words, in the coercive case, positivity of the mass function somewhere is equivalent to the existence of least-energy solutions of \eqref{eq:critlambda} whose energy is less than $K_3^{-3}$. In the non-coercive cas $\lambda \ge \Lambda_1$ our main result shows that positivity of the mass still implies existence of least-energy solutions of low energy: 

\begin{theo} \label{theoreme:n:3}
Let $\Omega$ be a smooth bounded domain of $\R^3$ and let $i \ge 1$. Let $\lambda_{i,*}, \bar{\lambda}_i$ be defined as in \eqref{deflambdaistar}, \eqref{def:lambdai:min}.
\begin{enumerate}
\item We have 
$$\Lambda_i \le \bar{\lambda}_i \le \lambda_{i,*} < \Lambda_{i+1}.$$ 
\item If $\bar{\lambda}_i = \Lambda_i$ then, for any $\lambda \in (\Lambda_i, \Lambda_{i+1})$, $\mathcal{E}_\lambda(\Om)$ is attained at a least-energy solution of \eqref{eq:critlambda} whose energy is strictly less than $K_3^{-3}$. 
\item If $\bar{\lambda}_i > \Lambda_i$ we have $\bar{\lambda}_i = \lambda_{i,*}$, and 
$$- \infty < \mathcal{E}_{\lambda}(\Om) < K_3^{-3} \quad \iff \quad \lambda \in (\lambda_{*,i}, \Lambda_{i+1}). $$
In addition, for any $\lambda \in (\lambda_{i,*}, \Lambda_{i+1})$, $\mathcal{E}_\lambda(\Om)$ is attained at a least-energy solution of \eqref{eq:critlambda} whose energy is strictly less than $K_3^{-3}$. Furthermore, the function $\lambda \mapsto \mathcal{E}_\lambda(\Om)$ is Lipschitz continuous and decreasing in $(\lambda_{i,*}, \Lambda_{i+1})$ and satisfies $ \lim \limits_{\lambda \underset{> }{\to} \lambda_{i,*}} \mathcal{E}_\lambda(\Om)= K_3^{-3}$ and $\lim \limits_{\lambda \underset{< }{\to} \Lambda_{i+1}}  \mathcal{E}_\lambda(\Om) = 0$.
\item If $\bar{\lambda}_i > \Lambda_i$ and if there is $\lambda \in [\Lambda_i, \lambda_{i,*}]$ for which \eqref{eq:critlambda} has a non-trivial solution, then $\mathcal{E}_{\lambda}(\Om) \ge K_3^{-3}$ and the following alternative holds: either $\mathcal{E}_{\lambda}(\Om) > K_3^{-3}$ or $\mathcal{E}_{\lambda}(\Om) = K_3^{-3}$ is not attained and the set of solutions of \eqref{eq:critlambda} is not compact in $H^1_0(\Om)$. 
\end{enumerate} 
\end{theo}

Theorem \ref{theoreme:n:3} provides in particular the first \emph{existence} result for \eqref{eq:critlambda} on any smooth bounded domain $\Om$ when $n = 3$ and $\lambda \ge \Lambda_1$. We state it as a Corollary: 

\begin{corol} \label{corol:n:3}
Let $\Omega$ be a smooth bounded domain of $\R^3$ and, for any $i \ge 0$, let $\lambda_{i,*}$ be as in \eqref{deflambdaistar}. If $ \lambda \in \bigcup_{i \ge 0} (\lambda_{i,*}, \Lambda_{i+1})$ equation \eqref{eq:critlambda} admits a non-trivial solution $u \in H^1_0(\Om)$ satisfying $\int_{\Om} u^6 dx < K_3^{-3}$. 
\end{corol}

Theorem \ref{theoreme:n:3} shows that the mass function $m_\lambda$ still plays a crucial role in the existence theory of \eqref{eq:critlambda} when $n=3$ even in the non-coercive case $i \ge 1$. More precisely, it shows that for $\lambda \ge \Lambda_1$ the picture for least-energy sign-changing\footnote{Recall that when $i \ge 1$ and $\lambda \in [\Lambda_i, \Lambda_{i+1})$  every solution of \eqref{eq:critlambda} changes sign} solutions of \eqref{eq:critlambda} \emph{is at worst as good as in the coercive case.} Indeed, either $\bar{\lambda}_i = \Lambda_i$ and a non-zero least-energy solution of \eqref{eq:critlambda} exists for all $\lambda \in (\Lambda_i, \Lambda_{i+1})$, or $\bar{\lambda}_i > \Lambda_i$ and a non-zero least-energy solution of \eqref{eq:critlambda} exists exactly for  $\lambda \in (\lambda_{i,*}, \Lambda_{i+1})$,  that is when  $\max_\Om m_\lambda >0$, thus following the same behavior than in the coercive case. It is still unclear to us whether the situation $\bar{\lambda}_{i} = \Lambda_i$ may actually occur for some $i \ge 1$ and some smooth bounded domains $\Om$ and whether, in this case, the strict inequality $\bar{\lambda}_i = \Lambda_i <  \lambda_{i,*}$ may also occur. The quantities $\lambda_{i,*}$ and $\bar{\lambda}_i$ are quite difficult to estimate when $i \ge 1$ and are related to the geometry of nodal sets of eigenfunctions of $-\Delta$. We refer to the discussion following the proof of Theorem \ref{theoreme:n:3} in Section \ref{sec:proof:main:results} for more details on this question.  

\medskip

Our analysis also applies to the higher dimensional case. When $n \ge 4$ we obtain a precise description of the energy function, which states as follows:

\begin{theo} \label{theoreme:n:4plus}
Let $\Omega$ be a smooth bounded domain of $\R^n$, $n \ge 4$. Then
\begin{enumerate}
\item $\lambda \mapsto \mathcal{E}_{\lambda}(\Om)$ is positive, Lipschitz continuous and decreasing in every interval $(\Lambda_i, \Lambda_{i+1})$ for $i \ge 0$. It is attained for every $\lambda \in (\Lambda_i, \Lambda_{i+1})$ and we have  $\lim \limits_{\lambda \underset{< }{\to} \Lambda_{i+1}}  \mathcal{E}_\lambda(\Om) = 0$.
\item If $n \ge 5$ and $i \ge 1$, $\mathcal{E}_{\Lambda_i}(\Om)$ is attained, satisfies $\mathcal{E}_{\Lambda_i}(\Om) < K_n^{-n}$ and $\lambda \in [\Lambda_i, \Lambda_{i+1}) \mapsto \mathcal{E}_{\lambda}(\Om)$ is Lipschitz continuous in $[\Lambda_i, \Lambda_{i+1}]$.
\item If $n =4$ we have $ \lim \limits_{\lambda \underset{> }{\to} \Lambda_{i}} \mathcal{E}_\lambda(\Om) \in (0, K_4^{-4}]$.
\end{enumerate}
\end{theo}
The case $i=0$ in Theorem \ref{theoreme:n:4plus} was proven in \cite{BN}. A striking consequence of Theorem \ref{theoreme:n:4plus} is that, when $n \ge 4$, the energy function $\lambda \in (0, + \infty) \mapsto \mathcal{E}_{\lambda}(\Om)$ is discontinuous \emph{exactly} at $\Sp(-\Delta)$. As $\lambda \to \Sp(-\Delta)$ from the left, we prove that least-energy solutions that attain $\mathcal{E}_\lambda(\Om)$ bifurcate from the spectrum, see Proposition \ref{prop:bifurcation} below, and thus behave like the bifurcating families of \cite{CeramiFortunatoStruwe}. When $n=4$, and as $\lambda \to \Sp(-\Delta)$ from the right, we believe that $\lim \limits_{\lambda \underset{> }{\to} \Lambda_{i}} \mathcal{E}_\lambda(\Om) = K_4^{-4}$ and that least-energy solutions for \eqref{eq:critlambda} blow-up.  A schematic graph of $\lambda \mapsto \mathcal{E}_{\lambda}(\Om)$ when $n \ge 5$ is drawn in Figure \ref{figuretheoreme}.

\medskip

\begin{figure} \label{figuretheoreme}
\centering
\includegraphics[scale = 0.4]{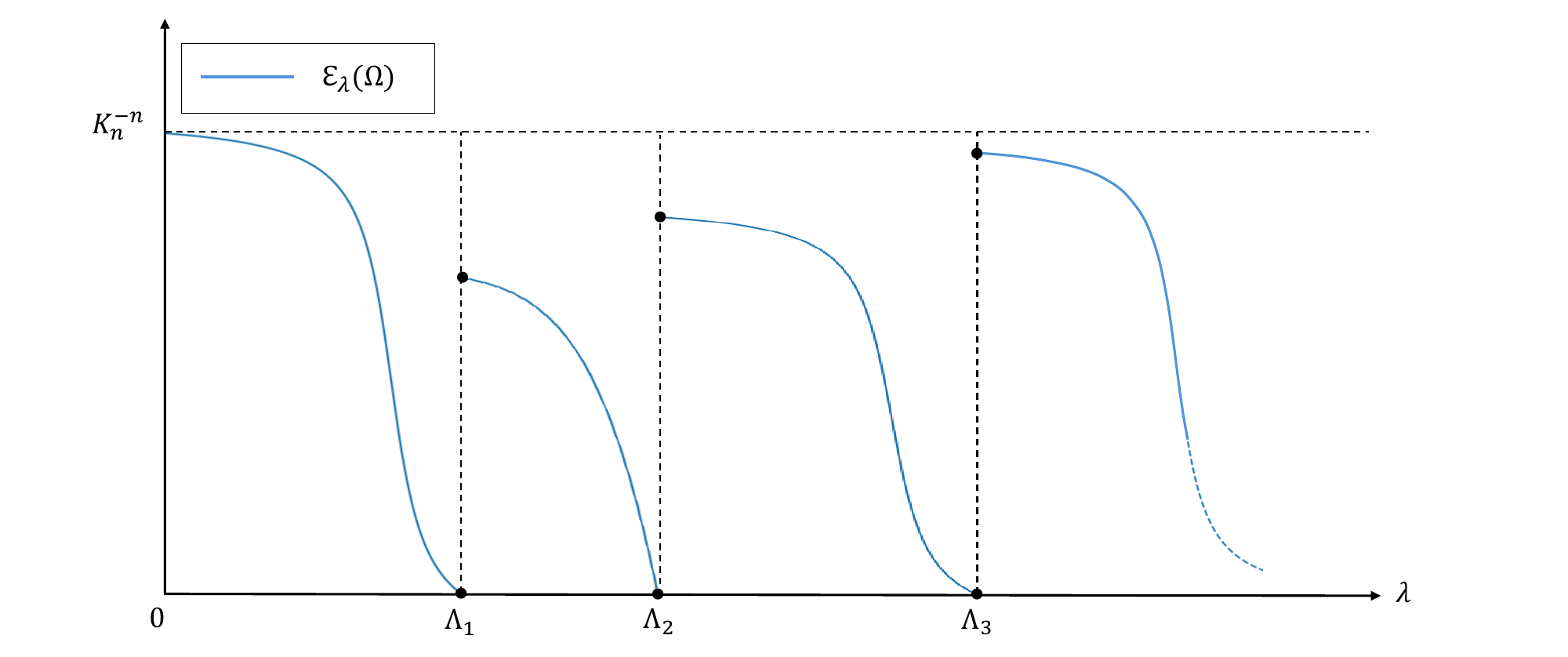}
\caption{Schematic representation of the graph of $\lambda \mapsto \mathcal{E}_{\lambda}(\Om)$ when $n \ge 5$.}
\end{figure}

\subsection{Strategy of proof of Theorems \ref{theoreme:n:3} and \ref{theoreme:n:4plus}.} \label{strategie_preuve}

We prove Theorems \ref{theoreme:n:3} and \ref{theoreme:n:4plus} by introducing a new variational point of view for \eqref{eq:critlambda}. More precisely, we introduce a family of new variational problems, inspired from conformal geometric and spectral-theoretic considerations, whose minimisers provide least-energy solutions of \eqref{eq:critlambda}. We describe these problems and explain the heuristic behind them. Let $\lambda \ge 0 $ and  $u \in L^{2^*}(\Om)$, $u >0$ a.e. in $\Om$. For $ p \ge 1$ we define the so-called $p$-th \emph{generalised eigenvalue of $-\Delta - \lambda$ with weight $u$} as
\begin{equation*}
\mpl(u) = \inf_{\text{dim} V = p} \max_{v \in V \backslash \{0\}} \frac{\int_{\Om}\left(  |\nabla v|^2 - \lambda v^2 \right) \, dx}{\int_{\Om} u^{2^*-2} v^2 \, dx},
\end{equation*} 
where the infimum is taken over all $p$-dimensional subsets $V$ of $H^1_0(\Om)$. Formally, $\mkl(u)$ is the $k$-th eigenvalue of $-\Delta-\lambda$ on the weighted Lebesgue space $L^2(\Om, u^{2^*-2}dx)$. Since $u>0$ a.e., $\mkl(u)$ is well-defined for every $p \ge 1$ and it possesses finitely many independent non-zero generalised eigenvectors: they are the only functions $\vp \in H^1_0(\Om)$, normalised by $\int_{\Om} u^{2^*-2} \vp^2 dx = 1$, that satisfy
\begin{equation} \label{eq:vec:propre:intro}
 - \Delta \vp - \lambda \vp = \mpl(u) u^{2^*-2} \vp \quad \text{ in } \Om. 
 \end{equation}
This is proven in Proposition \ref{prop:vp} below, where we show that $(\mpl(u))_{p \ge0}$ is non-decreasing.  The connection with sign-changing solutions of \eqref{eq:critlambda} arises through the weighted eigenvalue equation \eqref{eq:vec:propre:intro}. Let indeed $v \in H^1_0(\Om) \backslash \{0\}$ be a solution of \eqref{eq:critlambda}: simple arguments show that $\tilde v = \frac{v}{\Vert v \Vert_{2^*}}$ satisfies $\int_\Om |\tilde v|^{2^*}dx = 1$ and 
\begin{equation} \label{eq:vec:propre:intro:2}
 - \Delta \tilde v - \lambda \tilde v = \mu |\tilde v|^{2^*-2} \tilde v \quad \text{ in } \Om, 
 \end{equation}
where $\mu = \Vert v \Vert_{2^*}^{2^*-2}$. Using the uniqueness part in \eqref{eq:vec:propre:intro} and \eqref{eq:vec:propre:intro:2} we prove, in Proposition \ref{prop:structure:vp} below, that \emph{any} non-zero solution $v$ of \eqref{eq:critlambda} is a generalised eigenvector for some weight whose associated generalised eigenvalue is $\Vert v \Vert_{2^*}^{2^*-2}$. In view of this correspondence the strategy of proof that we adopt in this paper is as follows: in order to find solutions of \eqref{eq:critlambda} of least energy we look for \emph{functions $u \in L^{2^*}(\Om), u >0$, which minimise positive generalised eigenvalues over all weights}. This approach is made possible by a crucial property of the generalised eigenvalues, that we prove in Proposition \ref{prop:signvp} below, which is that their sign only depends on the position of $\lambda$ with respect to $\Sp(-\Delta)$. For any fixed $\lambda \ge 0$ and $u \in L^{2^*}(\Om), u >0$ a.e, the index of the smallest positive generalised eigenvalue is therefore independent of $u$ and provides us with a functional that we may minimise over all weight $u$. 

\medskip

Let us be more precise. Let $i \ge 0$ and $\lambda \in [\Lambda_i, \Lambda_{i+1})$. We will denote by $N(i)+1$ the index of the smallest eigenvalue of $-\Delta$, counted with multiplicity, which is strictly bigger than $\Lambda_i$. Equivalently, $N(i)$ is defined by the property that $\lambda_p \le \Lambda_i \iff p \le N(i)$. We let in the following: 
$$k = N(i)+1.$$
We prove in Proposition \ref{prop:signvp} below that, for any $u \in L^{2^*}(\Om)$ with $u >0$ a.e, $\mkl(u) >0$ while $\mpl(u) \le 0$ for $p \le k-1$. The integer $k$ is therefore the index of the smallest positive generalised eigenvalue regardless of $u$, and we will call it the principal generalised eigenvalue of $-\Delta - \lambda$. In view of \eqref{eq:vec:propre:intro} and \eqref{eq:vec:propre:intro:2} we introduce the following minimisation problem: 
\begin{equation*} 
\mkl(\Om) = \inf_{\underset{\Vert u \Vert_{2^*}=1}{u \in L^{2^*}(\Om), u>0 \text{ a.e.},} } \mkl(u), \quad \text{ where } k = N(i)+1.
\end{equation*}
For any $i \ge 0$ fixed this problem is well-defined for $\lambda \in [\Lambda_i, \Lambda_{i+1})$. The fundamental property of $\mkl(\Om)$ that we prove in this paper is as follows: for any $i \ge0$, $\lambda \in [\Lambda_i, \Lambda_{i+1})$, and if $k = N(i)+1$,
\begin{equation} \label{prop:minmuk:intro}
\mkl(\Om) \text{ is attained } \quad \iff \quad \mkl(\Om) < K_n^{-2},
\end{equation}
and if \eqref{prop:minmuk:intro} is satisfied minimisers of $\mkl(\Om)$ define non-zero least-energy solutions of \eqref{eq:critlambda}, that change sign if $i \ge 1$ (or, equivalently, $\lambda \ge \Lambda_1$). We then prove that the assumptions of Theorems \ref{theoreme:n:3} and \ref{theoreme:n:4plus} ensure that the strict inequality in \eqref{prop:minmuk:intro} is satisfied and that this implies that $\mathcal{E}_\lambda(\Om) = \mkl(\Om)^{\frac{n}{2}}$. The properties of $\lambda \mapsto \mathcal{E}_\lambda(\Om)$ then follow from investigating the function $\lambda \mapsto \mkl(\Om)$. 

\medskip

When $\lambda \in [\Lambda_i, \Lambda_{i+1})$ for some $i \ge 1$, least-energy solutions of \eqref{eq:critlambda} were first investigated in \cite{SzulkinWethWillem} (see also \cite{ChenShiojiZou}) and obtained as minimisers of a simple variational problem over $E_{\Lambda_i}(\Om)^{\perp}$. The problem $\mkl(\Om)$ that we propose in this paper, despite being more involved and requiring more preparation than the approach of \cite{SzulkinWethWillem}, has several important advantages. On the one side, the minimisation over all weights $u$ provides us with an additional degree of freedom  in our test-function computations and allows us to prove that $\mkl(\Om) < K_3^{-2}$, which yields the new existence results of Corollary \ref{corol:n:3} when $n=3$. On the other side, by proving the relation $\mathcal{E}_\lambda(\Om) = \mkl(\Om)^{\frac{n}{2}}$ in an appropriate range of $\lambda$ we reduce the analysis of $\lambda \mapsto \mathcal{E}_\lambda(\Om)$ to that of $\lambda \mapsto \mkl(\Om)$, which is an explicit function of $\lambda$, and obtain Theorem \ref{theoreme:n:4plus}. Most of the analytical difficulties that we face in this paper arise from the non-coercive assumption $\lambda \ge \Lambda_1$ and are related to the min-max definition of generalised eigenvalues and to their lack of differentiability in $u$. 

\medskip

Property \eqref{prop:minmuk:intro} is reminiscent of the celebrated minimisation approach of \cite{BN}. This is not surprising:  when $\lambda \in [0, \Lambda_1)$ we prove (Appendix \ref{app:m1} below) that
$$ \mu_1^{\lambda}(\Om) = \inf_{v \in H^1_0(\Om)\backslash \{0\}} \frac{\int_{\Om} \big( |\nabla v|^2 - \lambda v^2 \big)dx}{\big(\int_{\Om}|v|^{2^*}dx \big)^{\frac{n-2}{n}}},$$
so that $\mu_1^{\lambda}(\Om)$ is precisely the minimisation problem of \cite{BN}. When $\lambda \in [\Lambda_i, \Lambda_{i+1})$ for some $i \ge 1$, $\mkl(\Om)$ with $k = N(i)+1$ should thus be understood as the natural generalisation, in the higher-frequency setting $\lambda \in [\Lambda_i, \Lambda_{i+1})$, of the minimisation problem of \cite{BN}. The motivation for defining $\mkl(\Om)$ comes from conformal geometry: $\mkl(\Om)$ is the euclidean transposition of eigenvalue-optimisation problems in conformal geometry in dimension $3$ and higher that have attracted renewed attention in very recent years (see for instance \cite{AmmannHumbert, AmmannJammes, GurskyPerez, HumbertPetridesPremoselli, PetridesTewodrose, PremoselliVetois4}). The problem $\mkl(\Om)$ that we introduce in this paper therefore follows the original philosophy of \cite{BN}, which transposed in the euclidean setting ideas of Aubin \cite{AubinYamabe} for the Yamabe problem. The new variational framework that we develop in this paper is robust and may be applied to other nonlinear critical problems, such as polyharmonic equations, to produce least-energy solutions (we refer for this to \cite{HumbertPetridesPremoselli} where polyharmonic problems on manifolds are discussed). An adaptation of the arguments in this paper also gives an analogue of Theorems \ref{theoreme:n:3} and \ref{theoreme:n:4plus} for radial solutions of \eqref{eq:critlambda} on the unit ball, but we will not pursue this here.

\subsection{Structure of the paper}

The organisation of the paper is as follows. In Section \ref{eigenvalues} we introduce the definition of generalised eigenvalues, investigate their properties and prove the existence of generalised eigenfunctions. We prove in particular Proposition \ref{prop:structure:vp}, where we establish the connection between non-zero solutions of \eqref{eq:critlambda} and generalised eigenvectors. In Section \ref{vp:derivees} we prove that the sign of generalised eigenvalues only depends on the position of $\lambda$ with respect to $\Sp(-\Delta)$ (Proposition \ref{prop:signvp}) and we develop a variational theory for the generalised eigenvalues functionals $u \mapsto \mkl(u)$ (Proposition \ref{prop:dervp}) to overcome their lack of differentiability. Section \ref{vp:minimisation} contains the proof of the right implication in \eqref{prop:minmuk:intro} (Theorem \ref{prop:minmkl1}). We introduce a family of relaxed problems $(\mu_{k,\ve})_{\ve >0}$ that approximate $\mkl(\Om)$ as $\ve \to 0$ and that penalise the vanishing locus of $u$ (see \eqref{Fkeps} below). We prove that each relaxed problem admits a minimiser $u_\ve$ and that the family $(u_\ve)_{\ve >0}$ defines a minimising sequence for $\mkl(\Om)$ that converges in a strong enough sense to a minimiser of $\mkl(\Om)$. In the course of the proof we rely on the variational theory of Section \ref{vp:derivees}  and we prove in particular that if $u \in L^{2^*}(\Om), u>0$ a.e. attains $\mkl(\Om)$, $\mkl(u)$ is simple and any normalised generalised eigenvector $\vp$  satisfies 
 $$ u = |\vp| \quad \text{ in } \Om. $$
 The latter is the Euler-Lagrange equation associated to $\mkl(\Om)$ and together with \eqref{eq:vec:propre:intro} it implies that $\vp$ is a solution of \eqref{eq:critlambda} of energy $\int_\Om |\vp|^{2^*}\, dx = \mkl(\Om)^{\frac{n}{2}}$. That this $\vp$ is a least-energy solution then follows from \eqref{eq:vec:propre:intro:2}. In Section \ref{fonctionstest} we prove (Theorem \ref{prop:test:func}) by test-function computations that the strict inequality in \eqref{prop:minmuk:intro} is satisfied under the assumptions of Theorems \ref{theoreme:n:3} and \ref{theoreme:n:4plus}. We carefully estimate, for a well-chosen family of weights and of  $k$-dimensional test subspaces, their contribution to $\mkl(\Om)$. In the non-coercive case, the min-max nature of $\mkl(u)$ forces us to work with a higher degree of precision. In Section \ref{ordre:deux} we develop a second-order variational theory for $u \mapsto \mkl(u)$: we first expand $\mkl$ to second order along interior variations and use the latter to prove a stability inequality of sorts at minimisers (Proposition \ref{prop:DL:lambda:ordre:2}), from which we deduce the left implication in \eqref{prop:minmuk:intro}.
This is the first time, to our knowledge, that a second variation for generalised eigenvalue functionals is computed. Section \ref{sec:proof:main:results} contains the proof of Theorems \ref{theoreme:n:3} and \ref{theoreme:n:4plus}. Appendix \ref{app:m1} contains the proof that, in the case $i=0$, $\mkl(\Om)$ is exactly the celebrated minimisation problem of \cite{BN}. Appendix \ref{app:m2} proves a compactness result for sign-changing solutions of \eqref{eq:critlambda} that is used in the proof of Theorem \ref{theoreme:n:3}. And finally, Appendix \ref{annexe:masse} contains a detailed study of the mass function $x \in \Om \mapsto m_\lambda(x)$ in the non-coercive case $\lambda \ge \Lambda_1, \lambda \not \in \Sp(-\Delta)$.

\medskip

\textbf{Acknowledgements:} the second author wishes to thank J. Mederski, P.-D. Thizy and T. Weth for several useful discussions during the preparation of this work.

\section{Generalised eigenvalues and their properties} \label{eigenvalues}

We introduce in this section the notion of generalized eigenvalues that we investigate in this paper. Let $\Om \subseteq \R^n, n\ge3$ be any connected open set. In the rest of the paper if $p \in [1, + \infty]$ and $u \in L^p(\Om)$ we will denote its $L^p$ norm in $\Om$ by $\Vert u \Vert_p$. For $u \in L^{2^*}(\Om)\backslash \{0\}$,  $u\ge0$ a.e., where $2^* = \frac{2n}{n-2}$, we let 
$$K_u = \Big \{ v: \Om \to \R \text{ measurable such that }  u^{\frac{2}{n-2}} v = 0 \text{ a.e. in } \Om \Big \} $$
and
$$ \begin{aligned}
L^2(\Om, u^{2^*-2}dx) 
  &= \Big \{ v: \Om \to \R \text{ measurable such that }\int_{\Om} u^{2^*-2} v^2 dx < + \infty \Big \}. 
\end{aligned} $$
Recall that $2^*-2 = \frac{4}{n-2}$. We define $L^2_u(\Om)$ as
$$ \begin{aligned}
L^2_u(\Om) = L^2(\Om, u^{2^*-2}dx) / K_u,
 \end{aligned}  $$
and we endow it with the quotient norm. $L^2_u(\Om)$ endowed with the scalar product $(v,w)_{L^2_u} = \int_{\Om} u^{2^*-2} vw dx$ and the associated norm $ \Vert v \Vert_{L^2_u}^2 =  \int_{\Om} u^{2^*-2} v^2 dx$ is thus a Hilbert space. Denote by $D^{1,2}(\Om)$ the closure of $C^\infty_c(\Om)$ for the norm $\Vert \nabla \cdot \Vert_{2}$. If $\Om$ is bounded, $D^{1,2}(\Om) = H^1_0(\Om)$. H\"older and Sobolev inequalities imply that $D^{1,2}(\Om)\subset L^2_u(\Om)$ and $L^{2^*}(\Om) \subset L^2_u(\Om)$. 
\begin{defi} \label{defiL2u}
Let $u \in L^{2^*}(\Om)\backslash \{0\}$, $u \ge0$ a.e., $k\in \mathbb{N} \backslash \{0\}$, and let $V \subseteq D^{1,2}(\Om)$ be a linear subspace. We say that 
$$ \dim_u V = k $$
if there exist $v_1, \dots, v_k \in D^{1,2}(\Om)$ such that $V = \text{Span}(v_1, \dots, v_k)$ and such that the family $(v_1, \dots, v_k)$ is free in $L^2_u(\Om)$. 
\end{defi}
Equivalently, $\dim_u V = k$ if and only if there exist $v_1, \dots, v_k \in D^{1,2}(\Om)$ such that $V = \text{Span}(v_1, \dots, v_k)$ and $(u^{\frac{2}{n-2}}v_1, \dots, u^{\frac{2}{n-2}}v_k)$ is free in $L^2(\Om)$, and in this case $\dim_u V$ is also the real dimension of $V$ as a vector subspace of $L^2_u(\Om)$, which we denote by $\dim V$. We also define
\begin{equation} \label{defLu}
 \begin{aligned}
L^{2^*}_>(\Om) & = \big \{ u\in L^{2^*}(\Om) \backslash \{0\}, u > 0 \text{ a.e. in } \Omega \big \}.  \\
\end{aligned} 
\end{equation}
For $V \subset D^{1,2}(\Om)$ and $u \in L^{2^*}(\Omega) \backslash \{0\}$, $u \ge 0$ a.e.,  we have 
$$ \dim_u V \le \dim V $$
and it is easily seen that equality holds when $u \in L^{2^*}_{>}(\Om)$.

\medskip

From now on and until the end of this paper we will assume that $\Om$ is a smooth bounded domain of $\R^n$, $n \ge 3$. We recall that for $u \in H^1_0(\Om)$ we let $\Vert u \Vert_{H^1_0} = \Vert \nabla u \Vert_{2}$, that $0< \lambda_1 < \lambda_2 \le \cdots \le \lambda_k \le \cdots$ are the eigenvalues, counted with multiplicity,  of $-\Delta$ in $H^1_0(\Om)$ and that  $\Lambda_k$, for $k \ge 1$, are the eigenvalues counted without multiplicity. Recall the convention $\Lambda_0 = 0$, so that $ 0  =  \Lambda_0  < \Lambda_1  <  \cdots$. For $i \ge 1$ we denote by $E_{\Lambda_i}(\Om) \subset H^1_0(\Om)$ the eigenspace associated to $\Lambda_i$. We let $N$ be the multiplicity map: it is the increasing map $N: \mathbb{N} \to \mathbb{N}$ such that $N(0) = 0, N(1) = 1$ and, for any $k \ge 0$, 
\begin{equation} \label{defN} 
\lambda_{N(k) +1} = \cdots = \lambda_{N(k+1)} = \Lambda_{k+1},
\end{equation}
so that the number of eigenvalues equal to $\Lambda_{k+1}$ is $N(k+1) - N(k) \ge 1$. Let $\lambda \ge0$, $k \ge 1$ be an integer and $u \in L^{2^*}(\Om)\backslash \{0\}$, $u \ge 0$ a.e. We define:
\begin{equation} \label{muk}
\mkl(u) = \inf_{\text{dim}_u V = k} \max_{v \in V \backslash \{0\}} \frac{\int_{\Om}\left(  |\nabla v|^2 - \lambda v^2 \right) \, dx}{\int_{\Om} u^{2^*-2} v^2 \, dx},
\end{equation}
where the infimum is taken over all linear subspaces $V \subset H^1_0(\Om)$ satisfying $\dim_u V = k$ in the sense of Definition \ref{defiL2u}. We allow this infimum to be equal to $- \infty$. We also let, for  $v \in H^1_0(\Om)$, 
\begin{equation} \label{defQu}
Q_u^{\lambda}(v) = \frac{\int_{\Om}\left(  |\nabla v|^2 - \lambda v^2 \right) \, dx}{\int_{\Om} u^{2^*-2} v^2 \, dx}.
\end{equation}
This Section and the next one are devoted to investigating the properties of $\mkl(u)$. When $u \equiv 1$, \eqref{muk} is the classical min-max definition of the eigenvalues of $- \Delta- \lambda$ in $H^1_0(\Om)$. When $u \not \equiv 1$, $\mkl(u)$ generalises this definition over the measure space $L^{2}_u$. We call it a \emph{generalised eigenvalue of weight $u$}. In this section we prove that most of the well-known properties of eigenvalues of $-\Delta$ remain true for generalised eigenvalues, up to some small adaptations. We first prove that $\mu_1^{\lambda}(u)$ is always finite provided $u \in L^{2^*}_{>}(\Om)$, in which case $\mu_1^\lambda(u)$ is also simple:

\begin{prop} \label{prop:vp1}
Let $\Om \subset \R^n$, $n \ge 3$, be bounded and smooth. Let $\lambda \ge0$ and $u \in L^{2^*}_>(\Om)$, where $L^{2^*}_>(\Om)$ is as in \eqref{defLu}. Then:
\begin{enumerate}
\item There exists $\varphi_{1}^{\lambda}(u) \in H^1_0(\Om)\backslash \{0\}$ that attains $\mu_1^{\lambda}(u)$ and satisfies $\Vert \varphi_{1}^{\lambda}(u)\Vert_{L^2_u} = 1$. It satisfies, in a weak sense,
\begin{equation} \label{vp1}
\big( - \Delta - \lambda \big) \vp_{1}^{\lambda}(u) = \mu_1^{\lambda}(u) u^{2^*-2}    \vp_{1}^{\lambda}(u) \quad \text{ in } \Om. 
\end{equation}
\item Denote by $E_1^{\lambda}(u) \subseteq H^1_0(\Om)$ the vector space of solutions of \eqref{vp1}. Then $\dim  E_1^{\lambda}(u) = \dim_u  E_1^{\lambda}(u) = 1$ and if $\vp \in E_1^{\lambda}(u)\backslash \{0\}$ we either have $\vp > 0$ a.e. or $\vp < 0$ a.e in $\Om$. 
\end{enumerate}
\end{prop}
We call  $\vp_1^{\lambda}(u)$ the first generalised eigenfunction and $E_1^{\lambda}(u)$ the associated generalised eigenspace. 

\begin{proof}
A first important observation is that, by \eqref{muk}, $\mu_1^{\lambda}(u)$ rewrites as 
\ben \label{mu1eq}
 \mu_{1}^{\lambda}(u) = \inf_{v \in H^1_0(\Om)  \backslash \{0\}} \frac{\int_{\Om}\left(  |\nabla v|^2 - \lambda v^2 \right) \, dx}{\int_{\Om} u^{2^*-2} v^2 \, dx}. 
 \een
We first prove that when $u \in L^{2^*}_>(\Om)$, $\mu_1^{\lambda}(u) > - \infty$. We follow for this the arguments in \cite[Proposition 2.2]{ElSayed}. Assume by contradiction that $\mu_1^{\lambda}(u) = - \infty$ and let $(v_i)_{i \ge0}$ be a minimising sequence for $\mu_1^{\lambda}(u)$ normalised such that $\int_\Om u^{2^*-2} v_i^2 dx = 1$. By \eqref{mu1eq} we then have $\int_{\Om}\big( |\nabla v_i|^2 - \lambda v_i^2 \big) dx \to - \infty$ and thus $\Vert v_i \Vert_{H^1_0} \to + \infty$ as $i \to + \infty$. We let, for $i \ge 1$, $\tilde{v}_i = \frac{v_i}{\Vert v_i\Vert_{H^1_0}}$. Up to passing to a subsequence there exists $\tilde{v}_0 \in H^1_0(\Om)$ such that $\tilde{v}_i \rightharpoonup \tilde{v}_0$ in $H^1_0(\Om)$ and $\tilde{v}_i \to \tilde{v}_0$ in $L^2(\Om)$. By the normalisation of $v_i$ we have
$$ \int_{\Om} u^{2^*-2} \tilde{v}_i^2dx = \frac{1}{\Vert v_i \Vert_{H^1_0}^2} \to 0 $$
as $ i \to + \infty$, so that Fatou's lemma yields $\int_\Om u^{2^*-2} \tilde{v}_0^2dx \le 0$. Since $u>0$ a.e. in $\Omega$, this shows that $\tilde{v}_0 = 0$ p.p. in $\Om$. But for $i$ large enough we have $\int_{\Om}\big( |\nabla v_i|^2 - \lambda v_i^2 \big)dx  \le 0$, so dividing by $\Vert v_i \Vert_{H^1_0}^2$ and using that $\tilde{v}_0 = 0$ yields $\Vert \tilde{v}_i \Vert_{H^1_0}^2 \le o(1)$ as $i \to + \infty$, a contradiction with $\Vert \tilde{v}_i \Vert_{H^1_0} = 1$. Hence $\mu_1^{\lambda}(u) > - \infty$.

We now prove that $ \mu_{1}^{\lambda}(u)$ is attained. Let $(v_i)_{i\ge1}$ be a minimising sequence normalised such that $\int_\Om u^{2^*-2} v_i^2 dx = 1$. Without loss of generality we can assume that $v_i \ge0$ a.e. in $\Om$. We first claim that $(v_i)_{i\ge1}$ is bounded in $H^1_0(\Om)$. Assume by contradiction that $\Vert v_i \Vert_{H^1_0} \to + \infty$ as $i \to + \infty$ and let $\tilde{v}_i = \frac{v_i}{\Vert v_i\Vert_{H^1_0}}$. Then, up to a subsequence, $\tilde{v}_i \rightharpoonup \tilde{v}_0$ in $H^1_0(\Om)$ and $\tilde{v}_i \to \tilde{v}_0$ in $L^2(\Om)$. As before, 
$$ \int_{\Om} u^{2^*-2} \tilde{v}_i^2dx = \frac{1}{\Vert v_i \Vert_{H^1_0}^2} \to 0 $$
as $i \to + \infty$ so that, as before, $\tilde{v}_0 = 0$ a.e. in $\Om$. For $i$ large enough we then have 
$$ \int_{\Om} |\nabla \tilde{v}_i|^2 dx = \lambda \int_{\Om} \tilde{v}_i^2dx +  \frac{\mu_1^{\lambda}(u)+ o(1)}{\Vert v_i \Vert_{H^1_0}^2}= o(1)$$ 
as $i \to + \infty$, a contradiction with $\Vert \tilde{v}_i \Vert_{H^1_0} = 1$. Hence $(v_i)_{i\ge1}$ is bounded in $H^1_0(\Om)$. Denote by $\vp_1^\lambda(u)$ its weak limit in $H^1_0(\Om)$. We have $\vp_1^{\lambda}(u) \ge0$ a.e. in $\Om$ and we claim that, up to a subsequence, the following holds:
\begin{equation} \label{eq:cvforte}
\int_{\Om} u^{2^*-2}(v_i - \vp_1^\lambda(u))^2dx \to 0 \quad \text{ as } i \to + \infty. 
\end{equation}
We prove \eqref{eq:cvforte}. Let $R>0$. Then
$$ \begin{aligned}
\int_{\Om} u^{2^*-2}(v_i - \vp_1^{\lambda}(u))^2dx &= \int_{\{u \le R \}} u^{2^*-2}(v_i - \vp_1^{\lambda}(u))^2dx \\
&+ \int_{\{ u>R\}} u^{2^*-2}(v_i - \vp_1^\lambda(u))^2dx.
\end{aligned} $$
Up to passing to a subsequence $v_i \to \vp_1^\lambda(u)$ in $L^2(\Om)$, so the first integral in the right-hand side converges to $0$. Since the embedding $H^1_0(\Om) \subset L^{2^*}(\Om)$ is continuous and $\Vert v_i \Vert_{H^1_0} $ is bounded we obtain, by H\"older's inequality,
$$ \limsup_{i \to + \infty}\int_{\Om} u^{2^*-2}(v_i - \vp_1^\lambda(u))^2dx \le C \Big( \int_{\{u >R\}} u^{2^*}dx \Big)^{\frac{2^*-2}{2^*}} $$
for some $C>0$ independent of $R$. Letting $R \to + \infty$ proves \eqref{eq:cvforte}. The limiting function $\vp_1^\lambda(u)$ thus satisfies $\int_{\Om} u^{2^*-2} (\vp_1^{\lambda}(u))^2dx = 1$ and so $\vp_1^{\lambda}(u) \neq 0$. By weak lower semicontinuity of the $H^1_0$ norm,  $\vp_1^{\lambda}(u)$ attains $\mu_1^{\lambda}(u)$. That $\vp_1^{\lambda}(u)$ satisfies \eqref{vp1} follows form simple variational arguments. This proves point $(1)$ of Proposition \ref{prop:vp1}.

We now prove point (2), and we adapt here an argument from \cite[Proposition 2.4]{GurskyPerez}. We first prove that every solution of \eqref{vp1} has constant sign and vanishes on a set of measure zero. Let $\vp \in E_1^{\lambda}(u)$ be such that $\vp_+ = \max(\vp, 0) \neq 0$. We will prove that $|\{\vp_+ = 0 \}| = 0$, which will imply that $\vp \ge 0$ a.e. and $|\{\vp = 0\}| = 0$. By replacing $\vp$ by $-\vp$ if needed the same argument shows that if $\min(\vp,0) \neq 0$ then $\vp \le 0$ a.e. and $|\{\vp = 0\}| = 0$.  Let thus $\vp$ be a solution of \eqref{vp1} with $\vp_+ \neq 0$ and let $\psi = \vp_+$. Integrating \eqref{vp1} against $\psi$ shows that $\psi$ still attains $\mu_1^{\lambda}(u)$. Simple variational considerations thus show that $\psi$ still satisfies \eqref{vp1}. For $\ve >0$ we let $\psi_\ve = \sqrt{\ve^2 + \psi^2}$. Let $x_0 \in \Om$ and $r>0$ be such that $B(x_0, 2r) \subset \subset \Om$ and let $\eta \in C^\infty_c(B(x_0,2r))$ be such that $\eta \equiv 1$ in $B(x_0,r)$ and $\Vert \nabla \eta \Vert_{\infty} \le 2 r^{-1}$. We have 
\begin{equation} \label{vp1eps:1}
 \nabla \psi_\ve =\frac{\psi}{\psi_\ve} \nabla \psi \quad \text{ and } \quad  |\nabla \psi_\ve|^2 = \frac{\psi^2}{\psi_\ve^2} |\nabla \psi|^2 .
 \end{equation}
Integrating \eqref{vp1} against $\eta^2 \psi_\ve^{-1}$ and using Young's inequality shows that 
\ben \label{vp1eps:11}
\int_{\Om} \eta^2 \Bigg( \frac{|\nabla \psi|^2}{\psi_\ve^2} \frac{\psi}{\psi_\ve} - \frac12 \frac{|\nabla \psi|^2}{\psi_\ve^2}\Bigg) dx \le -\mu_1^{\lambda}(u) \int_{\Om} \eta^2 u^{2^*-2} \frac{\psi}{\psi_\ve}dx + 2 \int_{\Om} |\nabla \eta|^2dx. 
\een
Since $\psi \le \psi_\ve$ by definition, the right-hand side is easily estimated as 
$$ \Big|  -\mu_1^{\lambda}(u) \int_{\Om} \eta^2 u^{2^*-2} \frac{\psi}{\psi_\ve}dx + 2 \int_{\Om} |\nabla \eta|^2dx  \Big| \le C(1+ \Vert u \Vert_{2^*}^{2^*-2}) r^{n-2},$$
where $C>0$ is independent of $u$ and $r$. Since $\eta^2 |\nabla \psi|^2 \psi_\ve^{-2}$ and  $\eta^2 |\nabla \psi|^2 \psi_\ve^{-3}$ are integrable for every $\ve >0$ and non-increasing in $\ve$ the monotone convergence theorem applies separately to the two integrands in the left-hand side of \eqref{vp1eps:11}: passing \eqref{vp1eps:11} to the limit as $\ve \to 0$ then gives
$$\int_{B(x_0,r)} \frac{|\nabla \psi|^2}{\psi^2} dx  \le C(n,u) r^{n-2}. $$
By  \eqref{vp1eps:1} and since $\psi_\ve \ge  \psi$ we have $ |\nabla \ln \psi_\ve| = \frac{|\nabla \psi_\ve|}{\psi_\ve} \le  \frac{|\nabla \psi|}{\psi}$
 a.e. in $\Omega$. The latter inequalities then shows that for any $\ve $ small enough we have
\begin{equation*}
\int_{B(x_0,r)} |\nabla \ln \psi_\ve|^2 \le  C(n,u) r^{n-2},
\end{equation*}
where $C(n,u)$ is independent of $\ve$. With H\"older's inequality we then get 
$$\int_{B(x_0,s)} |\nabla \ln \psi_\ve|dx \le C s^{n-1} \quad \text{ for all } s \le r. $$
John-Nirenberg's inequality (see e.g. \cite[Theorem $7.21$]{GilTru}) shows the existence of some $p>0$ and $C$, independent of $\ve$ and depending only on $n,u,r$, such that 
$$ \int_{B(x_0,s)} e^{p|w_\ve|} dx \le C $$
for all $s \le r$, where we have let $w_\ve = \ln \psi_\ve - \frac{1}{|B(x_0,s)|} \int_{B(x_0,s)} \ln \psi_\ve dx$. Simple arguments then show that 
$$\Big( \int_{B(x_0,r)} \psi_\ve^p dx \Big) \Big(\int_{B(x_0,r)} \psi_\ve^{-p} dx\Big) \le C(n,u,r). $$
Letting $\ve \to 0$ by Fatou's lemma and monotone convergence then yields 
$$\Big( \int_{B(x_0,r)} \psi^p dx \Big) \Big(\int_{B(x_0,r)} \psi^{-p} dx\Big) \le C(n,u,r). $$
Since $\psi \neq 0$ by assumption this shows that $\int_{B(x_0,r)} \psi^{-p} dx < + \infty$. Hence $\{\psi = 0\} \cap B(x_0,r)$ has measure zero for every ball $B(x_0, 2r)$ compactly supported in $\Om$ and thus $\{\psi = 0\}$ has measure zero, which shows that $\vp >0$ a.e. in $\Omega$. 

We finally show that $E_1^{\lambda}(u)$ is one-dimensional. Let $\psi_1,\psi_2 \in H^1_0(\Om)$ be two solutions of \eqref{vp1} satisfying $\int_\Om \psi_1dx = \int_{\Om} \psi_2 dx = 1$. In particular $\psi_1 - \psi_2$ still solves \eqref{vp1}. Assume first that $(\psi_1 - \psi_2)_+ \neq 0$. The previous argument shows that $\psi_1 - \psi_2 >0$ a.e. which is impossible since $\int_{\Om} (\psi_1 - \psi_2)dx = 0$. Hence $\psi_1 - \psi_2 \le 0$. The same argument for $(\psi_1-\psi_2)_{-}$ shows that $\psi_1 = \psi_2$ a.e. in $\Om$. 
\end{proof}

We now consider generalised eigenvalues of higher order. We prove that for $u \in L^{2^*}_{>}(\Om)$ all the $\mkl(u)$, $k \ge 1$, are attained, we construct generalised eigenvectors and prove an equivalent characterisation of the $\mkl(u)$:

\begin{prop} \label{prop:vp}
Let $\Om \subset \R^n$, $n \ge 3$, be bounded and smooth. Let $\lambda \ge0$ and $u \in L^{2^*}_>(\Om)$. Then
\begin{enumerate}
\item For any $k \ge 1$, $\mkl(u)> - \infty$ and it is attained at some generalised eigenfunction $\vkl(u) \in H^1_0(\Om) \backslash \{0\}$ that satisfies weakly
\begin{equation} \label{vpk}
\big( - \Delta - \lambda \big) \vkl(u) = \mu_{k}^{\lambda}(u) u^{2^*-2}    \vkl(u) \quad \text{ in } \Om.
\end{equation}
The functions $(\vkl(u))_{k \ge 1}$ are normalised to satisfy
\begin{equation} \label{norma} 
\int_\Om u^{2^*-2} \vil(u) \vjl(u) dx =  \delta_{ij} \quad \text{ for } i \neq j,
\end{equation}
and for $k \ge 2$ each $\vkl(u)$ is sign-changing. 
\item The following equivalent characterisation of $\mkl(u)$ holds: for any $k \ge 1$,
\begin{equation} \label{muk2}
\mkl(u) = \inf_{v \in S_{k-1}} \frac{\int_{\Om}\left(  |\nabla v|^2 - \lambda v^2 \right) \, dx}{\int_{\Om} u^{2^*-2} v^2 \, dx} 
\end{equation}
where 
$$\begin{aligned} 
S_{k-1} & = \Big\{ v \in H^1_0(\Om)  \backslash \{0\} \text{ such that } \int_\Om u^{2^*-2} \vil(u) v dx = 0 \\
&\text{ for all } 1 \le i \le k-1 \Big\}.
\end{aligned} $$
\item For $k \ge 1$ we denote by $E_k^{\lambda}(u) \subset H^1_0(\Om)$ the vector space of solutions of \eqref{vpk}. Then $\dim E_k^{\lambda}(u) < + \infty$.

\item We have $\mkl(u) \le \mu_{k+1}^{\lambda}(u)$ for any $k \ge 1$ and there exists $\gamma = \gamma (u) >0$ such that $ \mu_1^{\lambda}(u) + \gamma \le \mu_{2}^{\lambda}(u)$.
\item We have $\mkl(u) \to + \infty$ as $k \to + \infty$. 
\item If there exist $\vp \in H^1_0(\Om) \backslash \{0\}$ and $\mu \in \R$ such that $\big( - \triangle - \lambda \big) \vp = \mu u^{2^*-2} \vp$ in  $\Om$ then $\mu =\mkl(u)$ for some $k \ge 1$ and $\vp \in E_k^{\lambda}(u)$. 

 \end{enumerate}
\end{prop}
When we say that for $k \ge 2$ each $\vkl(u)$ is sign-changing we mean that 
$$|\{\vkl(u) >0 \}| \cdot |\{\vkl(u) <0 \}|  >0.$$ 

\begin{proof}
Points (1) and (2) for $k=1$ have been proven in Proposition \ref{prop:vp1}. We prove (1) and (2) by induction: assume that (1) and (2) are true for some integer $k-1 \ge 1$, i.e. that generalised eigenvectors $(\vp_{1}^\lambda(u), \dots, \varphi_{k-1}^{\lambda}(u))$ attaining $\mu_{1}^{\lambda}(u), \dots, \mu_{k-1}^{\lambda}(u)$, satisfying \eqref{vpk} and \eqref{norma} have been constructed, and that \eqref{muk2} holds for $\mu_{\ell}^{\lambda}(u)$ for all $\ell \le k-1$. We first prove that \eqref{muk2} holds for $\mkl(u)$. We denote by $\tilde{\mu}_{k}^{\lambda}(u)$ the right-hand side in \eqref{muk2}. Let $v_0 \in S_{k-1}$. Then, by \eqref{norma},
$$ V := \text{Span} \Big \{ \vp_{1}^{\lambda}(u), \dots, \varphi_{k-1}^{\lambda}(u), v_0 \Big \} \subset H^1_0(\Om)$$
 satisfies $\dim_u V = k$. Up to multiplying $v_0$ by a constant we can assume that $\int_\Om u^{2^*-2} v_0^2 dx = 1$. The induction property shows that $\tilde{\mu}_{i}^{\lambda}(u) = \mu_{i}^{\lambda}(u)$ for any $1 \le i \le k-1$, so that $ \mu_1^{\lambda}(u)  \le \cdots \le  \mu_{k-1}^{\lambda}(u)$ by \eqref{muk2}. Since $v_0 \in S_{k-1}$ and $\int_\Om u^{2^*-2} v_0^2 dx = 1$ we have, by definition of $\tilde{\mu}_{k}^{\lambda}(u)$,
\begin{equation} \label{ineq:tildemuk2} 
 \int_{\Om}\left(  |\nabla v_0|^2 - \lambda v_0^2 \right) \, dx \ge \tilde{\mu}_{k}^{\lambda}(u) \ge \mu_{k-1}^{\lambda}(u). 
 \end{equation}
 Since $\dim_u V = k$ it is admissible in the definition of $\mkl(u)$. Using \eqref{muk}, \eqref{vpk}, \eqref{norma}, the normalisation of $v_0$ and \eqref{ineq:tildemuk2} we get 
 $$ \mkl(u) \le \max_{v \in V \backslash \{0\}} \frac{\int_{\Om}\left(  |\nabla v|^2 - \lambda v^2 \right) \, dx}{\int_{\Om} u^{2^*-2} v^2 \, dx}  = \int_{\Om}\left(  |\nabla v_0|^2 - \lambda v_0^2 \right)dx . $$
 Since $v_0 \in S_{k-1}$ is arbitrary this proves $ \mu_{k}^{\lambda}(u)  \le \tilde{\mu}_{k}^{\lambda}(u)$. For the reverse inequality we let $V \subseteq H^1_0(\Om)$ be such that $\dim_u V = k$. Then 
 $$\text{Span}( \vp_{1}^{\lambda}(u), \dots, \varphi_{k-1}^{\lambda}(u))^{\perp_u} \cap V \neq \{ 0\} \quad \text{ in } L^2_u(\Om)$$
 by simple dimensional considerations, where $\perp_u$ denotes the orthogonal complement in $L^2_u(\Om)$ for $(\cdot, \cdot)_{L^2_u(\Om)}$. We can thus let $v_0 \in V \cap S_{k-1}$. We then have
$$\begin{aligned}  \max_{ v \in V \backslash \{0\}}  \frac{\int_{\Om}\left(  |\nabla v|^2 - \lambda v^2 \right) \, dx}{\int_{\Om} u^{2^*-2} v^2 \, dx}  & \ge \frac{\int_{\Om}\left(  |\nabla v_0|^2 - \lambda v_0^2 \right) \, dx}{\int_{\Om} u^{2^*-2} v_0^2 \, dx}  \ge \tilde{\mu}_{k}^{\lambda}(u), 
\end{aligned} $$
where the second inequality follows from the definition of $\tilde{\mu}_{k}^{\lambda}(u)$ in \eqref{muk2}. Taking the infimum over subspaces $V$ finally yields $\mu_{k}^{\lambda}(u) \ge \tilde{\mu}_{k}^{\lambda}(u)$ and proves \eqref{muk2}. 

Once \eqref{muk2} has been established the existence of $\vkl(u)$ follows again from a standard minimisation procedure. Assume indeed that $\vp_1^{\lambda}(u), \dots, \vp_{k-1}^{\lambda}(u)$ have been constructed. We use \eqref{muk2} and let $(v_i)_{i \ge1}$ be a minimising sequence for $\mu_k^{\lambda}(u)$ in $S_{k-1}$ normalised as $\int_{\Om} u^{2^*-2}v_i^2dx = 1$. We first observe that $(v_i)_{i\ge1}$ is bounded in $H^1_0(\Om)$, which follows by a straightforward adaptation of the proof of Proposition \ref{prop:vp1} in the $k=1$ case. We denote by $\vp_k^{\lambda}(u)$ the weak limit of $(v_i)_{i\ge1}$ in $H^1_0(\Om)$, up to a subsequence. Then $\int_{\Om} u^{2^*-2} (\vp_k^{\lambda}(u))^2dx = 1$ holds, and this is proven by mimicking the proof of \eqref{eq:cvforte}. Also, $\vp_k^{\lambda}(u) \in S_{k-1}$ by weak convergence, which yields \eqref{norma}. The lower semi-continuity of the $H^1_0(\Om)$ norm then shows that $\vp_k^{\lambda}(u)$ attains $\mu_k^{\lambda}(u)$. The proof of \eqref{vpk} again follows from standard variational arguments. That every $k \ge 2$ is sign-changing follows from \eqref{norma} and the fact that $|\vp_{1}^\lambda(u)| >0$ a.e. in $\Om$ by Proposition \ref{prop:vp1}. This proves points (1) and (2) of Proposition \ref{prop:vp} for any $k\ge1$.

\medskip

We now prove (3). In Proposition \ref{prop:vp1} we already proved that $\dim E_1^\lambda(u) = 1$. Let $k \ge 1$ and let $(\vp_m)_{m \ge0}$ be a sequence of solutions of \eqref{vpk} associated to $\mkl(u)$ satisfying $\Vert \vp_m \Vert_{H^1_0} = 1$. We can assume, up to passing to a subsequence, that $\vp_m \rightharpoonup \vp_0$ in $H^1_0(\Om)$ as $m \to + \infty$. Since $\vp_m \to \vp_0$ strongly in $L^{\frac{n}{n-2}}(\Om)$ it is easily seen that $\vp_0 \in E_k^{\lambda}(u)$. We now claim that
\begin{equation*} 
\int_{\Om} u^{2^*-2}(\vp_m - \vp_0)^2dx \to 0 \quad \text{ as } m \to + \infty. 
\end{equation*}
The proof is identical to the proof of \eqref{eq:cvforte} and we omit it here. Since $\vp_m-\vp_0$ still satisfies \eqref{vpk} integrating it against $\vp_m-\vp_0$ and using the latter yields
$$ \int_{\Om} |\nabla (\vp_m - \vp_0)|^2dx = o(1) \quad \text{ as } m \to + \infty, $$
so that $\vp_m \to \vp_0$ in $H^1_0(\Om)$, and hence the unit ball of $E_k^{\lambda}(u)$ for $\Vert \cdot \Vert_{H^1_0}$ is compact. This concludes the proof of (3).

The first part of point (4) easily follows from the equivalent characterisation \eqref{muk2}. That $ \mu_1^{\lambda}(u) < \mu_{2}^{\lambda}(u)$ for any $u \in L^{2^*}_{>}(\Om)$ follows from  point (2) of Proposition \ref{prop:vp1} and from \eqref{vpk} and \eqref{norma}. 

We now prove point (5). We proceed by contradiction and assume that the sequence $(\mkl(u))_{k \ge1}$ is bounded. We claim that, as a consequence, $E = \overline{\bigoplus_{k\ge 1} E_k^{\lambda}(u)}$ is finite-dimensional, where the closure is taken in $H^1_0(\Om)$ for the $\Vert \cdot \Vert_{H^1_0}$ norm. This will be an obvious contradiction since the family $(\vkl(u))_{k\ge 1}$ is free in $H^1_0(\Om)$ by \eqref{norma}. To prove that $E$ is finite-dimensional, we let $(\vp_m)_{m \ge0}$ be a sequence of functions in the unit ball of $E$. For any $m \ge0$ there exists in particular $\mu_m \in \R$ such that $\vp_m$ satisfies 
\begin{equation*}
\big( - \Delta - \lambda \big) \vp_m = \mu_m u^{2^*-2} \vp_m \quad \text{ in } \Om.
\end{equation*}
Since $\Vert \vp_m \Vert_{H^1_0} = 1$ and $(\mu_m)_{m \ge 0}$ is bounded, the same arguments as in the proof of point (3) show that $(\vp_m)_{m \ge0}$ has a strongly convergent subsequence as $m \to + \infty$, and hence that the unit ball of $E$ is compact, which is a contradiction. 

We finally prove point (6). Assume that there exists $\vp \in H^1_0(\Om) \backslash \{0\}$ and $\mu \in \R$ such that $\big( - \triangle - \lambda \big) \vp = \mu u^{2^*-2} \vp$ in  $\Om$ and assume that $\mu \not \in (\mkl(u))_{k \ge 1}$. Let $k \ge 1$ be such that $\mu_{k-1}^{\lambda}(u)<  \mu < \mkl(u)$ (where we have let $\mu_0^{\lambda}(u) = - \infty$ by convention) and let 
$$ V = \text{Span} \Big \{ \vp_1^{\lambda}(u), \dots, \vp_{k-1}^{\lambda}(u),\vp \Big \}. $$
By \eqref{vpk} we have $(\vp, \vil(u))_{L^2_u} = 0$ for all $1 \le i \le k-1$. As a consequence, the variational characterisation \eqref{muk} of $\mkl(u)$ shows that $\mkl(u) \le \mu$, which is a contradiction. Hence $\mu = \mkl(u)$ for some $k \ge 1$. This also proves that for all $k \ge 1$ we have 
$$ E_k^{\lambda}(u) = \text{Span} \Big \{  \vil(u) ; i \ge 1, \mil(u) = \mkl(u) \Big \}, $$
and concludes the proof of (6) and of Proposition \ref{prop:vp}. 
\end{proof}

We conclude this section by proving that every solution of \eqref{eq:critlambda} is a generalised eigenvector. This property is the main observation that will allow us to construct least-energy solutions of \eqref{eq:critlambda} via $\mkl(\Om)$ (see Theorem \ref{prop:minmkl1} below). 

\begin{prop} \label{prop:structure:vp}
Let $\lambda \ge 0$. Let $u \in H^1_0(\Om) \backslash \{0\}$ and let $\tilde{u} = \frac{u}{\Vert u \Vert_{2^*}}$. Then $u$ is a solution of \eqref{eq:critlambda} if and only if there exists $k \ge 1$ such that $\mkl(|\tilde{u}|)  = \Vert u \Vert_{2^*}^{2^*-2}$ and $\tilde u \in E_k^{\lambda}(|\tilde{u}|)$. 
\end{prop}
In other words, up to a scaling factor, $u$ solves \eqref{eq:critlambda} if and only if it is a generalised eigenvector, with positive generalised eigenvalue, associated to the weight $|u|$. 

\begin{proof}
Assume first that $u$ satisfies \eqref{eq:critlambda}. It is then easily seen that $\tilde u$ satisfies 
$$-\triangle \tilde{u} - \lambda \tilde{u} = \mu |\tilde{u}|^{2^*-2} \tilde{u} \quad \text{ in } \Om, $$
where $\mu= \Vert u \Vert_{2^*}^{2^*-2}$. Point (6) of Proposition \ref{prop:vp} applies and shows that there exists $k \ge 1$ such that $\mkl(|\tilde{u}|) =  \Vert u \Vert_{2^*}^{2^*-2}$ and such that $\tilde{u} \in E_k^{\lambda}(|\tilde{u}|)$. Assume now that there exists $k \ge 1$ such that $\mkl(|\tilde{u}|)  = \Vert u \Vert_{2^*}^{2^*-2}$ and $\tilde{u} \in E_k^{\lambda}(|\tilde{u}|)$. By \eqref{vpk} $\tilde{u}$ then satisfies 
$$ \big(-\triangle - \lambda \big) \tilde{u} = \Vert u \Vert_{2^*}^{2^*-2} |\tilde{u}|^{2^*-2} \tilde{u} \quad \text{ in } \Om, $$
which is equivalent to saying that $u$ satisfies \eqref{eq:critlambda}. 
\end{proof}

\section{Variational theory for generalised eigenvalues} \label{vp:derivees}

In this section we develop a variational theory for generalised eigenvalues. We first investigate their sign:

 \begin{prop}  \label{prop:signvp}
  Let $\Om \subset \R^n$, $n \ge 3$, be bounded and smooth and let  $\lambda \ge 0$. Let $i \ge0$ be such that $\lambda \in [\Lambda_i, \Lambda_{i+1})$ and let $k = N(i)+1 \ge 1$. We have  
\begin{equation} \label{infmukpos}
 \inf_{u \in L^{2^*}_{>}(\Om)} \mu_k^{\lambda}(u)  \Vert u \Vert_{2^*}^{2^*-2}  > 0
 \end{equation}
and, for any $1 \le p \le k-1$,
\begin{equation} \label{supmupneg}
\sup_{u \in L^{2^*}_{>}(\Om)} \mpl(u)  \Vert u \Vert_{2^*}^{2^*-2} \le 0 .
\end{equation}
\end{prop}
We recall that $L^{2^*}_{>}(\Om)$ is defined in \eqref{defLu} and the multiplicity function $N$ is defined in \eqref{defN}. The normalisation factor $\Vert u \Vert_{2^*}^{2^*-2}$ is required so that the quantities appearing in \eqref{infmukpos} and \eqref{supmupneg} are scale-invariant, as is easily seen by \eqref{muk}. Proposition \ref{prop:signvp} shows in particular that if $\lambda \in [\Lambda_i, \Lambda_{i+1})$ and $k = N(i)+1$, for any $u \in L^{2^*}_{>}(\Om)$ we have
$$ \mpl(u) \le0 \text{ for all } p\le k-1 \quad \text{ and } \quad \mpl(u) >0 \text{ for all } p \ge k.$$
When $\lambda$ belongs to a fixed set $[\Lambda_i, \Lambda_{i+1})$ between two consecutive eigenvalues the generalised eigenvalue $\mkl(u)$, for $k = N(i)+1$, is thus the least positive one, regardless of $u$. We will call it the \emph{principal} generalised eigenvalue in this paper.

\begin{proof}
Let $u \in L^{2^*}_{>}(\Om)$ and let $(\vpl(u))_{p \ge 1}$ be the eigenvectors given by Proposition \ref{prop:vp}. Let $i \ge0$ be such that $\lambda \in [\Lambda_i, \Lambda_{i+1})$. We first prove that $\mpl(u) \ge 0$ for all $p \ge k$, where $k = N(i)+1$. Consider first the case $u \equiv 1$: by \eqref{muk},
$$\mu_{k}^{\lambda}(1) =  \lambda_{k} - \lambda  = \Lambda_{i+1} - \lambda >0 $$
by definition of $N$ in \eqref{defN}. Let $V\subseteq H^1_0(\Om)$ be such that $\dim_u V = \dim V =  k$. It is also admissible in the definition \eqref{muk} for $\mu_{k}^{\lambda}(1)$: since $\mu_{k}^{\lambda}(1) >0$ there thus exists $v_0 \in V \backslash \{0\}$ such that $\int_{\Om} \big( |\nabla v_0|^2 - \lambda v_0^2\big)dx >0$. Since $\int_\Om u^{2^*-2} v_0^2 dx >0$ we then have 
$$ \max_{v \in V \backslash \{0\}} \frac{\int_{\Om}\left(  |\nabla v|^2 - \lambda v^2 \right) \, dx}{\int_{\Om} u^{2^*-2} v^2 \, dx}  \ge  \frac{\int_{\Om}\left(  |\nabla v_0|^2 - \lambda v_0^2 \right) \, dx}{\int_{\Om} u^{2^*-2} v_0^2 \, dx} > 0.$$
This holds true for any $V$ with $\dim_u V = k$ and thus $\mu_{k}^{\lambda}(u) \ge 0$ by \eqref{muk}. Similar arguments show that $\mu_p^{\la}(u) \le 0$ for every $p \le k-1$, which proves \eqref{supmupneg}. 

We now claim that $\mkl(u) >0$ for all $u \in L^{2^*}_{>}(\Om)$. Assume by contradiction that there exists $u \in L^{2^*}_{>}(\Om)$ such that  $\mkl(u) = 0$ and let $\vkl(u)$ be the $k$-th eigenfunction associated to $\mkl(u)$ by Proposition \ref{prop:vp}. It satisfies  $- \Delta \vkl(u) - \lambda \vkl(u) = 0$ in $\Om$. If $\lambda \in (\Lambda_i, \Lambda_{i+1})$, so that $\lambda \not \in \Sp(- \Delta)$, this is an obvious contradiction. If $\lambda = \Lambda_i$, $\vkl(u) \in E_{\Lambda_i}(\Om)$, where we recall that $E_{\Lambda_i}(\Om)$ is the $i$-th eigenspace for $-\Delta$ in $H^1_0(\Om)$, and thus satisfies $- \Delta \vkl(u)  - \Lambda_i \vkl(u)  = 0$ in $\Om$. Let $\vp_1, \dots, \vp_{k-1}$ be the eigenvectors of $-\Delta$ associated to $\lambda_1, \dots, \lambda_{k-1}$. Integrating the latter against $\vp_\ell$ shows that $\int_{\Om} \vp_\ell \vkl(u) dx = 0$ for all $1 \le \ell \le k-1$, but this would then imply that 
$$\dim \Big( \bigoplus_{p=1..i} E_{\Lambda_p}(\Om) \Big) \ge k =  N(i)+1,$$
which contradicts the definition of $N$ in \eqref{defN}. Hence $\mkl(u) > 0$. 

\medskip

We now prove \eqref{infmukpos}. We proceed by contradiction and assume that there is a sequence $(u_m)_{m \ge1}$ in $L^{2^*}_{>}(\Om)$, $\Vert u_m \Vert_{2^*} = 1$, such that $\mu_{k,m} = \mkl(u_m) \to 0$ as $m \to + \infty$. Let $\vp_{k,m}$ be a generalized eigenvector associated to $\mu_{k,m}$ by Proposition \ref{prop:vp} and normalised by $\int_{\Om} u_m^{2^*-2} \vp_{k,m}^2dx = 1$. It satisfies:
\begin{equation} \label{eq:eqvp:00}
  - \Delta \vp_{k,m} - \lambda \vp_{k,m} = \mu_{k,m} u_m^{2^*-2} \vp_{k,m} \quad \text{ in } \Om. 
  \end{equation}
Assume first that $\lambda \in (\Lambda_i, \Lambda_{i+1})$, so that $\lambda \not \in \Sp(- \Delta)$. Then $- \Delta - \lambda$ is invertible and thus there exists $C = C_{\lambda}$ such that 
$$ \Vert \vp_{k,m} \Vert_{H^1_0} \le C \mu_{k,m} \Vert u_m^{2^*-2} \vp_{k,m} \Vert_{\frac{2n}{n+2}}.$$
H\"older's inequality shows that 
\begin{equation} \label{holder:spectre}
 \Vert u_m^{2^*-2} \vp_{k,m} \Vert_{\frac{2n}{n+2}}^{\frac{2n}{n+2}} \le \Big(\int_{\Om} u_m^{2^*}dx\Big)^{\frac{2}{n+2}} \Big( \int_{\Om} u_m^{2^*-2} \vp_{k,m}^2 dx\Big)^{\frac{n}{n+2}} = 1, 
 \end{equation}
and since $\mu_{k,m} \to 0$ we thus obtain that $ \Vert \vp_{k,m} \Vert_{H^1_0} = o(1)$ as $m \to + \infty$. Since $\Vert u_m \Vert_{2^*} = 1$ this contradicts the normalisation condition $\int_{\Om} u_m^{2^*-2} \vp_{k,m}^2dx = 1$. Assume now that $\lambda = \Lambda_i$. We split $\vp_{k,m}$ as 
$$ \vp_{k,m} = v_m + z_m, $$
where $z_m \in E_{\Lambda_i}(\Om)$ and $v_m$ is $H^1_0(\Om)$-orthogonal to $E_{\Lambda_i}(\Om)$. Since $-\Delta z_m - \Lambda_i z_m = 0$, and by \eqref{eq:eqvp:00}, $v_m$ satisfies 
\begin{equation} \label{eq:intermediaire}
 - \Delta v_m - \Lambda_i v_m = \mu_{k,m} u_m^{2^*-2} \vp_{k,m} \quad \text{ in } \Om. 
 \end{equation}
By definition $- \Delta-\Lambda_i$ is invertible in $E_{\Lambda_i}(\Om)^{\perp}$ and the same argument as above shows that $\Vert v_m \Vert_{H^1_0} = o(1)$ as $m \to + \infty$. Thus $\int_{\Om} u_m^{2^*-2} v_m^2 dx = o(1)$ and, by Cauchy-Schwarz's inequality, $\int_{\Om} u_m^{2^*-2} \vp_{k,m} v_m dx = o(1)$ as $m \to + \infty$. Integrating \eqref{eq:intermediaire} against $z_m$ thus yields, since $\mu_{k,m} >0$ for all $m \ge 0$,
$$ 0 = \int_{\Om} u_m^{2^*-2} \vp_{k,m} z_m = \int_{\Om} u_m^{2^*-2} \vp_{k,m}^2 dx-  \int_{\Om} u_m^{2^*-2} \vp_{k,m} v_m dx = 1 + o(1), $$
a contradiction. This proves \eqref{infmukpos}. 
 \end{proof}

One of the main difficulties that we face is that the functionals $u \in L^{2^*}_>(\Om) \mapsto \mpl(u)$ are not even differentiable in general. In the next Proposition we prove however that we can compute directional derivatives of the principal eigenvalue $u \mapsto \mkl(u)$ along interior variations of sorts at a given weight $u$:

\begin{prop} \label{prop:dervp}
Let $\Om \subset \R^n$, $n \ge 3$, be bounded and smooth and let  $\lambda \ge \Lambda_1$. Let $i \ge1$ be such that $\lambda \in [\Lambda_i, \Lambda_{i+1})$ and let $k = N(i)+1$. Let $u \in L^{2^*}_>(\Om)$, $h \in L^\infty(\Om)$ and, for $t \in \R$, let 
$$ u_t = (1+th)u. $$
Then the left and right derivatives of $t \mapsto \mkl ( u_t )$ at $0$ exist and are given by
\begin{equation} \label{eq:derleft}
\frac{d}{dt}_{|t=0_+}  \mkl ( u_t) = \inf_{\vp \in E_k^{\lambda}(u)\backslash \{0\}} \Bigg( - (2^*-2) \mkl(u) \frac{\int_\Om u^{2^*-2}h \vp^2dx}{\int_\Om u^{2^*-2} \vp^2 dx}\Bigg)
\end{equation}
and
\begin{equation} \label{eq:derright}
\frac{d}{dt}_{|t=0_-}  \mkl ( u_t) = \sup_{\vp \in E_k^{\lambda}(u)\backslash \{0\}} \Bigg( - (2^*-2) \mkl(u) \frac{\int_\Om u^{2^*-2}h \vp^2dx}{\int_\Om u^{2^*-2} \vp^2 dx}\Bigg).
\end{equation}
\end{prop}
We point out that, when $\lambda \in [\Lambda_i, \Lambda_{i+1})$ for some $i \ge 1$, Proposition \ref{prop:dervp} only holds for the specific value $k = N(i)+1$, that is for the principal generalised eigenvalue $\mkl(u)$. The reason for this is as follows: left and right differentiability of $t \mapsto \mkl(u_t)$ art $0$ relies on a strict spectral gap between $\mkl(u_t)$ and all the $\mpl(u_t)$, $p \le k-1$, that needs to be uniform in $t$ for $t$ small enough (see \eqref{eq:der1} below). When $k = N(i)+1$ this gap easily follows from Proposition \ref{prop:signvp}. Note that we do \emph{not} assume a priori the simplicity of $\mkl(u)$ in Proposition \ref{prop:dervp}. For a generalised eigenvalue $\mu_p^{\lambda}(u)$ with $p \ge k$ it is possible to compute directional derivatives, but the expressions are in general more complicated than \eqref{eq:derleft} and \eqref{eq:derright}. See for instance \cite{HumbertPetridesPremoselli} for a recent unified treatment of eigenvalue functions of conformally covariant operators in a geometric setting where similar computations are carried on. See also \cite{PetridesTewodrose} for a general framework to study critical points of eigenvalue-type functionals. 

\begin{proof}
We prove \eqref{eq:derleft}. Throughout the proof we let $k = N(i)+1$ where $N$ is the multiplicity function defined in \eqref{defN}. Let $u\in L^{2^*}_{>}(\Om)$, $h \in L^\infty(\Om)$ and $t \in \R$. For $|t|< \Vert h \Vert_{\infty}^{-1}$, $u_t = (1+th)u \in L^{2^*}_{>}(\Om)$ so that $\mu_p^{\lambda}(u_t)$ is well-defined for any $p \ge 1$. By Proposition \ref{prop:signvp} we also have $\mkl(u_t)>0$. Since $\lambda$ is fixed we will adopt for simplicity the following notations in this proof: $\mu_{p,t} = \mpl(u_t)$ and $E_{p,t} = E_p^{\lambda}(u_t)$ for any $p \ge1$. In particular, $E_{p,0} = E_p^{\lambda}(u)$. We will denote by $\Pi_{p,t}$ the orthogonal projection over $E_{p,t}$ in $L^2_{u_t}(\Om)$. The notation $O(t)$ will denote a quantity that can be uniformly bounded from above in absolute value by $C|t|$, where $C$ is a positive constant that does not depend on $t$. In the rest of the proof $Q_u^{\lambda}$ is as in \eqref{defQu}. We use in this proof arguments from  Gursky-Perez \cite{GurskyPerez} and Petrides \cite{Petrides2}. 

\medskip

\textbf{Step $1$:} We first prove that, for any $p \ge 1$ fixed,
\begin{equation} \label{eq:der2}
\mu_{p,t} = \mu_{p}^{\lambda}(u) + O(t) \quad \text{ as } t \to 0. 
\end{equation}
Let first $p \le  k-1$, so that  $\mu_{p}^{\lambda}(u) \le 0$ by Proposition \ref{prop:signvp}. Assume first that $\mpl(u)<0$. Let $V \subset H^1_0(\Om)$ with $\dim_u V= p$ and such that $ \max_{v \in V \backslash \{0\}} Q_u^{\lambda}(v) \le0$. For $t$ small enough we also have $\dim_{u_t} V = p$ and thus 
$$ \frac{1}{\big(1 - t \Vert h \Vert_{\infty}\big)^{2^*-2}}  \max_{v \in V \backslash \{0\}} Q_u^{\lambda}(v) \le  \max_{v \in V \backslash \{0\}} Q_{u_t}^{\lambda}(v) \le \frac{1}{\big(1 + t \Vert h \Vert_{\infty}\big)^{2^*-2}}  \max_{v \in V \backslash \{0\}} Q_u^{\lambda}(v). $$
Taking the infimum over $V$ yields, by \eqref{muk},
$$  \frac{1}{\big(1 - t \Vert h \Vert_{\infty}\big)^{2^*-2}}  \mu_{p}^{\lambda}(u) \le \mu_{p}^{\lambda}(u_t) \le  \frac{1}{\big(1 + t \Vert h \Vert_{\infty}\big)^{2^*-2}}    \mu_{p}^{\lambda}(u),$$
which proves \eqref{eq:der2}. Assume now that $\mpl(u) = 0$ and let $V \subset H^1_0(\Om)$ with $\dim_u V= p$. By \eqref{muk} we have $ \max_{v \in V \backslash \{0\}} Q_u^{\lambda}(v) \ge 0$. Then
$$ \frac{1}{\big(1 + t \Vert h \Vert_{\infty}\big)^{2^*-2}}  \max_{v \in V \backslash \{0\}} Q_u^{\lambda}(v) \le  \max_{v \in V \backslash \{0\}} Q_{u_t}^{\lambda}(v) \le \frac{1}{\big(1 - t \Vert h \Vert_{\infty}\big)^{2^*-2}}  \max_{v \in V \backslash \{0\}} Q_u^{\lambda}(v), $$
and taking the infimum over $V$ yields $\mpl(u_t) =0$, which again proves \eqref{eq:der2} for $p \le k-1$. The proof for $p \ge k$ follows the same lines: since $\mpl(u) >0$ by Proposition \ref{prop:signvp}, the only difference is that $ \max_{v \in V \backslash \{0\}} Q_u(v) >0$ for any $V \subset H^1_0(\Om)$ with $\dim_u V= k$. With Proposition \ref{prop:signvp} a consequence of \eqref{eq:der2} is that
\begin{equation} \label{eq:der1}
\liminf_{t \to 0} \big( \mu_{k,t} - \mu_{p,t} \big) >0 \quad \text{ holds for any } p \le k-1.
\end{equation}

\medskip

\textbf{Step $2$:} Let $p \le k$ and let $\vp_{t} \in E_{p,t}$ be such that $\int_\Om u_t^{2^*-2} \vp_{t}^2dx = 1$. We claim that $(\vp_{t})_{t>0}$ is bounded in $H^1_0(\Om)$ for $t$ small enough. We proceed by contradiction and assume there is a subsequence $(t_m)_{m \ge0}$, $t_m \to 0$ as $m \to + \infty$ such that $\Vert \vp_{m} \Vert_{H^1_0} \to + \infty$, where we have let $\vp_{m}=\vp_{t_m}$. For $m \ge 1$ we let $\tilde{\vp}_{m} =  \frac{\vp_{m}}{\Vert \vp_{m}\Vert_{H^1_0}}$. Up to passing to a subsequence there exists $\tilde{\vp}_0 \in H^1_0(\Om)$ such that $\tilde{\vp}_{m} \rightharpoonup \tilde{\vp}_0$ in $H^1_0(\Om)$ and $\tilde{\vp}_{m} \to \tilde{\vp}_0$ in $L^2(\Om)$. Since $\int_{\Om} u_{t_m}^{2^*-2} \vp_{m}^2dx = 1$ we have
$$ \int_{\Om} u_{t_m}^{2^*-2} \tilde{\vp}_{m}^2dx \to 0 $$
as $ m \to + \infty$, so that Fatou's lemma yields $\int_\Om u^{2^*-2} \tilde{\vp}_0^2dx \le 0$. Since $u \in L^{2^*}_{>}(\Om)$ this implies that $\tilde{\vp}_0 = 0$ p.p. in $\Om$. Integrating \eqref{vpk} against $\vp_{m}$ we have 
$$\int_{\Om}\big( |\nabla \tilde{\vp}_{m}|^2 - \lambda \tilde{\vp}_{m}^2 \big) dx  = \frac{\mu_{p,t_m}}{\Vert \vp_{m}\Vert_{H^1_0}^2} \to 0 $$
as $m \to + \infty$ by \eqref{eq:der2}. Since  $\tilde{\vp}_0 = 0$ this yields $\Vert \tilde{\vp}_{m} \Vert_{H^1_0}^2 = o(1)$, a contradiction with $\Vert \tilde{\vp}_{m} \Vert_{H^1_0} = 1$. Hence $(\vp_{t})_{t>0}$ is bounded in $H^1_0(\Om)$ for $t$ small enough.

\medskip

\textbf{Step $3$:} Let $p \le k$. We claim that, for $t$ small enough,
\begin{equation} \label{eq:der2bis}
\dim_{u_t} E_{p,t} \le \dim_u E_{p,0}.
\end{equation}
Let $p \le k$ and let $\vp_t \in E_{p,t}$ be such that $\int_\Om u_t^{2^*-2} \vp_t^2dx = 1$ for $t$ small enough. By Step $2$ the family $(\vp_t)_{t >0}$ is bounded in $H^1_0(\Om)$ as $t \to 0$ and weakly converges in $H^1_0(\Om)$ towards some $\vp_0$ in $H^1_0(\Om)$. Since $u_t \to u$ strongly in $L^{2^*}(\Om)$, using \eqref{eq:der2} we see that $\vp_0 \in E_{p,0}$. Up to passing to a subsequence in $t$ we can assume that $\Vert \vp_t - \vp_0\Vert_{2} \to 0$ as $t \to 0$. We then have, integrating \eqref{vpk} by parts: 
\begin{equation} \label{eq:der2ter}
 \begin{aligned}
\int_{\Om}& |\nabla(\vp_t - \vp_0)|^2dx = o(1) + \mu_{p}^{\lambda}(u)\int_{\Om} u^{2^*-2} (\vp_t - \vp_0)^2 dx \\
&+ (\mu_{p,t} - \mu_{p}^{\lambda}(u)  )\int_{\Om}u^{2^*-2} \vp_t(\vp_t - \vp_0) dx \\
&+ \mu_{p,t} \int_{\Om}(u_t^{2^*-2} - u^{2^*-2} ) \vp_t(\vp_t - \vp_0) dx.
\end{aligned} \end{equation}
By \eqref{eq:der2} and since 
\begin{equation} \label{eq:DLu}
\big| u_t^{2^*-2} - u^{2^*-2} \big| \le C |t| \Vert h \Vert_{\infty} u^{2^*-2} \quad \text{ a.e. in } \Om 
\end{equation}
we have 
$$ \begin{aligned}
& (\mu_{p}^{\lambda}(u) - \mu_{p,t} )\int_{\Om}u^{2^*-2} \vp_t(\vp_t - \vp_0) dx \\
&+ \mu_{p,t} \int_{\Om}(u_t^{2^*-2} - u^{2^*-2} ) \vp_t(\vp_t - \vp_0) dx = o(1) 
\end{aligned}
$$
as $t \to 0$. Since $(\vp_t)_{t>0}$ is bounded in $L^{2^*}(\Om)$, mimicking the proof of \eqref{eq:cvforte} it is also easily seen that
$$ \int_{\Om} u^{2^*-2} (\vp_t - \vp_0)^2 dx  = o(1)$$
as $t \to 0$. With \eqref{eq:der2ter} we have thus shown that $\vp_t \to \vp_0$ in $H^1_0(\Om)$ up to a subsequence as $t \to 0$, and $\vp_0$ thus satisfies $\int_\Om u^{2^*-2} \vp_0^2dx = 1$. Hence any finite family $(\vp_{1,t}, \dots, \vp_{\ell,t})$ of $E_{p,t}$ that is orthonormal for the $L^{2}_{u_t}(\Om)$ scalar product strongly converges, up to a subsequence as $t \to 0$, to an orthonormal family $(\vp_1, \dots, \vp_\ell)$ of $E_{p,0}$ for the $L^2_u(\Om)$ scalar product, which proves \eqref{eq:der2bis}.

\medskip

\textbf{Step $4$:} We claim that, for any $\phi \in E_{k,0}$,
\begin{equation} \label{eq:der3}
\Vert \Pi_{p,t}(\phi) \Vert_{L^2_{u_t}} = O(t) \quad \text{ for any } p \le k-1
\end{equation}
holds. Let indeed $\phi \in E_{k,0}$ and, for any $t \neq 0$ small enough, $\vp_t \in E_{p,t}$ be such that $\int_\Om u_t^{2^*-2} \vp_t^2dx = 1$. Integrating \eqref{vpk} for $\phi$ against $\vp_t$ yields
\begin{equation} \label{eq:der4}
\begin{aligned}
 \mkl(u) \int_{\Om} u^{2^*-2}& \phi \vp_t dx  = \mu_{p,t} \int_{\Om} u_t^{2^*-2}\phi  \vp_t  dx \\
 & =  \mu_{p,t}  \int_{\Om} u^{2^*-2} \phi \vp_t dx + \mu_{p,t} \int_{\Om} \big( u_t^{2^*-2} - u^{2^*-2}\big)\phi  \vp_t  dx.
\end{aligned}
\end{equation}
By Step $2$ $(\vp_t)_{t>0}$ is bounded in $H^1_0(\Om)$ as $t \to 0$. With \eqref{eq:DLu} we thus have 
$$ \Big| \int_{\Om} \big( u_t^{2^*-2} - u^{2^*-2}\big)\phi  \vp_t  dx \Big| = O(t) $$
as $ t \to 0$. Coming back to \eqref{eq:der4} with the latter and using \eqref{eq:der1} shows that
$$\int_{\Om} u^{2^*-2} \phi \vp_t dx = O(t) \quad \text{ as } t \to 0.$$
Since, by \eqref{eq:der2bis}, the dimension of $E_{p,t}$ as a subspace of $L^2_{u_t}(\Om)$ is bounded from above as $t \to 0$ this proves \eqref{eq:der3}.
 
 \medskip
 
 \textbf{Step $5$:} We claim that
 \begin{equation} \label{eq:der5}
\limsup_{t \to 0_+} \frac{\mu_{k,t} - \mkl(u)}{t}\le \min_{\phi \in E_{k,0}\backslash \{0\}} \Bigg( - (2^*-2) \mkl(u) \frac{\int_\Om u^{2^*-2}h \phi^2dx}{\int_\Om u^{2^*-2} \phi^2 dx}\Bigg).
\end{equation}
For $t$ small enough we denote by $\Pi_t$ the orthogonal projection over $\oplus_{i=1}^{k-1} E_{i,t}$ in $L^2_{u_t}(\Om)$. Let $\phi \in E_{k,0}\backslash \{0\}$. By \eqref{muk2} we have
\begin{equation} \label{eq:der6}
\mu_{k,t} \int_{\Om} u_t^{2^*-2}\big( \phi - \Pi_t(\phi) \big)^2dx \le \int_{\Om} \Big(\big| \nabla \big( \phi - \Pi_t(\phi)\big)\big|^2 - \lambda  \big( \phi - \Pi_t(\phi)\big)^2 \Big)dx.
\end{equation}
First, direct computations with \eqref{vpk} and \eqref{eq:der3} show that 
$$ \begin{aligned}
 &\int_{\Om} \Big(\big| \nabla \big( \phi - \Pi_t(\phi)\big)\big|^2 - \lambda  \big( \phi - \Pi_t(\phi)\big)^2 \Big)dx \\
 & = \mkl(u) \int_{\Om} u^{2^*-2}\phi^2 dx - 2 \mkl(u) \int_{\Om} u^{2^*-2} \phi \Pi_t(\phi)dx + O(t^2) .
\end{aligned} $$
Using \eqref{eq:DLu} and \eqref{eq:der3} we also have, by $L^2_{u_t}$-orthogonality,
$$ \begin{aligned}
 \int_{\Om} u^{2^*-2} \phi \Pi_t(\phi)dx & = \int_{\Om} u_t^{2^*-2} \big(\phi - \Pi_t(\phi) \big) \Pi_t(\phi) dx + O(t^2)  = O(t^2),
\end{aligned} $$
so that 
\begin{equation} \label{eq:der7}
 \begin{aligned}
 &\int_{\Om}\big| \nabla \big( \phi - \Pi_t(\phi)\big)\big|^2 - \lambda  \big( \phi - \Pi_t(\phi)\big)^2dx =\mkl(u) \int_{\Om} u^{2^*-2}\phi^2 dx + O(t^2) .
\end{aligned} 
\end{equation}
Independently, and again by \eqref{eq:der3} and $L^2_{u_t}$-orthogonality,
\begin{equation} \label{eq:der8}
 \begin{aligned}
 \int_{\Om} u_t^{2^*-2}\big( \phi - \Pi_t(\phi) \big)^2dx & =  \int_{\Om} u_t^{2^*-2} \phi^2dx - \int_{\Om} u_t^{2^*-2} \Pi_t(\phi)^2dx \\
 & =  \int_{\Om} u_t^{2^*-2} \phi^2dx  + O(t^2). 
 \end{aligned} 
\end{equation}
Plugging \eqref{eq:der7} and \eqref{eq:der8} in \eqref{eq:der6}, expanding $u_t^{2^*-2}$ with \eqref{eq:DLu} and using \eqref{eq:der2} we obtain, for $t>0$ small enough, that 
$$ \frac{\mu_{k,t} - \mkl(u)}{t}\le  - (2^*-2) \mkl(u) \frac{\int_\Om u^{2^*-2}h \phi^2dx}{\int_\Om u^{2^*-2} \phi^2 dx} + O(t).$$ 
By Proposition \ref{prop:signvp} we have $\mkl(u) >0$, so that taking the $\limsup$ as $t \to 0_+$ and then taking the minimum over $\phi \in E_{k,0}$ yields \eqref{eq:der5}. 

\medskip

 \textbf{Step $6$:} We claim that
 \begin{equation} \label{eq:der9}
\lim_{t \to 0_+} \frac{\mu_{k,t} - \mkl(u)}{t}= \min_{\phi \in E_{k,0}\backslash \{0\}} \Bigg( - (2^*-2) \mkl(u) \frac{\int_\Om u^{2^*-2}h \phi^2dx}{\int_\Om u^{2^*-2} \phi^2 dx}\Bigg).
\end{equation}
Let $(t_m)_{m \ge1}, t_m >0, t_m \to 0$ as $m \to + \infty $ be a subsequence along which 
$$ \lim_{m \to + \infty} \frac{\mu_{k,t_m} - \mkl(u)}{t_m}   = \liminf_{ t \to 0_+} \frac{\mu_{k,t} - \mkl(u)}{t}.$$
 Let, for any $m \ge 1$, $\vp_{m} \in E_{k,t_m}$ be such that $\int_\Om u_{t_m}^{2^*-2} \vp_{m}^2dx = 1$. Step $2$ shows that $(\vp_m)_{m\ge0}$ is bounded in $H^1_0(\Om)$. We denote by $\vp_0$ its weak limit up to a subsequence in $H^1_0(\Om)$. By \eqref{eq:der2} and \eqref{eq:DLu} we have $\vp_0 \in E_{k,0}$. For $m \ge 1$ we let 
$$ \alpha_m = \Vert \vp_{m} - \Pi_{k,0}(\vp_{m}) \Vert_{H^1_0} + t_m +  |\mu_{k,t_m} - \mkl(u)| $$
and 
$$ \Psi_{t_m} = \frac{\vp_{m} - \Pi_{k,0}(\vp_{m})}{\alpha_m}. $$
By \eqref{vpk} the function $\Psi_{t_m}$ satisfies
$$ \begin{aligned}
\big( - \Delta - \lambda \big) \Psi_{t_m}&  = \mkl(u) u^{2^*-2} \Psi_{t_m} + \frac{\mu_{k,t_m} - \mkl(u)}{\alpha_m} u^{2^*-2} \vp_{m} \\
&+ \mu_{k, t_m} \frac{u_{t_m}^{2^*-2} - u^{2^*-2}}{\alpha_m}\vp_{m}.
\end{aligned}$$ 
By definition of $\alpha_m$ the sequence $(\Psi_{t_m})_{m\ge1}$ is bounded in $H^1_0(\Om)$ and hence, up to a subsequence $\Psi_{t_m} \rightharpoonup \Psi_0$ as $m \to + \infty$ in $H^1_0(\Om)$, where $\Psi_0$ solves weakly in $\Om$
\begin{equation} \label{eq:der10}
 \begin{aligned}
 \big( - \Delta - \lambda \big) \Psi_0 -  \mkl(u) u^{2^*-2} \Psi_{0} & = \beta_0 u^{2^*-2} \vp_0 \\
 &+ (2^*-2) \mkl(u) \beta_1 u^{2^*-2} h \vp_0 ,
 \end{aligned} 
 \end{equation}
where we have let, up to a subsequence, 
 $$\beta_0 = \lim_{m \to +\infty} \frac{\mu_{k,t_m} - \mkl(u)}{\alpha_m} \quad \text{ and } \quad \beta_1 = \lim_{m \to +\infty} \frac{t_m}{\alpha_m}. $$
 We have proven in Step $3$ that $(\vp_m)_{m\ge0}$ strongly converges to $\vp_0$ in $H^1_0(\Om)$, and so $\Vert \vp_0 \Vert_{L^2_u} = 1$. A straightforward adaptation of the argument in Step $3$ applied to $\Psi_{t_m}$ similarly shows that $\Psi_{t_m} \to \Psi_0$ strongly in $H^1_0(\Om)$ as $m \to + \infty$. By construction $\Psi_{t_m} \in E_{k,0}^{\perp_{u}}$, so $\Psi_0 \in E_{k,0}^{\perp_u}$.  Integrating \eqref{eq:der10} against $\vp_0 \in E_{k,0}\backslash \{0\}$ then yields
 $$\beta_0 \int_{\Om} u^{2^*-2} \vp_0^2dx = - (2^*-2) \beta_1 \mkl(u) \int_{\Om} u^{2^*-2} h \vp_0^2dx. $$
 Assume that $\beta_1 = 0$. The latter then shows that $\beta_0 = 0$ and hence that $\Vert \Psi_0\Vert_{H^1_0} = 1$ by definition of $\alpha_m$. But \eqref{eq:der10} then shows that $\Psi_0 \in E_{k,0}$ and thus $\Psi_0 = 0$, a contradiction. Thus $\beta_1 >0$ and
 $$ \begin{aligned}
\lim_{m \to + \infty} \frac{\mu_{k,t_m} - \mkl(u)}{t_m} 
 = \frac{\beta_0}{\beta_1}  & = -(2^*-2)  \mkl(u) \frac{\int_\Om u^{2^*-2}h \vp_0^2dx}{\int_\Om u^{2^*-2} \vp_0^2 dx} \\
 & \ge \inf_{\phi \in E_{k,0}\backslash \{0\}} \Bigg( - (2^*-2) \mkl(u) \frac{\int_\Om u^{2^*-2}h \phi^2dx}{\int_\Om u^{2^*-2} \phi^2 dx}\Bigg).
 \end{aligned} $$
Together with \eqref{eq:der5} this proves \eqref{eq:der9}. This concludes the proof of \eqref{eq:derleft}. The proof of \eqref{eq:derright} is identical when $t \to 0, t<0$ and we omit it here.
\end{proof}

\section{Minimising the principal generalised eigenvalue} \label{vp:minimisation}

\subsection{Minimisation of $\mu_{k,\lambda}(u)$ under a strict inequality}

Let $\lambda \ge 0$ and let $i \ge 0$ such that  $\lambda \in [\Lambda_i, \Lambda_{i+1})$. Throughout this section and until the end of this paper we will always let 
$$ k = N(i)+1 $$
where $N$ is the multiplicity function for $\Sp(-\Delta)$ defined in \eqref{defN}. As we proved in Proposition \ref{prop:signvp} $\mkl(u)$ is, regardless of $u \in  L^{2^*}_{>}(\Om)$, the least positive generalised eigenvalue associated to $u$. We investigate in this section the following variational problem which consists in minimising $\mkl(u)$ over all weights $u$:
\begin{equation} \label{eq:minmuk}
\mkl(\Om) =  \inf_{u \in L^{2^*}_{>}(\Om)} \mkl(u) \Vert u \Vert_{2^*}^{2^*-2}.
\end{equation}
The quantity in the right-hand side of \eqref{eq:minmuk} is invariant if $u$ is scaled by a constant, and it is easy to see that we also have 
\begin{equation*} 
\mkl(\Om) = \inf_{u \in L^{2^*}_{>}(\Om), \Vert u \Vert_{2^*} = 1} \mkl(u). 
\end{equation*}
When $\lambda \in [0, \Lambda_1)$, i.e. when $i=0$, we have $k=1$ and simple arguments show that 
$$ \mu_1^{\lambda}(\Om) = \inf_{v \in H^1_0(\Om)\backslash \{0\}} \frac{\int_{\Om} \big( |\nabla v|^2 - \lambda v^2 \big)dx}{\big(\int_{\Om}|v|^{2^*}dx \big)^{\frac{n-2}{n}}}, 
$$ 
and that both these quantities are simultaneously attained (see Lemma \ref{lemme:muI} below). By the celebrated results in \cite{BN}, minimising $\mu_1^{\lambda}(\Om)$ when $0  \le \lambda \le \Lambda_1$ (when possible) yields a least-energy positive solution of  \eqref{eq:critlambda}. We will prove in this section that, for any $i \ge 1$, minimising $\mkl(\Om)$ with $k = N(i)+1$ similarly yields a least-energy sign-changing solution of \eqref{eq:critlambda}. 

\medskip 

A first property of $\mkl(\Om)$ is the following:
\begin{prop} \label{prop:ineq:large:00}
Let $\Om \subset \R^n$, $n \ge 3$, be bounded and smooth, let $i \ge0$ and let $\lambda \in [\Lambda_i, \Lambda_{i+1})$. Then 
\begin{equation} \label{ineq:mkl:large}
0 < \mkl(\Om) \le K_n^{-2},
\end{equation}
where $K_n$ is as in \eqref{defK0}.
\end{prop}

\begin{proof}
That $\mkl(\Om) >0$  follows from Proposition \ref{prop:signvp}, in particular from \eqref{infmukpos}. We thus only have to prove that $\mkl(\Om) \le K_n^{-2}$. We define 
\begin{equation} \label{defU0}
U_0(x)=\left(1+\frac{|x|^2}{n(n-2)} \right)^{-\frac{n-2}{2}} \quad x \in \R^n.
\end{equation}
It is well-known (see \cite{Aub, Talenti}) that $U_0$ satisfies 
\begin{equation}\label{critUp}
	-\Delta U_0 = \, U_0^{2^*-1} \hbox{ in } \rr^n
\end{equation}
and
\begin{equation} \label{energieU0}
\int_{\R^n} |\nabla U_0|^2dx = \int_{\R^n} U_0^{2^*} dx = K_n^{-n}.
\end{equation}
 In particular $U_0$ attains the infimum in \eqref{defK0}. We let $\vp_1, \dots, \vp_{k-1}$ be eigenfunctions of $-\Delta$ in $H^1_0(\Omega)$, normalised by $\Vert \vp_i \Vert_2 = 1$, respectively associated to the eigenvalues $\lambda_1, \dots, \lambda_{k-1}$. For $1 \le \ell,p,q \le k-1$ we thus have 
 \begin{equation*} 
 \begin{aligned}   
 - \Delta \vp_\ell & =  \lambda_\ell \vp_\ell \quad \text{ in } \Om, \\
  	\int_{\Om} \langle \nabla \vp_p, \nabla \vp_q \rangle dx & = \lambda_p \delta_{pq} \text{ and }  \int_\Om \varphi_{p} \varphi_{q} dx =  \delta_{pq}.
 \end{aligned} 
 \end{equation*}
Let $x_0 \in \Om$ be fixed. For $\ve >0$ we define, for $x \in \Omega$, $B_{\ve}(x)	=  \ve^{- \frac{n-2}{2}} U_0\left( \frac{x-x_0}{\ve}\right) $ and we let $\tilde{u}_\ve \in H^1_0(\Om)$ be the unique solution of 
$$ \left \{ \begin{aligned}
- \Delta \tilde{u}_\ve & = B_{\ve}^{2^*-1} \quad \text{ in } \Om, \\
 \tilde{u}_\ve & = 0 \quad \text{ in } \partial \Om. 
\end{aligned} \right.$$
It is easily seen by the maximum principle that $\tilde{u}_\ve >0$ in $\Om$ and, since $x_0 \in \Om$ is fixed, that $\Vert \tilde{u}_\ve - B_\ve \Vert_{H^1_0} \to 0$ as $\ve \to 0$. We let in what follows
\begin{equation*}
\begin{aligned}
u_\ve = \frac{\tilde{u}_{\ve}}{\Vert \tilde{u}_\ve \Vert_{2^*}},
	\end{aligned} 
	\end{equation*}
so that $u_\ve >0 $ in $\Om$. Straightforward computations show that
\begin{equation} \label{calculs:ineg:large}
\begin{aligned}
& \int_{\Om} |\nabla u_\ve|^2 dx  = K_n^{-n} + o(1),  \quad  \int_{\Om} u_\ve^2 dx =o(1), \\
& \int_{\Om} u_{\ve}^{2^*-1} \varphi_{\ell}\, dx  = o(1),\quad \text{ and } \quad  
 \int_{\Om} \vp_\ell u_\ve dx =   o(1) 
 \end{aligned} 
\end{equation}
as $\ve \to 0$ (see for instance \cite{BN}). We let 
$$ V_\ve = \text{Span} \big( \vp_{1}, \dots, \vp_{k-1}, u_\ve\big), $$
which is a subset of $H^1_0(\Om)$. We first observe that $\dim V_\ve = k$: this follows from \eqref{calculs:ineg:large} and the fact that the Gram determinant of $( \vp_{1}, \dots, \vp_{k-1}, u_\ve)$ for the $H^1_0(\Om)$ scalar product converges to $ \lambda_1 \cdots \lambda_{k-1} K_n^{-2} >0 $ as $\ve \to 0$. As a consequence, $V_\ve$ is an admissible subspace to compute $\mu_k^{\lambda}(u)$ for every $u \in L^{2^*}_{>}(\Om)$. Let $\eta >0$ be fixed. For any $\ve >0$ we let $v_\ve \in V_\ve \backslash \{0\}$ be such that 
$$ Q_{u_\ve + \eta}^{\lambda}(v_\ve) = \max_{v \in V_\ve \backslash \{0\}} Q_{u_\ve + \eta}^{\lambda}(v), $$
where $Q_{u_\ve + \eta}^{\lambda}$ is as in \eqref{defQu}. We let $\alpha_\ve = (\alpha_{1,\ve}, \cdots, \alpha_{k, \ve}) \in \mathbb{S}^{k-1}$ be such that $v_\ve = \sum_{\ell=1}^{k-1} \alpha_{\ell,\ve} \vp_\ell + \alpha_{k,\ve} u_\ve$. Straightforward computations using \eqref{calculs:ineg:large} show that 
\begin{equation} \label{test:est100}
\begin{aligned}
  \int_{\Om} \left( |\nabla v_\ve|^2-\lambda v_\ve^2\right)\, dx & = \sum_{\ell=1}^{k-1} (\lambda_\ell - \lambda) \alpha_{\ell,\ve}^2+ \alpha_{k,\ve}^2 \int_{\Om} |\nabla u_\ve|^2 dx + o(1) \\
\end{aligned} 
\end{equation}
and that 
\begin{equation} \label{test:est2:200}
\begin{aligned}
	\int_{\Om} \big(u_{\ve} + \eta \big)^{2^*-2}v_\ve^2\, dx & \ge \alpha_{k,\ve}^2 
	 + \eta^{2^*-2}  \sum_{\ell=1}^{k-1}\alpha_{\ell,\ve}^2 + o(1)
	\end{aligned} 
\end{equation}
as $\ve \to 0$. Assume first that, up to a subsequence, $\alpha_{k, \ve} \to 0$ as $\ve \to 0$, so that $ \sum_{\ell=1}^{k-1}\alpha_{\ell,\ve}^2 \to 1$. Since $\lambda \ge \lambda_\ell$ for all $1 \le \ell \le k-1$, \eqref{test:est100} and  \eqref{test:est2:200} show that $\limsup_{\ve \to 0} Q_{u_\ve + \eta}^{\lambda}(v_\ve) \le 0$, and hence that $\mkl(\Om)  \le 0$. We may thus assume that, up to a subsequence, $\alpha_{k,\ve} \not \to 0$ as $\ve \to 0$. Equation  \eqref{test:est2:200}  then shows that 
$$\int_{\Om} \big(u_{\ve} + \eta \big)^{2^*-2}v_\ve^2\, dx  \ge (1+o(1)) \alpha_{k,\ve}^2,$$ 
so that combining the latter with \eqref{test:est100}, using again \eqref{calculs:ineg:large} and since $\lambda \ge \lambda_\ell$ for all $1 \le \ell \le k-1$, we have 
 $$ Q_{u_\ve + \eta}^{\lambda}(v_\ve)  \le K_n^{-2} + o(1) $$
as $\ve \to 0$. Since $\Vert u_\ve + \eta \Vert_{2^*} \le  1 + \eta |\Om|^{\frac{n-2}{2n}}$ we obtain in the end that 
$$\mkl(\Om) \le  \mkl(u_\ve + \eta) \Vert u_\ve + \eta \Vert_{2^*}^{2^*-2} \le K_n^{-2}\big(1 + \eta |\Om|^{\frac{n-2}{2n}}\big)^{\frac{4}{n-2}} + o(1) $$
as $\ve \to 0$. Letting first $\ve \to 0$ then $\eta \to 0$ concludes the proof of Proposition \ref{prop:ineq:large:00}.
\end{proof}
We now turn to the main result of this section. It states that \eqref{eq:minmuk} is achieved provided the second inequality in \eqref{ineq:mkl:large} is also strict: 

\begin{theo} \label{prop:minmkl1}
Let $\Om \subset \R^n$, $n \ge 3$, be bounded and smooth. Let $i \ge 0$, $\lambda \in [\Lambda_i, \Lambda_{i+1})$ and let $k = N(i)+1 $. Assume that 
 \begin{equation} \label{eq:test1}
\mkl(\Om) < K_n^{-2} 
\end{equation}
holds. Then $\mkl(\Om)$ is attained, ie there exists $u \in L^{2^*}_{>}(\Om) $ with  $\Vert u \Vert_{2^*} = 1$ such that $\mkl(\Om) = \mkl(u) >0$. Furthermore:
\begin{enumerate}
\item $\mkl(u)$ is simple, i.e. $\dim E_{k}^{\lambda}(u)= 1$. In particular there is a unique generalised eigenvector $\vp = \vp_k^{\lambda}(u)$ associated to $u$ by Proposition \ref{prop:vp} and normalised by $\int_{\Om} u^{2^*-2} \vp^2dx = 1$, which changes sign if $i \ge 1$.
\item  We have $u = |\vp|$. As a consequence, $\psi:= {\mu_{k}^{\lambda}(\Om)}^{\frac{n-2}{4}} \vp$ satisfies $\int_{\Om} |\psi|^{2^*}dx = \mu_{k}^{\lambda}(\Om)^{\frac{n}{2}}$ and is a non-zero solution of \eqref{eq:critlambda}:
$$ - \Delta \psi - \lambda \psi =  |\psi|^{2^*-2}\psi   \quad \text{ in } \Om . $$
Also, $\psi \in  C^{3,\alpha}(\overline{\Om})$ for some $0< \alpha < 1$, and if $i \ge 1$ $\psi$ changes sign.
\item The solution $\psi$ constructed in (2) is a least-energy solution of \eqref{eq:critlambda}, in the sense that 
$$ \int_{\Om} |\psi|^{2^*} dx = \inf_{w \in \mathcal{S}_\lambda\backslash \{0\}}  \int_{\Om} |w|^{2^*} dx $$
where $\mathcal{S}_{\lambda}$ is the set of all $H^1_0(\Om)$ solutions of \eqref{eq:critlambda}. 
\item $\psi$ has at most $k = N(i)+1$ nodal domains. If $i=1$, $\psi$ has exactly two nodal domains and if $i=0$ it can be assumed to be positive in $\Om$. 
\end{enumerate}
\end{theo}
Theorem \ref{prop:minmkl1} shows that the minimisation problem \eqref{eq:minmuk} is the direct counterpart, in the case $\lambda \ge \Lambda_1$, of the celebrated Br\'ezis-Nirenberg \cite{BN} minimisation problem  in the coercive case $\lambda \in [0, \Lambda_1$). When $i \ge 1$ Theorem \ref{prop:minmkl1} shows that achieving $\mkl(\Omega)$ yields a non-zero solution of \eqref{eq:critlambda} of least energy which changes sign. A crucial feature of Theorem \ref{prop:minmkl1} is that if $u \in L^{2^*}_{>}(\Om)$ attains $\mkl(\Omega)$ then $E_{k}^{\lambda}(u)$ is one-dimensional. We prove this in the course of the proof and we do not assume it \emph{a priori}. We also prove that $\dim E_k^{\lambda}(u) = 1$ implies the fundamental relation 
\begin{equation} \label{eq:EL}
 u = |\vp|, 
 \end{equation}
where $\vp$ is, up to $\pm1$, the unique normalised eigenvector in $E_{k}^{\lambda}(u)$. Relation \eqref{eq:EL} has to be understood as the Euler-Lagrange equation associated to the non-differentiable minimisation problem \eqref{eq:minmuk}. That $\dim E_k^{\lambda}(u) = 1$ follows from the uniform spectral gap of Proposition \ref{prop:signvp}. For this reason Theorem \ref{prop:minmkl1} only applies to the \emph{principal} generalised eigenvalue $\mkl(u)$, with $k = N(i)+1$. The problem of minimising $u \mapsto \mu_p^{\lambda}(u) \Vert u \Vert_{2^*}^{2^*-2}$ for some $p > k$ can also be considered, and Euler-Lagrange equations in this setting are in general more complicated: see \cite{HumbertPetridesPremoselli}. 

\medskip

We briefly explain the strategy of proof of Theorem \ref{prop:minmkl1}. Unlike in the coercive case $\lambda \in [0, \Lambda_1)$, and due to the variational characterisation of generalised eigenvalues, a control on $\mkl(\Omega)$ does not immediately ensure a strong $L^{2^*}(\Om)$ control on a sequence of minimising weights. To prove Theorem \ref{prop:minmkl1} we therefore introduce a family of relaxations of \eqref{eq:minmuk}: we construct minimising weights for each perturbed variational problem and prove that these weights, together with their associated families of generalised eigenfunctions, converge strongly enough. The relaxation is as follows: for any $\ve >0$ we let 
$$ X_\ve = \Big \{ u \in L^{2^*}_{>}(\Om) \text{ such that } \int_{\Om} u^{-\ve}dx < + \infty \Big \}. $$
For $u \in X_\ve$ we introduce the following functional:
\begin{equation} \label{Fkeps}
\Fke(u) = \mkl(u) \Vert u \Vert_{2^*}^{2^*-2} + \ve  \Big(\int_{\Om} u^{-\ve}dx \Big)\Vert u \Vert_{2^*}^{2^*-2}
\end{equation}
and consider the associated minimisation problem:
\begin{equation} \label{mukeps}
\mkle = \inf_{u \in X_\ve, \Vert u \Vert_{2^*}=1} \Fke(u). 
\end{equation}
In Lemma \ref{lemme2} below we first prove that, under the assumptions of Theorem  \ref{prop:minmkl1}, $\mkle$ is attained, i.e. there is $u_\ve \in X_\ve$ with $\Vert u_\ve\Vert_{2^*} = 1$ such that $F_{k,\ve}(u_\ve) = \mkle$. The minimality of $u_\ve$ then allows us to show, in Lemma \ref{lemme4} below, that the eigenspace $E_k^{\lambda}(u_\ve)$ associated to $\mkl(u_\ve)$ is one-dimensional, which we then use in Lemma \ref{lemme4bis} to obtain an Euler-Lagrange equation for $u_\ve$. This equation relates $u_\ve$ and its normalised eigenvector  $\vp_{k,\ve}\in E_k^{\lambda}(u_\ve)$ (see \eqref{eq:egalpp} below) and although it is more involved than \eqref{eq:EL} it allows us to prove that $u_\ve$ converges towards a minimiser of $\mkl(\Om)$ as $\ve \to 0$. Assumption \eqref{eq:test1} is used at every step in the proof, both to construct $u_\ve$ and to prove the compactness of families of eigenfunctions. The functional $\Fke$ in \eqref{Fkeps} is an adaptation of the relaxation introduced in \cite{GurskyPerez}. The correction term $\ve \int_{\Om} u^{-\ve}dx$ has two main advantages: it breaks the scale invariance of problem \eqref{eq:minmuk} and ensures that minimisers of $F_{k,\ve}$ are in $L^{2^*}_{>}(\Om)$ (which is not closed under strong convergence in $L^{2^*}(\Om)$). Knowing that $u_\ve \in L^{2^*}_{>}(\Om)$ is crucial to apply Proposition \ref{prop:dervp} and to obtain the Euler-Lagrange equation \eqref{eq:egalpp} for $u_\ve$. Our analysis here is more involved than the one in  \cite{GurskyPerez} since we consider positive eigenvalues and our eigenvalues equations are thus energy-critical with focusing sign. Observe that as soon as $k \ge 2$ (ie $i \ge 1$) it is highly unlikely that \emph{every} minimising sequence for eigenvalue-optimisation problems like $\mkl(\Omega)$ converges, and this is why we have to define a relaxation of \eqref{eq:minmuk}. For problem \eqref{eq:minmuk} this explains why the non-coercive case $\lambda \ge \Lambda_1$ is considerably more difficult than the coercive case. Let us finally mention that other relaxations of $\mkl(\Om)$ may be considered: we refer for instance to \cite{HumbertPetridesPremoselli} where energy-subcritical relaxations are considered.

\subsection{Minimisers of the relaxed problem} \label{minrelaxed}

We first prove that \eqref{eq:test1} remains true for $\Fke$ for $\ve$ small enough:

\begin{lemme} \label{lemme1}
We have:
$$ \lim_{\ve \to 0} \mkle = \mkl(\Om). $$
As a consequence, for $\ve >0$ small enough, we have 
\begin{equation} \label{eq:test1e}
\mkle < K_n^{-2}.
\end{equation}
\end{lemme}

\begin{proof}
Let $\eta >0$ and $u_0\in C^\infty(\overline{\Om}), u_0 >0$ in $\overline{\Om}$, with $\Vert u_0 \Vert_{2^*} = 1$ and $\mkl(u_0) \le \mkl(\Om) + \eta$. By dominated convergence we have $\int_{\Om} u_0^{-\ve}dx \to |\Om|$ as $\ve \to 0$. Thus 
$$\mkle \le  \Fke(u_0) \le  \mkl(\Om) + \eta + o(1) $$
as $\ve \to 0$, which then gives $ \limsup_{\ve \to 0} \mkle \le \mkl(\Om) + \eta$. Letting $\eta \to 0$ proves that $ \limsup_{\ve \to 0} \mkle \le \mkl(\Om) $. For the reverse inequality, remark that, by \eqref{Fkeps}, $\Fke(u) \ge \mkl(\Om)$ for any $u  \in X_\ve$. Taking the infimum over $u \in X_\ve$ and then the $\liminf$ as $\ve \to 0$ yields $\liminf_{\ve \to 0} \mkle \ge \mkl(\Om)$, which proves the first part of the Lemma. Estimate \eqref{eq:test1e} then follows from \eqref{eq:test1}.
\end{proof}

We then prove that $\mkle$ is attained:

\begin{lemme} \label{lemme2}
Let $\ve_0 >0$ be such that, for every $0 < \ve \le \ve_0$, \eqref{eq:test1e} is satisfied. Then, for all $0 < \ve \le \ve_0$, $\mkle$ is attained at some function $u_\ve \in X_\ve $ with $\Vert u_\ve \Vert_{2^*} = 1$.   
\end{lemme}

\begin{proof}
We assume that $\ve>0$ is fixed small enough so that \eqref{eq:test1e} is satisfied. We let $(u_m)_{m\ge 1}$ be a minimising sequence for $\mkle$: $u_m \in X_\ve$, $\Vert u_m \Vert_{2^*} =1$ and 
\begin{equation} \label{eq:cve1}
\Fke(u_m) = \mkl(u_m)  + \ve \int_{\Om} u_m^{-\ve}dx  \to \mkle
\end{equation}
 as $m \to + \infty$. For every $m \ge 1$, $u_m \in L^{2^*}_>(\Om)$, so by Proposition \ref{prop:vp} a generalised eigenfunction $\vkl(u_m)$ exists. We will let for simplicity in what follows $\vp_{k,m} = \vkl(u_m)$ and $ \mu_{k,m} = \mkl(u_m)$. This function satisfies $\int_{\Om} u_m^{2^*-2} \vp_{k,m}^2dx = 1$  and 
\begin{equation} \label{eq:cve2}
 - \Delta \vp_{k,m} - \lambda \vp_{k,m} = \mu_{k,m} u_m^{2^*-2} \vp_{k,m} \quad \text{ in } \Om. 
 \end{equation}
Let in what follows $w_m = u_m^{2^*-2}$. Then $(w_m)_{m \ge1}$ is uniformly bounded in $L^{\frac{n}{2}}(\Om)$ and there exists $w_\ve \in L^{\frac{n}{2}}(\Om)$, $w_\ve \ge 0$ a.e., such that $w_m \rightharpoonup w_\ve$ weakly in $L^{\frac{n}{2}}(\Om)$, up to a subsequence, as $m \to + \infty$. Let $u_\ve = w_\ve^{\frac{1}{2^*-2}} \ge0$ a.e. By lower semi-continuity of the $L^{\frac{n}{2}}(\Om)$ norm we have $ \int_{\Om} u_\ve^{2^*}dx \le 1$. Since $t \mapsto t^{-\frac{\ve}{2^*-2}}$ is convex in $(0, + \infty)$, $u \mapsto \int_{\Om} u^{-\frac{\ve}{2^*-2}} dx$ is weakly lower semi-continuous over $\{ w \in L^{\frac{n}{2}}(\Om), w \ge 0 \text{ a.e.} \}$. As a consequence of the weak convergence and of \eqref{eq:cve1} we get
\begin{equation} \label{eq:cve3}
\begin{aligned}
 \int_{\Om} u_\ve^{-\ve}dx = \int_{\Om} {w_\ve}^{- \frac{\ve}{2^*-2}} &\le \liminf_{m \to + \infty}  \int_{\Om} w_m^{-\frac{\ve}{2^*-2}}dx \\
 &=  \liminf_{m \to + \infty}  \int_{\Om} u_m^{-\ve}dx  \le \frac{1}{\ve} \mkle < + \infty. 
\end{aligned} 
 \end{equation}
This shows that $u_\ve \in X_\ve$ and hence that $u_\ve \neq 0$. By \eqref{eq:cve1} the sequence $(\mu_{k,m})_{m \ge1}$ is bounded. We claim that there exists $C >0$ such that 
\begin{equation} \label{eq:bornevpkm}
\Vert \vp_{k,m} \Vert_{H^1_0} \le C
\end{equation}
for all $m \ge 1$. First, if $\lambda \in (\Lambda_i, \Lambda_{i+1})$, that is if $\lambda \not \in \Sp(-\Delta)$, $-\Delta- \lambda$ is invertible and \eqref{eq:bornevpkm} follows from \eqref{eq:cve2} since the right-hand side of \eqref{eq:cve2} is uniformly bounded in $L^{\frac{2n}{n+2}}(\Om)$ (this is the same argument that we used in \eqref{holder:spectre}). We may thus assume that $\lambda = \Lambda_i$. We proceed as in the proof of Proposition \ref{prop:signvp}: We split $\vp_{k,m}$ as 
$$ \vp_{k,m} = v_m + z_m, $$
where $z_m \in E_{\Lambda_i}(\Om)$ and $v_m$ is $H^1_0(\Om)$-orthogonal to $E_{\Lambda_i}(\Om)$. Since $-\Delta z_m - \Lambda_i z_m = 0$, by \eqref{eq:cve2} $v_m$ satisfies 
\begin{equation*} 
 - \Delta v_m - \Lambda_i v_m = \mu_{k,m} u_m^{2^*-2} \vp_{k,m} \quad \text{ in } \Om. 
 \end{equation*}
By definition $- \Delta-\Lambda_i$ is invertible in $E_{\Lambda_i}(\Om)^{\perp}$ and the same argument as above shows that $\Vert v_m \Vert_{H^1_0} \le C $ for some $C>0$ and for all $m \ge 0$. To prove \eqref{eq:bornevpkm} we thus proceed by contradiction and assume that $\Vert z_m \Vert_{H^1_0} \to + \infty$ as $m \to + \infty$, and we let $\tilde{\vp}_{k,m} = \frac{\vp_{k,m}}{\Vert z_m\Vert_{H^1_0}}$. Since $(v_m)_{m \ge0}$ is bounded in $H^1_0(\Om)$ we have 
$$ \tilde{\vp}_{k,m} = \frac{z_m}{\Vert z_m\Vert_{H^1_0}} + \frac{v_m}{\Vert z_m\Vert_{H^1_0}} = \frac{z_m}{\Vert z_m\Vert_{H^1_0}}  + o(1) \quad \text{ in } H^1_0(\Om). $$
The function  $ \tilde{z}_m:=\frac{z_m}{\Vert z_m\Vert_{H^1_0}}$ is bounded in $H^1_0(\Om)$ and thus weakly converges to some $Z \in H^1_0(\Om)$. Since $\tilde{z}_m \in E_{\Lambda_i}(\Om)$ for every $m \ge 0$ and since $E_{\Lambda_i}(\Om)$ is finite-dimensional, $\tilde{z}_m$ strongly converges to $Z$ in $H^1_0(\Om)$, and we have $\Vert Z \Vert_{H^1_0} =1$ and $Z \in E_{\Lambda_i}(\Om)$. Thus $\tilde{\vp}_{k,m}$ strongly converges to $Z$ in $H^1_0(\Om)$ as $m \to + \infty$. The normalisation condition $\int_{\Om} u_m^{2^*-2} \vp_{k,m}^2dx = 1$ now becomes:
\begin{equation} \label{eq:bornevpkm:2}
 \frac{1}{\Vert z_m \Vert_{H^1_0}^2} =  \int_{\Om} u_m^{2^*-2} \tilde{\vp}_{k,m}^2dx = \int_{\Om} u_m^{2^*-2} Z^2dx + o(1). 
 \end{equation}
Since $u_m^{2^*-2} \rightharpoonup u_{\ve}^{2^*-2}$ in $L^{\frac{n}{2}}(\Om)$ we can pass \eqref{eq:bornevpkm:2} to the limit  and we obtain 
$$ \int_{\Om} u_\ve^{2^*-2} Z^2dx = 0. $$
Since $u_\ve \in X_\ve$ we have in particular $u_\ve \in L^{2^*}_{>}(\Om)$, which implies that $Z = 0$ a.e. in $\Om$. But since $Z \in E_{\Lambda_i}(\Om)$ standard elliptic theory shows that it is smooth, and hence $Z = 0$, which contradict $\Vert Z \Vert_{H^1_0} = 1$. This proves \eqref{eq:bornevpkm}.

We denote by $\vp_{k,\ve}$ the weak limit in $H^1_0(\Om)$, up to a subsequence, of $( \vp_{k,m} )_{m \ge 1}$, and we let $\hat{\mu}_{k,\ve} = \lim_{m \to +\infty} \mu_{k,m}$. Since $\vp_{k,m}$ strongly converges to $\vp_{k,\ve}$ in $L^\frac{n}{n-2}(\Om)$ by Sobolev's embedding and since $w_m$ weakly converges to $w_\ve = u_\ve^{2^*-2}$ in $L^{\frac{n}{2}}(\Om)$ we get that $\vp_{k,\ve}$ weakly satisfies
 \begin{equation} \label{eq:cve4}
  - \Delta \vp_{k,\ve} - \lambda \vp_{k,\ve} = \hat{\mu}_{k,\ve}  u_\ve^{2^*-2}\vp_{k,\ve} \quad \text{ in } \Om. 
 \end{equation}
We claim that $\vp_{k,\ve} \neq 0$. Assume otherwise: then $\vp_{k,m} \to 0$ in $L^2(\Om)$ up to a subsequence as $m \to + \infty$. Integrating \eqref{eq:cve2} against $\vp_{k,m}$ and using H\"older and Sobolev inequalities yields, since $\Vert u_m \Vert_{2^*} = 1$,
$$ \begin{aligned} \int_{\Om} |\nabla \vp_{k,m}|^2dx + o(1) & \le \mu_{k,m} \big( \int_{\Om} |\vp_{k,m}|^{2^*}dx \big)^{\frac{n-2}{n}}  \le \mu_{k,m}K_n^2  \int_{\Om} |\nabla \vp_{k,m}|^2dx, \end{aligned} $$
where $K_n$ is the optimal constant in Sobolev's inequality given by \eqref{defK0}. By \eqref{eq:test1e} we have $\limsup_{m \to  + \infty} \mu_{k,m} < K_n^{-2}$, so the latter yields $\int_{\Om} |\nabla \vp_{k,m}|^2 dx = o(1)$. But this is a contradiction with $\int_{\Om} u_m^{2^*-2} \vp_{k,m}^2dx = 1$, hence $\vp_{k,\ve} \neq 0$.

We now remark that $\hat{\mu}_{k,\ve} >0$. Indeed, we have $\hat{\mu}_{k,\ve} =  \lim_{m \to + \infty} \mu_{k,m}$ (up to a subsequence), and $\mu_{k,m} = \mkl(u_m) \ge \mkl(\Om) >0$, where we used Proposition \ref{prop:signvp}. By Proposition \ref{prop:signvp}, and since $\lambda \in [\Lambda_i, \Lambda_{i+1})$, the $N(i)$ first generalised eigenvalues $\mu_1^{\lambda}(u_\ve), \dots, \mu_{N(i)}^{\lambda}(u_\ve)$ are nonpositive. Since $\hat{\mu}_{k,\ve} >0$, \eqref{eq:cve4} implies that 
$$ \int_{\Om} u_\ve^{2^*-2} \vp_{k,\ve} \vp_p^{\lambda}(u_\ve)dx = 0 \quad \text{ for all } 1 \le p \le k-1. $$
Thus $V =  \text{Span} \big( \vp_1^{\lambda}(u_\ve), \dots, \vp_{k-1}^\lambda(u_\ve), \vp_{k,\ve}\big) $ satisfies $\dim_{u_\ve} V = k$ and is an admissible subspace in the definition of $\mkl(u_\ve)$: \eqref{muk} and \eqref{eq:cve4} then show that 
\begin{equation} \label{eq:cve5}
\mkl(u_\ve) \le \hat{\mu}_{k,\ve}.
\end{equation}
 It remains to let $ \tilde{u}_\ve = \frac{u_\ve}{\Vert u_\ve \Vert_{2^*}}$, so that $\Vert \tilde{u}_\ve \Vert_{2^*} = 1$. Since $\Vert u_\ve \Vert_{2^*} \le 1$ we have, using \eqref{eq:cve1},  \eqref{eq:cve3}, \eqref{eq:cve5} for $\ve$ small enough:
$$ \begin{aligned}
\Fke(\tilde{u}_\ve) & =  \mkl(\tilde{u}_\ve) + \ve  \Big(\int_{\Om} \tilde{u}_\ve^{-\ve}dx \Big)\\
& = \mkl(u_\ve) \Vert u_\ve \Vert_{2^*}^{2^*-2} + \ve \Vert u_\ve \Vert_{2^*}^{\ve} \int_{\Om}u_\ve^{-\ve}dx \\
& \le \Vert u_\ve \Vert_{2^*}^{\ve}\Big( \hat{\mu}_{k,\ve} + \ve \int_{\Om} u_\ve^{-\ve}dx \Big) \\
& \le \Vert u_\ve \Vert_{2^*}^{\ve}\liminf_{m \to + \infty} \Big(  \mkl(u_m) + \ve \int_{\Om} u_m^{-\ve}dx \Big)  = \Vert u_\ve \Vert_{2^*}^{\ve} \mkle.
\end{aligned} $$
By \eqref{mukeps} this shows at once that $\Vert u_\ve \Vert_{2^*} = 1$ and that $u_\ve$ attains $\mkle$. 
\end{proof}

\begin{lemme} \label{lemme4}
We have $\dim_{u_\ve} E_k^\lambda(u_\ve) = 1$. 
\end{lemme}

\begin{proof}
Let $h \in L^\infty(\Om)$. For $|t| < \Vert h \Vert_{\infty}^{-1}$ we let $u_{\ve,t} = (1+th)u_\ve$ and 
$$ \tilde{u}_{\ve,t} = \frac{(1+th)u_\ve}{\Vert (1+th)u_\ve\Vert_{2^*}}.$$
First, direct computations show that 
$$ \frac{d}{dt}_{|t = 0} \Vert u_{\ve,t} \Vert_{2^*} = \int_{\Om} u_\ve^{2^*}h dx \quad \text{ and } \quad  \frac{d}{dt}_{|t = 0} \int_{\Om} u_{\ve,t}^{-\ve}dx =- \ve \int_{\Om}u_\ve^{-\ve}h dx. $$
We now apply Proposition \ref{prop:dervp}: with the latter we obtain that left and right derivatives of $t \mapsto \Fke(\tilde{u}_{\ve,t})$ at $t=0$ exist and are given by 
\begin{equation} \label{eq:derlefte}
\begin{aligned}
& \frac{d}{dt}_{|t=0_+} \Fke(\tilde{u}_{\ve,t}) = \inf_{\vp \in E_k^{\lambda}(u_\ve)\backslash \{0\}} \Bigg( - (2^*-2) \mkl(u_\ve) \frac{\int_\Om u_\ve^{2^*-2}h \vp^2dx}{\int_\Om u_\ve^{2^*-2} \vp^2 dx}\Bigg) \\
& + \ve^2 \Big[ \Big(\int_\Om u_\ve^{-\ve}dx\Big)\Big( \int_{\Om} u_\ve^{2^*}h dx\Big) - \int_{\Om} u_\ve^{-\ve}h dx\Big] + (2^*-2)\mkl(u_\ve)\int_{\Om} u_\ve^{2^*}h dx
\end{aligned}
\end{equation}
and
\begin{equation} \label{eq:derrighte}
\begin{aligned}
& \frac{d}{dt}_{|t=0_-} \Fke(\tilde{u}_{\ve,t}) = \sup_{\vp \in E_k^{\lambda}(u_\ve)\backslash \{0\}} \Bigg( - (2^*-2) \mkl(u_\ve) \frac{\int_\Om u_\ve^{2^*-2}h \vp^2dx}{\int_\Om u_\ve^{2^*-2} \vp^2 dx}\Bigg) \\
& + \ve^2 \Big[ \Big(\int_\Om u_\ve^{-\ve}dx\Big) \Big(\int_{\Om} u_\ve^{2^*}h dx\Big) - \int_{\Om} u_\ve^{-\ve}h dx\Big] + (2^*-2)\mkl(u_\ve)\int_{\Om} u_\ve^{2^*}h dx
\end{aligned}
\end{equation}
Since $u_\ve$ minimizes $F_{k,\ve}$ we have $ \Fke(\tilde{u}_{\ve,t}) \ge \Fke(u_\ve)$ for all $t$ small enough and all $h \in L^\infty(\Omega)$, which implies that
\begin{equation} \label{eq:extremale}
  \frac{d}{dt}_{|t=0_+} \Fke(\tilde{u}_{\ve,t})  \ge 0 \ge  \frac{d}{dt}_{|t=0_-} \Fke(\tilde{u}_{\ve,t}). 
  \end{equation}
Using \eqref{eq:derlefte} and \eqref{eq:derrighte} we then obtain that 
\begin{equation} \label{eq:dim1}
\begin{aligned}
&  \sup_{\vp \in E_k^{\lambda}(u_\ve)\backslash \{0\}} \Bigg( - (2^*-2) \mkl(u_\ve) \frac{\int_\Om u_\ve^{2^*-2}h \vp^2dx}{\int_\Om u_\ve^{2^*-2} \vp^2 dx}\Bigg) \\
  & \quad\le  \inf_{\vp \in E_k^{\lambda}(u_\ve)\backslash \{0\}} \Bigg( - (2^*-2) \mkl(u_\ve) \frac{\int_\Om u_\ve^{2^*-2}h \vp^2dx}{\int_\Om u_\ve^{2^*-2} \vp^2 dx}\Bigg) 
  \end{aligned} 
  \end{equation}
for any $h \in L^\infty(\Om)$. We claim that \eqref{eq:dim1} shows that $\dim_{u_\ve} E_k^{\lambda}(u_\ve) = 1$. By Proposition \ref{prop:vp} a non-zero generalised eigenvector $\psi_1$ satisfying $\int_\Om u_\ve^{2^*-2} \psi_1^2dx = 1$ exists. Assume that there exists $\psi_2 \in E_k^{\lambda}(u_\ve)$ with $\int_\Om u_\ve^{2^*-2} \psi_2^2dx = 1$ and $\int_{\Om} u_\ve^{2^*-2} \psi_1 \psi_2 dx = 0$. From  \eqref{eq:dim1} we get that for any $h \in L^\infty(\Om)$ there exists $\theta_h \in \R$ such that for any $\vp \in E_k^{\lambda}(u_\ve)\backslash \{0\}$, 
$$\frac{\int_\Om u_\ve^{2^*-2}h \vp^2dx}{\int_\Om u_\ve^{2^*-2} \vp^2 dx} = \theta_h \quad \text{ holds}.$$
Choosing first $\vp = \psi_1$ and $\vp = \psi_2$ gives that $\int_{\Om} u_\ve^{2^*-2} h \psi_i^2dx = \theta_h$ for $i=1,2$, and 
$$ \int_{\Om} \big( u_\ve^{2^*-2} \psi_1^2 -u_\ve^{2^*-2} \psi_2^2 \big)hdx = 0 \quad \text{ for all } h \in L^\infty(\Om),$$
and hence  $|\psi_1| = |\psi_2|$ a.e. in $\Om$ since $u_\ve \in X_\ve$. Choosing $\vp = \psi_1 + \psi_2$ shows that 
$$ \int_{\Om} u_\ve^{2^*-2} \psi_1 \psi_2 h dx = 0 \quad \text{ for all } h \in L^\infty(\Om), $$ 
and hence $ \psi_1 \psi_2 = 0$ a.e. in $\Om$. Thus $\psi_1 = \psi_2 = 0$ a.e., a contradiction. Hence $\dim_{u_\ve} E_k^{\lambda}(u_\ve) = 1$, which implies that $\dim E_k^{\lambda}(u_\ve) = 1$ since $u_\ve \in L^{2^*}_{>}(\Om)$. 
\end{proof}

\subsection{Convergence to a minimiser of \eqref{eq:minmuk} and proof of Theorem \ref{prop:minmkl1}} \label{minconvergence}

We let, for $\ve >0$ small enough,
\ben  \label{eq:gammas} \gamma_{1,\ve}=1 + \frac{\ve^2}{(2^*-2) \mkl(u_\ve)} \int_{\Om} u_\ve^{-\ve}dx \quad \text{ and } \gamma_{2,\ve} = \frac{\ve^2}{(2^*-2) \mkl(u_\ve)}. 
\een
By \eqref{Fkeps} and \eqref{eq:test1e}, and since $u_\ve$ attains $\mkle$, we have $ \ve \int_\Om u_\ve^{-\ve} dx \le K_n^{-2}$. Independently, by definition of $\mkl(\Om)$ in \eqref{eq:minmuk} we have $\mkl(u_\ve) \ge \mkl(\Om)$, so that by \eqref{infmukpos} we have $\liminf_{\ve\to0} \mkl(u_\ve) >0$. This shows that 
\begin{equation*}
\gamma_{1,\ve} = 1 + O(\ve) \quad \text{ and } \quad \gamma_{2,\ve} = O(\ve^2) 
\end{equation*}
 as $\ve \to 0$. As in the proof of Lemma \ref{lemme2} we denote by $\vp_{k,\ve}$ a non-zero generator of $E_k^{\lambda}(u_\ve)$, normalised by $ \int_{\Om}u_\ve^{2^*-2} \vp_{k,\ve}^2dx = 1$. It satisfies, for $\ve>0$ small enough :
\begin{equation} \label{eq:vpke}
 - \Delta \vp_{k,\ve} - \lambda \vp_{k,\ve} = \mkl(u_\ve) u_\ve^{2^*-2} \vp_{k,\ve} \quad \text{ in } \Om. 
 \end{equation}
Since $\Vert u_\ve \Vert_{2^*} = 1$ and $u_\ve$ attains $\mkle$, there exists $C>0$ such that 
\begin{equation} \label{controle:muke}
\mkl(u_\ve) \le C 
\end{equation}
for all $\ve >0$ small enough, by  \eqref{Fkeps} and \eqref{eq:test1e}. We first obtain an Euler-Lagrange equation relating $u_\ve$ and $\vp_{k,\ve}$:

\begin{lemme} \label{lemme4bis}
We have 
 \begin{equation} \label{eq:egalpp}
  u_\ve^2 - \frac{\gamma_{2,\ve}}{\gamma_{1,\ve}} \frac{1}{u_\ve^{2^*-2+\ve}} = \frac{1}{\gamma_{1,\ve}} \vp_{k,\ve}^2 
  \end{equation}
a.e. in $\Om$, where $\gamma_{1,\ve}$ and $\gamma_{2,\ve}$ are given by \eqref{eq:gammas}. 
\end{lemme}
\begin{proof} 
Since $E_k^{\lambda}(u_\ve)$ is one-dimensional by Lemma \ref{lemme4}, \eqref{eq:dim1} is actually an equality. Therefore, by \eqref{eq:derlefte} and \eqref{eq:derrighte}, the minimality condition \eqref{eq:extremale} becomes
$$   \frac{d}{dt}_{|t=0_+} \Fke(\tilde{u}_{\ve,t})  = \frac{d}{dt}_{|t=0_-} \Fke(\tilde{u}_{\ve,t}) = 0$$
for any $h \in L^\infty(\Om)$, which simplifies into 
$$ \begin{aligned}
\int_{\Om} h \Bigg[ &\Big( (2^*-2) \mkl(u_\ve) + \ve^2 \int_{\Om} u_\ve^{-\ve} dx \Big) u_\ve^{2^*} \\
 &- \Big(\ve^2 u_\ve^{-\ve} + (2^*-2) \mkl(u_\ve) u_\ve^{2^*-2} \vp_{k,\ve}^2 \Big)\Bigg]dx = 0
 \end{aligned} $$
 for any $h \in L^\infty(\Om)$. This implies, with \eqref{eq:gammas}, that
 $$u_\ve^{2^*} - \frac{\gamma_{2,\ve}}{\gamma_{1,\ve}} u_\ve^{-\ve} = \frac{1}{\gamma_{1,\ve}} u_\ve^{2^*-2} \vp_{k,\ve}^2 $$
 a.e. in $\Om$. Since $u_\ve \in X_\ve$, we can divide by $u_\ve^{2^*-2}$ a.e. and we get \eqref{eq:egalpp}. 
\end{proof}
 
Since $\gamma_{i,\ve} >0$ for $i=1,2$, a first consequence of \eqref{eq:egalpp} is that $|\vp_{k,\ve}| \le \gamma_{1,\ve}^{\frac12} u_\ve$, and hence $|\vp_{k,\ve}| \le 2 u_\ve$ a.e. in $\Om$ for $\ve$ small enough. As a consequence, $\Vert \vp_{k,\ve} \Vert_{2^*} \le 2$. Using \eqref{controle:muke}, integrating \eqref{eq:vpke} against $\vp_{k,\ve}$ and using H\"older's inequality then shows that 
\begin{equation} \label{eq:bornitude:vpke}
\Vert \vp_{k,\ve} \Vert_{H^1_0} \le C
\end{equation}
for some fixed $C >0$, for all $\ve$ small enough. For $u >0$ we now define
 $$G_\ve(u) = u - \frac{\gamma_{2,\ve}}{\gamma_{1,\ve}}  u^{- \frac{2^*-2+\ve}{2}}. $$
 Since $\gamma_{i,\ve} >0$ for $i=1,2$ it is easily seen that $G_\ve: (0, + \infty) \to \R$ is smooth and invertible. We let $H_\ve = G_\ve^{-1}$ and we let $v_{\ve, *}$ be the unique solution of $G_\ve(v) = 0$ in $(0, + \infty)$. Simple considerations using \eqref{eq:gammas} show that 
 \begin{equation} \label{eq:vestar}
 v_{\ve,*} = \left( \frac{\gamma_{2,\ve}}{\gamma_{1,\ve}} \right)^{\frac{2}{2^*+\ve}} = \frac{1}{\big((2^*-2) \mkl(u_\ve)\big)^{\frac{n-2}{n}}} \ve^{\frac{2(n-2)}{n}} (1+ o(1)) 
 \end{equation}
  as $\ve \to 0$, that $v \mapsto H_\ve(v)$ converges uniformly to the identity map in compact subsets of $[0, + \infty)$ and that there exists $C>0$ depending only on $n$ such that
 \begin{equation} 
  \label{eq:propG}  
  H_\ve(v) \ge v_{\ve, *} \quad \text{ and } \quad v \le H_\ve(v) \le v +C v_{\ve, *} \quad \text{ for all } v \ge 0.
 \end{equation}
 Equality \eqref{eq:egalpp} shows that
  \begin{equation}  \label{eq:egalG}
   u_\ve = H_\ve \Big( \frac{1}{\gamma_{1,\ve}} \vp_{k,\ve}^2 \Big)^{\frac12} 
  \end{equation} 
a.e. in $\Om$. As a consequence, $u_\ve \ge v_{\ve,*}^{\frac12}$ a.e. in $\Om$, and with \eqref{eq:vestar} we obtain that 
\begin{equation} \label{eq:nodalu}
\int_{\Om} u_\ve^{-\ve}dx = O(1) \quad \text{ as } \ve \to 0.
\end{equation} Coming back to the definition of $\Fke$ in \eqref{Fkeps}, using Lemma \ref{lemme1} and since $u_\ve$ attains $\mkle$ we obtain with \eqref{eq:nodalu} that
\begin{equation} \label{eq:convmkl}
\mkl(u_\ve) \to \mkl(\Om) \quad \text{ as } \ve \to 0. 
 \end{equation}
Using \eqref{eq:propG} we can apply the regularity theory of \cite{BrezisKato} to equation \eqref{eq:vpke}, which shows that for any $\ve >0$, $\vp_{k,\ve} \in C^{3}(\overline{\Om})$. As a consequence of \eqref{eq:egalG}, $u_\ve \in C^0(\overline{\Om})$ and \eqref{eq:egalpp} holds true everywhere in $\overline{\Om}$. By \eqref{eq:bornitude:vpke} we can let $\vp_{k,0}$ be the weak limit of $\vp_{k,\ve}$, up to a subsequence as $\ve \to 0$. Sobolev's embeddings show that $\vp_{k,\ve} \to \vp_{k,0}$ strongly in $L^q(\Om)$ for any $1 \le q < 2^*$, still up to a subsequence as $\ve \to 0$, and  \eqref{eq:propG} and  \eqref{eq:egalG} then show that 
\begin{equation} \label{eq:convue} 
u_\ve \to u_0 =  |\vp_{k,0}|
\end{equation}  strongly in $L^q(\Om)$ for any $1 \le q < 2^*$, up to a subsequence as $\ve \to 0$. We now prove the strong convergence of $\vp_{k,\ve}$ towards $\vp_{k,0}$:
 
\begin{lemme}
Up to a subsequence in $\ve$, $(\vp_{k,\ve})_{\ve >0}$ strongly converges to $\vp_{k,0}$ in $H^1_0(\Om)$ as $\ve \to 0$.
\end{lemme} 
 
\begin{proof}
First, using \eqref{eq:convmkl} and \eqref{eq:convue} together with \eqref{eq:vpke} shows that $\vp_{k,0}$ satisfies weakly 
 \begin{equation} \label{eq:convu0}
   - \Delta \vp_{k,0} - \lambda \vp_{k,0} = \mkl(\Om) |\vp_{k,0}|^{2^*-2} \vp_{k,0} \quad \text{ in } \Om. 
   \end{equation}
The regularity theory  of   \cite{BrezisKato} then shows that $\vp_{k,0} \in C^{3,\alpha}(\overline{\Om})$ for some $\alpha >0$. Substracting \eqref{eq:convu0} to \eqref{eq:vpke} shows that $\vp_{k,\ve}-\vp_{k,0}$ satisfies
 $$ \begin{aligned}
   - \Delta (\vp_{k,\ve} - \vp_{k,0}) &- \lambda (\vp_{k,\ve}-\vp_{k,0})  = \mkl(u_\ve) u_\ve^{2^*-2} (\vp_{k,\ve} - \vp_{k,0})  \\
   & + (\mkl(u_\ve) - \mkl(\Om)) u_\ve^{2^*-2} \vp_{k,0} + \mkl(\Om) (u_\ve^{2^*-2} - u^{2^*-2}) \vp_{k,0}. 
   \end{aligned}$$
   Integrating against $\vp_{k,\ve}-\vp_{k,0}$ and using \eqref{eq:convmkl} shows that 
   $$\begin{aligned}
    \int_{\Om} |\nabla (\vp_{k,\ve} - \vp_{k,0})|^2 dx &= o(1) +  \mkl(\Om) \int_\Om u_\ve^{2^*-2} (\vp_{k,\ve} - \vp_{k,0})^2 dx \\
    &+ \int_{\Om}\mkl(\Om) (u_\ve^{2^*-2} - u^{2^*-2}) \vp_{k,0} (\vp_{k,\ve} - \vp_{k,0})dx. 
    \end{aligned} $$
  Since $\vp_{k,0} \in C^3(\overline{\Om})$, $(u_\ve^{2^*-2} - u^{2^*-2})_{\ve >0}$ is bounded in $L^\frac{n}{2}(\Om)$ and $\vp_{k,\ve} \to \vp_{k,0}$ strongly in $L^{\frac{n}{n-2}}(\Om)$, the last integral in the right-hand side converges to $0$ up to a subsequence as $\ve \to 0$. We then have, by Sobolev's inequality and since $\Vert u_\ve \Vert_{2^*} = 1$,
    $$
\bal    \int_{\Om} |\nabla (\vp_{k,\ve} - \vp_{k,0})|^2 dx & = o(1) +  \mkl(\Om) \int_\Om u_\ve^{2^*-2} (\vp_{k,\ve} - \vp_{k,0})^2 dx \\
& \le o(1) + \mkl(\Om) K_n^2 \int_{\Om} |\nabla (\vp_{k,\ve} - \vp_{k,0})|^2 dx, 
\eal$$
which shows with \eqref{eq:test1} that $\Vert \vp_{k,\ve} - \vp_{k,0}\Vert_{H^1_0} \to 0$ as $\ve \to 0$, up to a subsequence. 
\end{proof}

We are now in position to conclude the proof of Theorem \ref{prop:minmkl1}.

 \begin{proof}[End of the proof of Theorem \ref{prop:minmkl1}]
Since $\vp_{k,\ve} \to \vp_{k,0}$ strongly in $H^1_0(\Om)$, \eqref{eq:propG} and \eqref{eq:egalG} now show that $u_\ve \to u_0 = |\vp_{k,0}|$ strongly in $L^{2^*}(\Om)$. Since $\vp_{k,0}$ solves \eqref{eq:convu0}, unique continuation results (see \cite{Aronszajn} or \cite{HardtSimon}) show that $\vp_{k,0}$ vanishes on a set of measure zero in $\Om$. In particular, $u_0 = |\vp_{k,0}| \in L^{2^*}_{>}(\Om)$. The normalisation condition  $\int_{\Om} u_\ve^{2^*-2}\vp_{k,\ve}^2dx = 1$ also passes to the limit and shows that 
 $$ \int_{\Om} u_0^{2^*}dx = \int_{\Om} u_0^{2^*-2} \vp_{k,0}^2dx = \int_{\Om} |\vp_{k,0}|^{2^*}dx = 1, $$
 and thus that $\vp_{k,0} \neq 0$. Proposition \ref{prop:signvp} shows that $\mu_1^{\lambda}(u_0), \dots, \mu_{N(i)}^{\lambda}(u_0)$ are nonpositive.  Since $\mkl(\Omega)>0$, \eqref{eq:convue} and \eqref{eq:convu0} imply that 
$$ \int_{\Om} u_0^{2^*-2} \vp_{k,0} \vp_p^{\lambda}(u_0)dx = 0 \quad \text{ for all } 1 \le p \le k-1. $$
Thus $V =  \text{Span} \big( \vp_1^{\lambda}(u_0), \dots, \vp_{k-1}^\lambda(u_0), \vp_{k,0}\big) $ satisfies $\dim_{u_0} V = k$ and is an admissible subspace in the definition of $\mkl(u_0)$: \eqref{muk} and \eqref{eq:convu0} then show that $\mkl(u_0) \le \mkl(\Om)$. By definition of $\mkl(\Om)$ we thus have $\mkl(u_0) = \mkl(\Om)$, and thus $u_0$ attains $\mkl(\Om)$ and $\vp_{k,0} \in E_k^{\lambda}(u_0)$. A straightforward adaptation of the proof  of Lemma \ref{lemme4} again shows that $\dim_{u_0} E_k^{\lambda}(u_0) = 1$. Hence $\vp_{k,0} = \vkl(u_0)$ and if $k \ge 2$, that is if $i \ge 1$, it is sign-changing. This proves point (1) of Theorem  \ref{prop:minmkl1}. 
 
 \medskip
 We now prove point (2). That $\psi_0 = {\mu_{k}^{\lambda}(\Om)}^{\frac{n-2}{4}} \vp_{k,0}$ solves \eqref{eq:critlambda} follows from a simple scaling argument, and we see that 
 $$ \int_{\Om} |\psi_0|^{2^*}dx = \mkl(\Om)^{\frac{n}{2}}. $$
 The regularity of $\psi_0$ is again a consequence of the regularity theory of \cite{BrezisKato} (see also \cite{Trudinger}) and of standard elliptic theory. 

 \medskip
 
 We prove point (3) of Theorem \ref{prop:minmkl1}. Let $v \in H^1_0(\Om)$ be a non-zero sign-changing solution of \eqref{eq:critlambda}. By Proposition \ref{prop:structure:vp}, there exists $p \ge 1$ such that $\Vert v \Vert_{2^*}^{2^*-2} = \mu_{p}^{\lambda}(|\tilde{v}|)$ and $\mu_{p}^{\lambda}(|\tilde{v}|) >0$, where we have let $\tilde v = \frac{v}{\Vert v \Vert_{2^*}}$. Proposition \ref{prop:signvp} then shows that $p \ge k$, so that $\mpl(|\tilde{v}|) \ge \mkl(|\tilde{v}|)$. Since $\psi_0 := {\mu_{k}^{\lambda}(\Om)}^{\frac{n-2}{4}} \vp_{k,0}$, a simple computation shows that $\Vert \psi_0 \Vert_{2^*}^{2^*-2} = \mkl(\Om)$. We then have 
 $$ \Vert v \Vert_{2^*}^{2^*-2} = \mu_{p}^{\lambda}(|\tilde{v}|) \ge \mkl(|\tilde{v}|) \ge \mkl(\Om) = \Vert \psi_0 \Vert_{2^*}^{2^*-2}, $$
which shows that $\psi_0$ is a solution of \eqref{eq:critlambda} of least energy. Clearly, if $i \ge 1$, it changes sign. 

\medskip

We finally prove point (4) of Theorem \ref{prop:minmkl1}. The proof closely follows the lines of Courant's nodal domain theorem. Assume indeed that $\psi$ has at least $k+1 = N(i)+2$ nodal domains in $\Om$ denoted $\Omega_1, \dots, \Omega_{k+1}$. Let $u_0 = |\psi|$ and denote by $\vp_{\ell,0} = \vp_\ell^{\lambda}(u_0)$, for $1 \le \ell \le k-1$, the $k-1$ first generalised eigenfunctions associated to $u_0$ and given by Proposition \ref{prop:vp}. By \eqref{supmupneg} we have $\mu_\ell^{\lambda}(u_0) \le0$ for $1 \le \ell \le k-1$. Let $a_1, \dots, a_{k}$ be real numbers chosen such that 
 $$ \theta:= \sum_{\ell=1}^{k-1} a_\ell \vp_{\ell,0} \mathds{1}_{\Omega_\ell} + a_{k} \psi \mathds{1}_{\Omega_{k}}  $$
 satisfies $\int_{\Om} u_0^{2^*-2} \vp_{\ell,0}\theta dx = 0$ for all $1 \le \ell \le k-1$ and $\int_{\Om} u_0^{2^*-2} \theta^2 dx = 1$. Then $V = \text{Span} \big( \vp_{1,0}, \dots, \vp_{k-1,0}, \theta)$ is an admissible subspace in the definition of $\mu_{k}^{\lambda}(u_0)$ and we easily see that 
 $$ \sup_{v \in V \backslash \{0\}} Q_{u_0}^{\lambda}(v) = Q_{u_0}^{\lambda}(\theta) = \mu_{k}^{\lambda}(\Om), $$
 where $Q_{u_0}^{\lambda}$ is as in \eqref{defQu}. We then have $\mu_{k}^{\lambda}(u_0)  = \mu_{k}^{\lambda}(\Om)$ and $\theta$ attains $\mu_{k}^{\lambda}(u_0)$: $\theta$ is therefore a generalised eigenfunction associated to $u_0$ and satisfies 
 $$ - \Delta \theta - \lambda \theta = \mu_{k}^{\lambda}(\Om) u_0^{2^*-2} \theta \quad \text{ in } \Om. $$
 However, $\theta$ vanishes on the non-empty open set $\Om_{k+1}$ by construction. Since $u_0$ is continuous, standard unique continuation results (see for instance \cite{Aronszajn}) show that $\theta \equiv 0$ in $\Om$, a contradiction. Hence $\psi$ has at most $k$ nodal domains. If $i=1$, that is if $\lambda \in [\Lambda_1, \Lambda_{2})$, $\psi$ changes sign and it hence has exactly two nodal domains. If $i=0$, that is $\lambda \in [0, \Lambda_1)$, $\psi$ is positive up to sign. This proves Theorem \ref{prop:minmkl1}. 
  \end{proof}
  
The arguments in the proof of Theorem \ref{prop:minmkl1} more generally prove the following result, where we do not assume that the strict inequality $\mkl(\Om) < K_n^{-2}$ holds: 
  \begin{prop} \label{prop:extremales}
Let $\lambda \in [\Lambda_i, \Lambda_{i+1})$ for some $i \ge 0$ and assume that $\mkl(\Om)$ is attained at some function $u \in L^{2^*}_{>}(\Om) $ with  $\Vert u \Vert_{2^*} = 1$. Then $\mkl(u)$ is simple and $u = |\vp|$, where $\vp = \vp_k^{\lambda}(u)$ is the unique generalised eigenvector associated to $u$ by Proposition \ref{prop:vp} and normalised by $\int_{\Om} u^{2^*-2} \vp^2dx = 1$. In addition, for any fixed $h \in L^\infty(\Om)$, the mapping $t \mapsto \mkl \big((1+th)u \big)$ is differentiable at $0$ and 
$$ \frac{d}{dt}_{|t=0}  \mkl \big((1+th)u \big) = - (2^*-2) \mkl(u) \int_{\Om} u^{2^*} h dx.$$
  \end{prop}
  
  \begin{proof}
For $v \in L^{2^*}_{>}(\Om)$ we define 
\begin{equation*} 
F(v) = \mkl(v) \Vert v \Vert_{2^*}^{2^*-2}.
\end{equation*}
Let $u\in L^{2^*}_{>}(\Om) $,  $\Vert u \Vert_{2^*} = 1$ attain $\mkl(\Om)$. Let $h \in L^\infty(\Om)$ and let, for $|t| < \Vert h \Vert_{\infty}^{-1}$, $u_{t} = (1+th)u$. Since $u$ attains $\mkl(\Om)$ we have $ F(u_{t}) \ge F(u)$ for all $t$ small enough, which implies that
\begin{equation*}
  \frac{d}{dt}_{|t=0_+} F(u_{t})  \ge 0 \ge  \frac{d}{dt}_{|t=0_-} F(u_{t}). 
  \end{equation*}
Using \eqref{eq:derlefte} and \eqref{eq:derrighte}, which are formally true for $\ve = 0$, we obtain that 
\begin{equation*}
\begin{aligned}
&  \sup_{\vp \in E_k^{\lambda}(u)\backslash \{0\}} \Bigg( - (2^*-2) \mkl(\Om) \frac{\int_\Om u^{2^*-2}h \vp^2dx}{\int_\Om u^{2^*-2} \vp^2 dx}\Bigg) \\
  & \quad\le  \inf_{\vp \in E_k^{\lambda}(u)\backslash \{0\}} \Bigg( - (2^*-2) \mkl(\Om) \frac{\int_\Om u_\ve^{2^*-2}h \vp^2dx}{\int_\Om u^{2^*-2} \vp^2 dx}\Bigg) 
  \end{aligned} 
  \end{equation*}
for any $h \in L^\infty(\Om)$. By mimicking the proof of Lemma \ref{lemme4} this shows that $\dim_{u} E_k^{\lambda}(u) = 1$. We let $\vp \in E_k^{\lambda}(u)$ satisfy $\int_{\Om} u^{2^*-2} \vp^2 dx = 1$. The arguments in Lemma \ref{lemme4bis}, together with Proposition \ref{prop:dervp}, then show that $t \mapsto \mkl \big((1+th)u \big)$ is differentiable at $0$, that $u = |\vp|$ a.e. and that 
$$ \frac{d}{dt}_{|t=0}  \mkl \big((1+th)u \big) = - (2^*-2) \mkl(u) \int_{\Om} u^{2^*} h dx.$$
  \end{proof}

\subsection{Additional properties of $\mkl(\Om)$.}

We conclude this section by proving several properties of $\mkl(\Om)$ that will be useful in the proof of our main results. We first investigate the function $\lambda \mapsto \mkl(\Om)$:

\begin{prop} \label{prop:C0:mkl}
Let $n \ge 3$ and $\Om$ be a smooth bounded domain of $\R^n$. Let $\lambda \in [\Lambda_i, \Lambda_{i+1})$ for some $i \ge 0$. Let $k = N(i)+1$, where $N$ is given by \eqref{defN}, and consider $\mkl(\Om)$ given by \eqref{eq:minmuk}. Then 
\begin{itemize}
\item $\lambda \mapsto \mkl(\Om)$ is Lipschitz, positive and non-increasing in $[\Lambda_i, \Lambda_{i+1})$
\item We have $   \lim \limits_{\lambda \underset{< }{\to} \Lambda_{i+1}} \mkl(\Om)= 0$.
\item If $\mkl(\Om)$ is attained at some $\lambda \in  [\Lambda_i, \Lambda_{i+1})$ then $\mu_k^{\lambda'}(\Om) < \mkl(\Om)$ for every $\lambda' \in  [\Lambda_i, \Lambda_{i+1}), \lambda' > \lambda$. 
\end{itemize}
\end{prop}
A consequence of the last bullet is that if $I \subseteq [\Lambda_i, \Lambda_{i+1})$ is an interval such that $\mkl(\Om)$ is attained for every $\lambda \in I$, then $\lambda \in I \mapsto \mkl(\Om)$ is decreasing.

\begin{proof}
 It is easily seen by \eqref{muk} and \eqref{eq:minmuk} that $\lambda \in [\Lambda_i, \Lambda_{i+1}) \mapsto \mkl(\Om)$ is non-increasing. By \eqref{infmukpos} it is positive in  $[\Lambda_i, \Lambda_{i+1})$. We now prove that $\lambda \in [\Lambda_i, \Lambda_{i+1}) \mapsto \mkl(\Om)$ is Lipschitz. Let for this $\lambda, \lambda' \in   [\Lambda_i, \Lambda_{i+1})$, $\lambda \le \lambda'$. Let $u \in L^{2^*}_{>}(\Om)$ with $\Vert u \Vert_{2^*} =1$ and $V\subset H^1_0(\Om)$, $\dim_u V = k$. For $v \in V \backslash \{0\}$ we have, by Poincar\'e's inequality:
$$  \begin{aligned}
 Q_u^{\lambda}(v)  & = Q_u^{\lambda'}(v)  + (\lambda'-\lambda) \frac{\int_{\Om} v^2 dx}{\int_{\Om} u^{2^*-2}v^2dx} \\
 & \le Q_u^{\lambda'}(v)  + \frac{\lambda'-\lambda}{\Lambda_1} \frac{\int_{\Om} |\nabla v|^2 dx}{\int_{\Om} u^{2^*-2}v^2dx}
\end{aligned} $$
where $Q_u^{\lambda}$ is as in \eqref{defQu}. Taking the supremum over $v \in V \backslash \{0\}$ then the infimum over all subspaces $V$ with $\dim_u V = k$ yields 
$$ \mkl(u) \le \mu_k^{\lambda'}(u) +  \frac{\lambda'-\lambda}{\Lambda_1} \mu_k^{0}(u), $$
where $\mu_k^{0}(u)$ is as in \eqref{muk}. Taking the infimum over $u \in L^{2^*}_{>}(\Om), \Vert u \Vert_{2^*} = 1$ and since $\mkl(\Om)$ is nonincreasing in $\lambda$ gives in the end 
\begin{equation} \label{croissance:muk}
  \mu_k^{\lambda'}(\Om) \le \mkl(\Om) \le \mu_k^{\lambda'}(\Om) +  \frac{\lambda'-\lambda}{\Lambda_1} \mu_k^{0}(\Om), 
  \end{equation}
which proves that $\mkl(\Om)$ is a Lipschitz function of $\lambda$. Assume now that $ \lambda \in [\Lambda_i, \Lambda_{i+1})$ is such that  $\mkl(\Om)$ is attained. Let $\lambda' \in [\Lambda_i, \Lambda_{i+1})$, with $\lambda' > \lambda$. Let $u_\lambda \in L^{2^*}_{>}(\Om)$ attain $\mkl(\Om)$ and let $\vp_1^{\lambda}(u_\lambda), \dots, \vkl(u_{\lambda})$ be the $k$ first generalised eigenvectors for $u_\lambda$. Let $V = \text{Span}\{\vp_1^{\lambda}(u_\lambda), \dots, \vkl(u_{\lambda})\}$. Since $\lambda' > \lambda$ and $u_\lambda \in L^{2^*}_{>}(\Om)$ we have 
$$ \max_{v \in V \backslash \{0\}} Q_{u_\lambda}^{\lambda'}(v) < \max_{v \in V \backslash \{0\}} Q_{u_\lambda}^{\lambda}(v) = \mkl(\Om) .$$
Taking the infimum over $k$-dimensional subspaces $V$ and then over functions $u \in L^{2^*}_{>}(\Om)$  gives $\mu_k^{\lambda'}(\Om) < \mkl(\Om)$.

We now compute the limit of $\mkl(\Om)$ as $\lambda$ converges to $\Lambda_{i+1}$ from the left. Let $u\equiv 1$ and let $V_0 = \text{Span}\big( \vp_1, \dots, \vp_{k} \big)$, where $ \vp_1, \dots, \vp_{k}$ are the first $k$-eigenfunctions of $- \Delta$ in $H^1_0(\Om)$, normalised so that $\int_{\Om} \vp_\ell^2 dx = 1$ for $1 \le \ell \le k$. It satisfies $\dim_1 V_0 = \dim V_0 = k$. By construction $\vp_{k}$ is associated to the eigenvalue $\Lambda_{i+1}$. This choice of $u$ and $V$, together with the definition \eqref{eq:minmuk} of $\mkl(\Om)$, show that 
\begin{equation} \label{test:function:C0}
\begin{aligned}
0 & < \mkl(\Om) \le \mkl(1)|\Om|^{\frac{2}{n}}  \\
& \le |\Om|^{\frac{2}{n}} \sup_{v \in V_0 \backslash \{0\}} \frac{\int_{\Om} \big( |\nabla v|^2 - \lambda v^2 \big) dx}{\int_{\Om} v^2 dx } = |\Om|^{\frac{2}{n}} (\Lambda_{i+1} - \lambda),
\end{aligned}
\end{equation}
from which the result follows. 
\end{proof}

The next result relates $\mkl(\Om)$ to the energy function $\mathcal{E}_{\lambda}(\Om)$ of \eqref{eq:critlambda} given by \eqref{minimalenergy}  provided low-energy solutions of \eqref{eq:critlambda} exist: 

\begin{prop} \label{prop:liens}
Let $\lambda \in [\Lambda_i, \Lambda_{i+1})$ for some $i \ge 0$ and let $k = N(i)+1$ where $N$ is given by \eqref{defN}. Assume that there exists a solution $v$ of \eqref{eq:critlambda} satisfying $\int_{\Om} |v|^{2^*} dx< K_n^{-n}$, where $K_n$ is given by \eqref{defK0}. Then $\mkl(\Om)$ is attained and we have $\mkl(\Om)^{\frac{n}{2}} = \mathcal{E}_{\lambda}(\Om)$. 
\end{prop}
Under the assumption of Proposition \ref{prop:liens} we have in particular $\mathcal{E}_\lambda(\Om) < K_n^{-n}$.

\begin{proof}
Let $v$ be a solution of \eqref{eq:critlambda} such that $E(v) < K_n^{-n}$. Let $\tilde v = \frac{v}{\Vert v \Vert_{2^*}}$. Proposition \ref{prop:structure:vp} shows that there exists $p \ge 1$ such that $\mpl(\tilde v) = \Vert v \Vert_{2^*}^{2^*-2} < K_n^{-2}$. Proposition \ref{prop:signvp} then shows that $p \ge k$, where $k = N(i)+1$. Since the sequence $(\mpl(\tilde v))_{p \ge1}$ is nonincreasing, we then have $\mkl(\Om) \le \mkl(\tilde v) < K_n^{-n}$. Theorem \ref{prop:minmkl1} then shows that a least-energy solution of \eqref{eq:critlambda} exists and that it has energy $\mkl(\Om)^{\frac{n}{2}}$. Thus, $\mkl(\Om)^{\frac{n}{2}} = \mathcal{E}_{\lambda}(\Om)$. 
\end{proof}

\section{Proving the strict inequality $\mkl(\Om)<K_n^{-2}$} \label{fonctionstest}

In this section we prove inequality \eqref{eq:test1}. We first recall some well-known facts about the Green's function. Assume for the moment that $n=3$ and let $\lambda \ge0$, $\lambda \not \in \Sp(-\Delta)$, so that $-\Delta - \lambda$ is invertible in $H^1_0(\Om)$. It is well-known (see e.g.  \cite{RobDirichlet}) that $- \Delta - \lambda$ admits a Green's function  $G_\lambda: \overline{\Om} \times \overline{\Om} \backslash \{x=y\}$ which, for every $x \in \Om$, is the unique solution of 
\begin{equation} \label{green:coercif}
  - \Delta G_\lambda(x, \cdot) - \lambda G_\lambda(x, \cdot) = \delta_x \quad \text{ in } \Om, \quad G_\lambda(x, \cdot) = 0 \text{ in } \partial \Om.
 \end{equation}
Standard properties of Green's functions (see again \cite{RobDirichlet}) show that, for any $x \in \Om$, there is a real number $m_\lambda(x) \in \R$ such that the following holds
\begin{equation}\label{eq:asymG}
	G_\lambda(x,y)=\frac{1}{4 \pi |x-y|}+ m_\lambda(x) + O(|x-y|) \quad \text{ as } y \to x.
\end{equation}
Also, there exists $C>0$ depending on $\lambda$ and $\Om$ such that, for any $x \neq y$, 
\begin{equation}\label{eq:estimateGp}
	|G_\lambda(x,y)| \le C|x-y|^{-1} \quad \bb{ and } \quad |\nabla G_\lambda(x,y)|\le C|x-y|^{-2}.
\end{equation}
Note that \eqref{eq:estimateGp} remains true uniformly in $\lambda$ provided $\lambda$ belongs to a compact subset $K \subset \R \backslash \Sp(-\Delta)$, in which case $C$ depends on $K, \Om$. As mentioned in the introduction, $m_\lambda(x)$ in \eqref{eq:asymG} is called \emph{the mass} of $G_{\lambda}$ at $x$. Note that the mass function is the opposite of the usual Robin's function, defined e.g. in  \cite[Appendix A]{DelPinoDolbeaultMusso}. The properties of $m_\lambda$ have been thoroughly investigated in the coercive case $\lambda \in (0, \Lambda_1)$ (see e.g. \cite{DelPinoDolbeaultMusso}), but less so in the non-coercive case $\lambda \ge \Lambda_1$ where the maximum principle does not hold anymore. We refer to Appendix \ref{annexe:masse} where we prove several properties of $m_\lambda$ for any $\lambda \not \in \Sp(-\Delta)$: in particular that it is increasing in $\lambda$ as $\lambda$ varies between consecutive eigenvalues of $- \Delta$ (see \eqref{der:mass:lambda} below) and that for a fixed $\lambda \not \in \Sp(-\Delta)$, $m_\lambda(x) \to - \infty$ as $x \to \partial \Om$ (see \eqref{masse:moins:infini} below). 

\medskip

The main result of this section is as follows:

 \begin{theo}\label{prop:test:func}
 Let $\Om \subset \R^n$, $n \ge 3$, be bounded and smooth and let  $\lambda >0$. Suppose that one of the following assumptions is satisfied: 
\begin{itemize}
\item $n\ge5$ 
\item $n=4$ and $\lambda \not \in \Sp(-\Delta)$
\item  $n=3$, $\lambda \not \in \Sp(-\Delta)$ and there is $x \in \Om$ such that $m_\lambda(x) >0$, where $m_{\lambda}$ is as in \eqref{eq:asymG}. 
\end{itemize}
Let $i \ge0$ be such that $\lambda \in [\Lambda_i, \Lambda_{i+1})$ and let $k = N(i)+1$. Then
 	\begin{equation} \label{eq:test12}
 		\mkl(\Om) < K_n^{-2},
 	\end{equation}
where $K_n$ is as in \eqref{defK0} and $\mkl(\Om)$ is as in \eqref{eq:minmuk}.  
 \end{theo}
When $\lambda \in(0, \Lambda_1)$, i.e. when $i=0$, Theorem  \ref{prop:test:func} follows from Lemma \ref{lemme:muI} below and the celebrated results in \cite{BN, Druetdim3}. We may thus restrict to the case $i \ge 1$ in the following. For  $\lambda \ge 0$ we define, for $u \in H_0^1(\Om)$,
\begin{equation}\label{eq:Jlambdap}
	J_\lambda(u)= \frac{\int_\Om \left( |\nabla u|^2-\lambda u^2\right)\, dx}{\left( \int_\Om |u|^{\crit}\, dx\right) ^{\frac{2}{\crit}}}.
\end{equation}
We first consider the $n\ge 4$ case, whose proof is very simple:

\begin{proof}[Proof of Theorem \ref{prop:test:func} when $n \ge 5$ or ($n=4$ and $\lambda \not \in \Sp(-\Delta)$)] 
We assume first that $n \ge 5$ or that $n = 4$ and $\lambda \not \in \Sp(-\Delta)$. Let $i \ge1$ be such that $\lambda \in [\Lambda_i, \Lambda_{i+1})$ and let $k = N(i)+1 \ge 2$. By the results in  \cite{SzulkinWethWillem} there exists a least-energy sign-changing solution $u_0$ of \eqref{eq:critlambda} that satisfies $\Vert u_0 \Vert_{2^*}^{2^*}  <K_n^{-n}$. We let $\vp_1^{\lambda}(u_0), \dots, \vp_{k-1}^{\lambda}(u_0)$ be generalised eigenvalues associated to $u_0$ as in Proposition \ref{prop:vp} and we let
$$ V = \text{Span} \Big \{\vp_1^{\lambda}(u_0), \dots, \vp_{k-1}^{\lambda}(u_0), u_0 \Big\}. $$
By \eqref{supmupneg} and since $u_0$ solves \eqref{eq:critlambda} it is easily seen that $(u_0, \vp_\ell^{\lambda}(u_0) )_{L^2_{u_0}} = 0$ for all $1 \le \ell \le k-1$, so that $\dim_{u_0} V = k$. By \eqref{supmupneg} we then have 
$$ \mkl(u_0) \le \max_{v \in V \backslash \{0\}} Q_{u_0}^{\lambda}(v) = Q_{u_0}^{\lambda}(u_0)  = 1 $$
since $u_0$ solves \eqref{eq:critlambda}. This implies that 
$$\mkl(\Om) \le  \mkl(u_0) \Vert u_0 \Vert_{2^*}^{2^*-2} \le \Vert u_0 \Vert_{2^*}^{2^*-2} <  K_n^{-2}.$$
\end{proof}
The rest of this section is devoted to the proof of Theorem \ref{prop:test:func} when $n=3$. The proof follows from test-function computations as in the coercive case. The main novelty, however, is that when $i \ge 1$ we have to find a suitable test weight and test $k$-dimensional subspace, and that due to the min-max characterisation of $\mkl(u)$ our test-function computations are only semi-explicit. Estimating $\mkl(u)$ forces us to be very precise in our expansions and to analyse all the interactions between the weight $u$ and elements of $V$. 
\medskip

As before we let $i \ge1$ be such that $\lambda \in (\Lambda_i, \Lambda_{i+1})$ and let $k = N(i)+1 \ge 2$. We let $U_0$ be as in \eqref{defU0} and we let $\vp_1, \dots, \vp_{k-1}$ be eigenfunctions of $-\Delta$ in $H^1_0(\Omega)$, normalised by $\Vert \vp_i \Vert_2 = 1$, respectively associated to the eigenvalues $\lambda_1, \dots, \lambda_{k-1}$. For $1 \le \ell,p,q \le k-1$ we thus have 
 \begin{equation} \label{norma1} 
 \begin{aligned}   
 - \Delta \vp_\ell & =  \lambda_\ell \vp_\ell \quad \text{ in } \Om, \\
  	\int_{\Om} \langle \nabla \vp_p, \nabla \vp_q \rangle dx & = \lambda_p \delta_{pq} \text{ and }  \int_\Om \varphi_{p} \varphi_{q} dx =  \delta_{pq}.
 \end{aligned} 
 \end{equation}
Let $x_0 \in \Omega$ satisfy $m_\lambda(x_0) >0$. For $\ve >0$ we define, for $x \in \Omega$,
\begin{equation*} 
\begin{aligned}
	B_{\ve}(x)	= 4 \pi |x_0-x|  G_\lambda(x_0, x) \ve^{- \frac12} U_0\left( \frac{x-x_0}{\ve}\right).
	\end{aligned} 
	\end{equation*}
We also define, for any $\ve >0$,
  \begin{equation} \label{def:ueps}
 u_\ve = \frac{|B_{\ve}|}{\Vert B_{\ve} \Vert_{6}} \quad \text{ and } \quad  \tilde{u}_\ve = \frac{B_{\ve}}{\Vert B_{\ve} \Vert_{6}} ,
   \end{equation} 
so that $\tilde{u}_\ve \in H^1_0(\Om)$. Since $G_\lambda(x_0, \cdot)$ satisfies \eqref{green:coercif}, the results in \cite{HardtSimon} ensure that $G_\lambda(x_0, \cdot)$ vanishes on a set of measure zero in $\Omega$. Since $U_0 >0$ we thus have $u_\ve >0$ a.e. in $\Om$, and thus $u_{\ve}\in L^{6}_>(\Om)$. Straightforward computations using \eqref{green:coercif}, \eqref{eq:asymG} and \eqref{energieU0} show that
\begin{equation} \label{DLnormeB}
\int_{\Om} u_{\ve}^{6}dx = K_3^{-3} + O(\ve^3)
\end{equation} 
and that 
\begin{equation}\label{eq:Jlambdauepx03}
\begin{aligned}
		J_{\lambda}(\tilde{u}_{\ve})& =K_3^{-2} - C_3 m_\lambda(x_0) \ve  + o(\ve) 
\end{aligned} 
	\end{equation}
as $\ve \to 0$, where $m_\lambda(x_0)$ is as in \eqref{eq:asymG} and where $C_3$ is a positive constant (see for instance \cite{SchoenYamabe} or \cite{EspositoPistoiaVetois}). Similarly, if $0 \le \ell,s \le k-1$ are integers, we have the following estimates as $\ve \to 0$: 
\begin{equation} \label{label:unique}
\begin{aligned}
\int_{\Om} u_\ve^{4} \vp_\ell \vp_s dx & = O(\ve), \\
	\int_{\Om} u_{\ve}^{4}\tilde{u}_\ve \varphi_{\ell}\, dx& = 4 \pi \sqrt{3} K_3^{\frac{5}{2}}\vp_\ell(x_0) \ve^{\frac{1}{2}} + o(\ve), \\
\int_{\Om} \vp_\ell \tilde{u}_\ve dx &  = 
\frac{ 4 \pi \sqrt{3}}{\lambda_\ell-\lambda}K_3^{\frac{1}{2}}  \vp_{\ell}(x_0)  \ve^{\frac{1}{2}}+ o( \ve^{\frac{1}{2}}) \quad \text{ and } \\
\int_{\Om} \langle \nabla \vp_\ell, \nabla \tilde{u}_\ve \rangle dx & =   \frac{ \lambda_\ell}{\lambda_\ell-\lambda}4 \pi \sqrt{3} K_3^{\frac{1}{2}}  \vp_{\ell}(x_0)  \ve^{\frac{1}{2}}+ o( \ve^{\frac{1}{2}}). 
\end{aligned}
\end{equation}
These computations can again be found in \cite{EspositoPistoiaVetois, SchoenYamabe}. We now prove Theorem \ref{prop:test:func}:

\begin{proof}[Proof of Theorem \ref{prop:test:func} when $n=3$]

 If $u \in L^{6}_{>}(\Om)$ and $v \in H^1_0(\Om)\backslash \{0\}$ we let 
 \begin{equation}\label{Fuepvep0}
 	F(u,v)=\frac{\int_{\Om}\left( |\nabla v|^2 - \lambda v^2 \right) \, dx}{\int_{\Om} u^{4} v^2 \, dx} \Vert u \Vert_{6}^{4}.
 \end{equation} 
 Let 
$$ V_\ve = \text{Span} \big( \vp_{1}, \dots, \vp_{k-1}, \tilde{u}_\ve\big), $$
where $\tilde{u}_\ve$ is given by \eqref{def:ueps}. We first observe that $V_\ve$ is a $k$-dimensional subset of $H^1_0(\Om)$. Indeed, since $\Vert u_\ve \Vert_{6} = 1$ we have, by \eqref{eq:Jlambdauepx03} and the definition of $\tilde{u}_\ve$, that
\begin{equation*} 
 \int_{\Om} |\nabla \tilde{u} _\ve|^2 dx = K_3^{-2} + O(\ve) 
 \end{equation*}
as $\ve \to 0$. Together with \eqref{norma1} and \eqref{label:unique} this shows that the Gram determinant of $( \vp_{1}, \dots, \vp_{k-1}, \tilde{u}_\ve)$ for the $H^1_0(\Om)$ scalar product converges to $ \lambda_1 \cdots \lambda_{k-1} K_3^{-2} >0 $ as $\ve \to 0$, hence $\dim V_\ve = k$. Since $u_\ve \in L^{6}_{>}(\Om)$ we also have $\dim_{u_\ve} V_\ve = k$, so that $V_\ve$ is an admissible subspace to compute $\mu_k^{\lambda}(u_\ve)$. We are going to show that, under the assumptions of Theorem \ref{prop:test:func}, there is a sequence $(\ve_m)_{m \ge 0}$ with $\ve_m \to 0$ as $m \to + \infty$ such that 
\begin{equation} \label{supve}
 \max_{ v \in V_{\ve_m} \backslash \{0\} } F(u_{\ve_m},v) < K_3^{-2}
  \end{equation}
for $m$ large enough, where $F$ is as in \eqref{Fuepvep0}. In view of \eqref{muk} and \eqref{eq:minmuk} this will imply \eqref{eq:test12}. For any $\ve>0$, we let $(\beta_{1,\ve},...,\beta_{k,\ve})\in \rr^k\backslash\{0\}$ be such that
 \begin{equation}\label{defve}
 	v_\ve:=\sum_{\ell=1}^{k-1}\beta_{\ell,\ve} \varphi_{\ell}+\beta_{k,\ve}\tilde{u}_{\ve}
 \end{equation}
attains the left-hand side of \eqref{supve}, that is satisfies 
 \begin{equation*} 
 	F(u_{\ve},v_{\ve})=\max_{\beta\in \rr^{k}\backslash\{0\}}F\left( u_{\ve},\sum_{\ell=1}^{k-1}\beta_{\ell} \varphi_{\ell}+\beta_{k}\tilde{u}_{\ve}\right).
 \end{equation*}
 Without loss of generality we can assume that
\begin{equation} \label{norma2}
\sum_{\ell=1}^{k}\beta_{\ell,\ve}^2=1.
\end{equation}
First, by \eqref{eq:Jlambdap} and \eqref{def:ueps}  we have 
$$  \int_{\Omega}\big( |\nabla \tilde{u}_\ve|^2 - \lambda \tilde{u}_\ve^2 \big) dx = J_\lambda (\tilde{u}_\ve). $$
With  \eqref{norma1} straightforward computations then show that
\begin{equation} \label{test:est1}
\begin{aligned}
  \int_{\Om} \left( |\nabla v_\ve|^2-\lambda v_\ve^2\right)\, dx & = \sum_{\ell=1}^{k-1} (\lambda_\ell - \lambda) \beta_{\ell,\ve}^2+ \beta_{k, \ve}^2 J_{\lambda}(\tilde{u}_\ve) \\
& + 2 \beta_{k,\ve}\sum_{\ell=1}^{k-1}  \beta_{\ell,\ve}(\lambda_\ell - \lambda) \int_{\Om} \vp_\ell \tilde{u}_\ve dx .
\end{aligned} 
\end{equation}
Straightforward computations with \eqref{def:ueps} also show that 
\begin{equation} \label{test:est2:2}
\begin{aligned}
	\int_{\Om} u_{\ve}^{4}v_{\ve}^2\, dx =  \beta_{k, \ve}^2 + 2\beta_{k,\ve}\sum_{\ell=1}^{k-1}\beta_{\ell,\ve} \int_{\Om} u_{\ve}^{4} \tilde{u}_\ve \varphi_{\ell}\, dx +  \int_{\Om} u_\ve^{4}\Big( \sum_{\ell=1}^{k-1} \beta_{\ell,\ve} \vp_\ell \Big)^2dx. \\
	\end{aligned} 
\end{equation}
By \eqref{norma2} there exists a subsequence $(\ve_m)_{m\ge 0}$, with $\ve_m \to 0$ as $m \to + \infty$, such that $\beta_{\ell,\ve_m}$ converges towards $\beta_{\ell,0}$ for any $1 \le \ell \le k$, that still satisfy $\sum_{\ell=1}^k \beta_{\ell,0}^2 = 1$. In the following, for the sake of clarity, we will omit the subscript $m$, and will keep denoting the sequences $\beta_{\ell,\ve_m}, u_{\ve_m}, v_{\ve_m}, \dots$ by $\beta_{\ell,\ve}, u_\ve, v_\ve, \dots$. But from now on all these quantities will be computed with respect to this subsequence. 

\medskip

Assume first that $\beta_{k,0} = 0$, so that $\sum_{\ell=1}^{k-1} \beta_{\ell,0}^2 = 1$. Since $\lambda \in (\Lambda_i, \Lambda_{i+1})$, by \eqref{label:unique} and \eqref{test:est1} we have 
$$  \int_{\Om} \left( |\nabla v_\ve|^2-\lambda v_\ve^2\right)\, dx \to \sum_{\ell=1}^{k-1} (\lambda_\ell - \lambda) \beta_{\ell,0}^2 \le  (\Lambda_i - \lambda) < 0 $$
as $\ve \to 0$. The numerator in $F(u_\ve, v)$ is thus nonpositive for $\ve$ small enough and since $\int_{\Om} u_{\ve}^{4}v_{\ve}^2\, dx >0$ we obtain $ \max_{ v \in V_{\ve} \backslash \{0\} } F(u_\ve,v) \le 0$ as $\ve \to 0$, which proves \eqref{supve} in this case. Assume now that $0 <\beta_{k,0}^2 < 1$, so that $\sum_{\ell=0}^{k-1} \beta_{\ell,0}^2 >0$. Since $\Vert u_\ve \Vert_{6}= 1$,  with  \eqref{eq:Jlambdauepx03}, \eqref{label:unique},  \eqref{test:est1} and \eqref{test:est2:2} we obtain that 
$$ F(u_\ve, v_\ve) \to \frac{\sum_{\ell=1}^{k-1} (\lambda_\ell - \lambda) \beta_{\ell,0}^2 + \beta_{k,0}^2 K_3^{-2}}{\beta_{k,0}^2} < K_3^{-2}  $$
as $\ve \to 0$, since $\lambda_\ell < \lambda$ for any $0 \le \ell \le k-1$. This proves \eqref{supve} in this case. We may therefore assume in the rest of this subsection that $\beta_{k,0}^2 = 1$ and $\beta_{\ell,0} = 0$ for all $1 \le \ell \le k-1$. A consequence of \eqref{label:unique} is then that
$$ \int_{\Om} u_\ve^{4}\big( \sum_{\ell=1}^{k-1} \beta_{\ell,\ve} \vp_\ell \big)^2dx = o(\ve) \quad \text{ as } \ve \to 0.$$ 
Plugging the latter in \eqref{test:est2:2} and using \eqref{label:unique}  now gives, as $\ve \to 0$,
\begin{equation} \label{test:est2}
\begin{aligned}
	\int_{\Om}& u_{\ve}^{4}v_{\ve}^2\, dx  =  \beta_{k, \ve}^2 + 8 \pi \sqrt{3} K_3^{\frac{5}{2}} \ve^{\frac{1}{2}}\beta_{k,\ve}\sum_{\ell=1}^{k-1}\beta_{\ell,\ve}\vp_\ell(x_0) + o(\ve) . 
\end{aligned} 
\end{equation}
We now use \eqref{label:unique} to rewrite \eqref{test:est1} as 
\begin{equation} \label{calculs:31}
\begin{aligned}
  \int_{\Om} \Big( |\nabla v_\ve|^2& -\lambda v_\ve^2\Big)\, dx  = \sum_{\ell=1}^{k-1} (\lambda_\ell - \lambda) \beta_{\ell,\ve}^2+ \beta_{k, \ve}^2 J_{\lambda}(\tilde{u}_\ve) \\
& + 8 \pi \sqrt{3}K_3^{\frac12}  \ve^{\frac12} \beta_{k,\ve}\sum_{\ell=1}^{k-1}  \beta_{\ell,\ve} \vp_\ell(x_0) + o\Big(  \ve^{\frac12} \sum_{\ell=1}^{k-1} |\beta_{\ell,\ve}| \Big).
\end{aligned}
\end{equation}
Combining \eqref{test:est2} and \eqref{calculs:31}, and since $\Vert u_\ve \Vert_{6} = 1$ and $|\beta_{k, \ve}| \to 1$, we obtain with \eqref{eq:Jlambdauepx03} that as $\ve \to 0$
\begin{equation} \label{new:eq:3}
\begin{aligned}
F(u_\ve, v_\ve) & =   \sum_{\ell=1}^{k-1} \big(\lambda_\ell - \lambda + o(1) \big) \beta_{\ell,\ve}^2 + J_{\lambda}(\tilde{u}_\ve) + o\Big(  \ve^{\frac12} \sum_{\ell=1}^{k-1} |\beta_{\ell,\ve}| \Big) + o(\ve). \\
\end{aligned} 
\end{equation}
For each $1 \le \ell \le k-1$ fixed, two situations may occur along the subsequence we are considering as $\ve \to 0$: 
\begin{equation*} 
 \text{ either } \quad \frac{|\beta_{\ell,\ve}| }{\ve^{\frac{1}{2}}} \to + \infty \quad \text{ or } \quad  |\beta_{\ell,\ve}| = O(\ve^{\frac{1}{2}}).
 \end{equation*}
We let $0 \le L \le k-1$ be the number of indices $\ell \in \{1, \dots, k-1\}$ satisfying the first case, which can be assumed to be constant up to passing to a subsequence. Up to relabeling the indices we may assume that $\beta_{1, \ve}, \dots, \beta_{L,\ve}$ satisfy the first case and we may thus write that 
\begin{equation*}
 \ve^{-\frac{1}{2}}  \sum_{\ell=1}^L |\beta_{\ell,\ve}| \to + \infty  \quad \text{ and } \quad   \sum_{\ell = L+1}^{k-1} |\beta_{\ell,\ve}| = O (\ve^{\frac{1}{2}})
 \end{equation*}
Young's inequality then shows that 
$$ \ve^{\frac12}  \sum_{\ell=1}^{k-1} | \beta_{\ell,\ve}| =  o \Big( \sum_{i=1}^{L} \beta_{\ell,\ve}^2 \Big) +  O(\ve),  $$
so that \eqref{new:eq:3} becomes 
\begin{equation*}
\begin{aligned}
F(u_\ve, v_\ve) & =   \sum_{\ell=1}^{k-1} \big(\lambda_\ell - \lambda + o(1) \big) \beta_{\ell,\ve}^2 + J_{\lambda}(\tilde{u}_\ve) + o(\ve)  \le J_{\lambda}(\tilde{u}_\ve) + o(\ve) ,
\end{aligned} 
\end{equation*}
where we again used that $\lambda > \Lambda_i$.  Using \eqref{eq:Jlambdauepx03} we finally obtain that
\begin{equation*}
\begin{aligned}
F(u_\ve, v_\ve) & \le  K_3^{-2}  - C_3 m_\lambda(x_0) \ve  + o(\ve) \\
\end{aligned} 
\end{equation*}
as $\ve \to 0$. Since $m_\lambda(x_0) >0$ by assumption this proves \eqref{supve} up to choosing $\ve$ small enough and concludes the proof of Theorem \ref{prop:test:func}.
\end{proof}

\section{Second-order expansion of $\mkl(u)$ at a global minimum point} \label{ordre:deux}

In this section we perform a second-order expansion of the generalised principal eigenvalue at a global minimum point. We obtain a stability inequality of sorts that we use to prove that $\mkl(\Om)$ cannot be attained while being equal to $K_n^{-2}$. We recall that if $u\in L^{2^*}_{>}(\Om)$ attains $\mkl(\Om)$, Proposition \ref{prop:extremales} shows that $E_k^{\lambda}(u)$ is one-dimensional and spanned by a nonzero function $\vp_k$ that satisfies $u = |\vp_k|$ and $\int_\Om u^{2^*-2} \vp_k^2 \, dx = \int_{\Om} u^{2^*} \, dx = 1$. We let in the following $E_{k,0} = E_k^{\lambda}(u)$ and, for any $h \in L^\infty(\Om)$ fixed,  we let $\tilde{\Psi}  \in E_{k,0}^{\perp_{u}}$, where $\perp_u$ denotes the orthogonal complement in $L^2_u(\Om)$, be the unique solution in $H^1_0(\Om)$ of 
\begin{equation} \label{DL:inf:mkl:22}
\begin{aligned}
\big( - \Delta - \lambda - \mkl(\Om) u^{2^*-2} \big)\tilde{\Psi} & =- (2^*-2) \mkl(\Om) \Big( \int_{\Om} u^{2^*}h \, dx\Big) u^{2^*-2} \vp_k  \\
& + (2^*-2) \mkl(\Om) u^{2^*-2}h \vp_k \quad \text{ in } \Om. 
\end{aligned} 
\end{equation}
$\tilde{\Psi}$ is well-defined since the r.h.s of \eqref{DL:inf:mkl:22} integrates to $0$ against $\vp_k$. For simplicity we omit the dependence of $\tilde{\Psi}$ in $h$ in what follows and simply write $\tilde{\Psi}$. The first result that we prove is the following improvement of Proposition \ref{prop:dervp}:

\begin{prop} \label{prop:derseconde}
Let $\Om$ be a smooth bounded domain of $\R^n$, $i \ge 0$, $\lambda \in [\Lambda_i, \Lambda_{i+1})$ and let $k = N(i)+1$ where $N$ is given by \eqref{defN}. Consider $\mkl(\Om)$ defined by \eqref{eq:minmuk} and assume that $\mkl(\Om)$ is attained at some $u \in L^{2^*}_{>}(\Om)$ with $\Vert u \Vert_{2^*} = 1$. Let $h \in L^\infty(\Om)$ and, for $t \ge 0$ small, let $u_t = (1+th) u$ and $\mu_{k,t} = \mkl(u_t)$. Then, as $t \to 0$, we have 
\begin{equation} \label{DL:mkl:2}
\begin{aligned}
\frac{\mu_{k,t}}{\mkl(\Om)} & = 1 - (2^*-2) \Big(\int_{\Om} u^{2^*}hdx \Big) t\\
&  + (2^*-2) \Bigg[ (2^*-2) \Big( \int_{\Om} u^{2^*}h dx\Big)^2 - \frac{2^*-3}{2} \int_{\Om} u^{2^*} h^2 dx  \\
& - \int_{\Om} u^{2^*-2} \tilde{\Psi} \vp_k h \, dx\Bigg] t^2 + o(t^2),
\end{aligned}  
\end{equation}
where $\tilde{\Psi}$ is given by \eqref{DL:inf:mkl:22}.
\end{prop}

Proposition \ref{prop:derseconde} requires $u$ to be a global minimum of $\mkl$, since it heavily relies on the property that $E_{k,0}$ is a line.

\begin{proof}
The proof of \eqref{DL:mkl:2} follows from the same arguments in the proof of \eqref{eq:der9}, except that we push the expansions one order further. We use the notations in the proof of Proposition \ref{prop:dervp}. Let $(t_m)_{m \ge1}, t_m >0$ be a sequence such that $t_m \to 0$ as $m \to + \infty $. For $m \ge 1$ we will let $\mu_{k,t_m} = \mu_k^{\lambda}(u_{t_m})$. Let, for any $m \ge 1$, $\vp_{k,m} \in E_{k,t_m}$ be such that $\int_\Om u_{t_m}^{2^*-2} \vp_{k,m}^2dx = 1$. It satisfies 
$$  - \Delta \vp_{k,m} - \lambda \vp_{k,m} = \mu_{k,t_m} u_{t_m}^{2^*-2} \vp_{k,m} \quad \text{ in } \Om. $$
Step $2$ of Proposition \ref{prop:dervp} shows that $(\vp_{k,m})_{m\ge0}$ is bounded in $H^1_0(\Om)$. Using  \eqref{eq:der2} and since $u_{t_m} \to u$ strongly in $L^{\infty}(\Om)$ and $E_{k,0}$ is spanned by $\vp_k$ we may assume that $\vp_{k,m} \to \vp_k$ in $H^1_0(\Om)$, up to a subsequence and up to sign. For $m \ge 1$ we let 
$$\Psi_{1,m} = \vp_{k,m} - \Pi_{k,0}(\vp_{k,m}).$$
 Using Proposition \ref{prop:extremales} it is easily seen that $\Psi_{1,m}$ satisfies
\begin{equation} \label{DL:inf:mkl:21}
\begin{aligned}
\big( - \Delta &- \lambda - \mkl(\Om) u^{2^*-2} \big)\Psi_{1,m} \\
&  =- (2^*-2) \mkl(\Om) \Big( \int_{\Om} u^{2^*}h \, dx\Big) u^{2^*-2} \vp_k  \cdot t_m \\
& + (2^*-2) \mkl(\Om) u^{2^*-2}h \vp_k \cdot t_m + t_m \eta_m \quad \text{ in } \Om,
\end{aligned} 
\end{equation}
where $\Vert \eta_m \Vert_{\infty} = o(1)$ as $m \to + \infty$. Since $u \in L^\infty(\Om)$, standard elliptic theory shows that there exists $C>0$ such that for any function $f \in E_{k,0}^{\perp_u}$ we have
\begin{equation} \label{DL:inf:mkl:210} 
\Vert f \Vert_{H^1_0} \le C \Vert \big( - \Delta - \lambda - \mkl(\Om) u^{2^*-2} \big)f  \Vert_{H^{-1}}. 
\end{equation}
Since $\Psi_{1,m} \in E_{k,0}^{\perp_{u}}$ for all $m \ge 1$ and by \eqref{DL:inf:mkl:21} we thus obtain first that $\Vert \Psi_{1,m} \Vert_{H^1_0} = O(t_m) $ and then that 
$\tilde{\Psi}_{1,m}: = \frac{\Psi_{1,m}}{t_m}$ weakly converges in $H^1_0(\Om)$, up to a subsequence, to the unique solution $\tilde{\Psi}$ of \eqref{DL:inf:mkl:22}. By \eqref{DL:inf:mkl:22} and  \eqref{DL:inf:mkl:21}, $\tilde{\Psi}_{1,m} - \tilde{\Psi}$ satisfies 
$$\big( - \Delta - \lambda - \mkl(\Om) u^{2^*-2} \big)(\tilde{\Psi}_{1,m} - \tilde{\Psi}) = \eta_m \quad \text{ in } \Om ,$$
so that using again \eqref{DL:inf:mkl:210} we also obtain that 
\begin{equation} \label{DL:inf:mkl:23}
\tilde{\Psi}_{1,m} \to \tilde{\Psi} \quad \text{ in } H^1_0(\Om) \quad \text{ as } m \to + \infty.
\end{equation}
 Independently, since $E_{k,0}$ is one-dimensional by Proposition \ref{prop:extremales} and since $\vp_{k,m} \to \vp_k$ in $H^1_0(\Om)$, there exists a sequence $(\ve_m)_{m \ge1}$, with $\lim_{m \to + \infty} \ve_m = 0$, such that 
$$\Pi_{k,0}(\vp_{k,m}) - \vp_k = \ve_m \vp_k.$$
 We claim that 
\begin{equation} \label{DL:inf:mkl:24}
\begin{aligned}
\ve_m = O(t_m) \quad \text{ as } m \to + \infty. 
\end{aligned} 
\end{equation}
Indeed, using the definition of $\tilde{\Psi}_{1,m}$ and $\ve_m$, we may write 
\begin{equation} \label{splitting_vpk}\begin{aligned}
 \vp_{k,m} & = \vp_{k,m} - \Pi_{k,0}(\vp_{k,m}) +  \Pi_{k,0}(\vp_{k,m}) - \vp_k + \vp_k \\
 &= t_m \tilde{\Psi}_{1,m} + (1+\ve_m) \vp_k .
 \end{aligned}
 \end{equation}
Since $\Vert u_{t_m} - u \Vert_{\infty} = O(t_m)$ and by \eqref{splitting_vpk} we may expand the condition $\int_{\Om} u_{t_m}^{2^*-2} \vp_{k,m}^2 \, dx = 1$ which gives $1  = (1+\ve_m)^2 + O(t_m)$, from which \eqref{DL:inf:mkl:24} follows. For $m \ge 1$ we  let 
$$ \begin{aligned}
\alpha_m &= t_m \big \Vert \tilde{\Psi}_{1,m} - \tilde{\Psi} \big \Vert_{H^1_0} + t_m^2 + t_m |\ve_m| \\
& + \Big| \mu_{k,t_m} - \mkl(\Om) + (2^*-2)\mkl(\Om) \Big( \int_{\Om} u^{2^*}h \, dx\Big) \cdot t_m\Big|
\end{aligned} $$
and 
$$ \tilde{\Psi}_{2,m} = \frac{t_m}{\alpha_m} \Big( \tilde{\Psi}_{1,m} - \tilde{\Psi} \Big) \in E_{k,0}^{\perp_u}. $$
Using the equation satisfied by $\vp_{k,m}$, \eqref{DL:inf:mkl:22} and \eqref{splitting_vpk} shows that $\tilde{\Psi}_{2,m}$ satisfies
\begin{equation} \label{DL:inf:mkl:25}
\begin{aligned}
 \big( - \Delta & - \lambda - \mkl(\Om) u^{2^*-2} \big)\tilde{\Psi}_{2,m}\\
& = \frac{1}{\alpha_m} \Bigg[ \mu_{k,t_m} - \mkl(\Om) + (2^*-2) \mkl(\Om) \Big( \int_{\Om} u^{2^*}h \, dx\Big) t_m \Bigg] u^{2^*-2} \vp_{k,m} \\
& - (2^*-2)^2  \mkl(\Om)\Big( \int_{\Om} u^{2^*}h \, dx\Big) \frac{t_m^2}{\alpha_m} u^{2^*-2} h \vp_{k,m} \\
& + (2^*-2) \mkl(\Om) \Big( \int_{\Om} u^{2^*}h \, dx\Big) \frac{t_m}{\alpha_m} u^{2^*-2} \big(- t_m \tilde{\Psi}_{1,m} - \ve_m \vp_k \big) \\
& + \frac{(2^*-2)(2^*-3)}{2}  \frac{t_m^2}{\alpha_m}u^{2^*-2} h^2 \vp_{k,m} \\
& + (2^*-2) \mkl(\Om)\frac{t_m}{\alpha_m} u^{2^*-2}h  \big(t_m \tilde{\Psi}_{1,m} + \ve_m \vp_k \big)   + \eta_m
\end{aligned} 
\end{equation}
in $\Om$, where again $\eta_m$ satisfies $\Vert \eta_m \Vert_\infty = o(1)$ as $m \to + \infty$. By definition $(\tilde{\Psi}_{2,m})_{m \ge1}$ is bounded in $H^1_0(\Om)$ and thus weakly converges to some $\tilde{\Psi}_2 \in E_{k,0}^{\perp_u}$ in $H^1_0(\Om)$. Passing \eqref{DL:inf:mkl:25} to the limit shows that $\tilde{\Psi}_2$ satisfies 
\begin{equation} \label{DL:inf:mkl:26}
\begin{aligned}
& \big( - \Delta  - \lambda - \mkl(\Om) u^{2^*-2} \big)\tilde{\Psi}_{2} = \Gamma_1 u^{2^*-2} \vp_k \\
& + (2^*-2) \mkl(\Om)  \Gamma_2\Bigg[ \frac{2^*-3}{2}  u^{2^*-2} h^2 \vp_k - (2^*-2) \Big( \int_{\Om} u^{2^*}h \, dx\Big) u^{2^*-2} h \vp_k  \Bigg] \\
& + (2^*-2) \mkl(\Om) \Gamma_3\Bigg[  u^{2^*-2} h \vp_k - \Big( \int_{\Om} u^{2^*}h \, dx\Big)u^{2^*-2} \vp_k \Bigg] \\
& + (2^*-2) \Gamma_2 \mkl(\Om) \Bigg[ u^{2^*-2} h \tilde{\Psi} - \Big( \int_{\Om} u^{2^*}h \, dx\Big)u^{2^*-2} \tilde{\Psi} \Bigg],
\end{aligned} 
\end{equation}
in $\Om$, where $\tilde{\Psi}$ is as in \eqref{DL:inf:mkl:22} and where we have let, up to a subsequence,
$$\begin{aligned}
\Gamma_1 & = \lim_{m \to + \infty} \frac{ \mu_{k,t_m} - \mkl(\Om) + (2^*-2) \mkl(\Om) \Big( \int_{\Om} u^{2^*}h \, dx\Big) t_m }{\alpha_m}, \\
\Gamma_2 & = \lim_{m \to + \infty} \frac{t_m^2}{\alpha_m} \quad \text{ and } \quad \Gamma_3  = \lim_{m \to + \infty} \frac{t_m \ve_m }{\alpha_m}. \\
\end{aligned} $$
Substracting \eqref{DL:inf:mkl:25} and \eqref{DL:inf:mkl:26} and using \eqref{DL:inf:mkl:210} and \eqref{DL:inf:mkl:23} we conclude as before that $\tilde{\Psi}_{2,m}$ strongly converges to $\tilde{\Psi}_2$ in $H^1_0(\Om)$. Integrating \eqref{DL:inf:mkl:26} against $\vp_k$ shows, since $u = |\vp_k|$, $\Vert u \Vert_{2^*} = 1$ and since $\vp_k$ and $\tilde{\Psi}$ are orthogonal in $L^2_u(\Om)$, that 
\begin{equation}  \label{DL:inf:mkl:27} 
\begin{aligned} 
\Gamma_1  = & - (2^*-2) \mkl(\Om) \Gamma_2\Bigg[   \frac{2^*-3}{2}  \int_{\Om} u^{2^*} h^2 \, dx - (2^*-2) \Big( \int_{\Om} u^{2^*}h \, dx\Big)^2 \Bigg]  \\
&- (2^*-2) \mkl(\Om) \Gamma_2 \int_{\Om} u^{2^*-2} \tilde{\Psi} \vp_k h \, dx . 
\end{aligned} 
\end{equation}
We first assume that $\Gamma_2 = 0$. By definition of $\Gamma_2$ this means that $t_m^2  = o(\alpha_m)$ as $m \to + \infty$, and thus implies by \eqref{DL:inf:mkl:24} that we also have $\Gamma_3 = 0$ and then, by \eqref{DL:inf:mkl:27}, that $\Gamma_1 = 0$. Coming back to the definition of $\alpha_m$ we thus have 
$$ \alpha_m = (1+o(1)) t_m \big \Vert \tilde{\Psi}_{1,m} - \tilde{\Psi} \big \Vert_{H^1_0}$$
as $m \to + \infty$, which implies that $\Vert \tilde{\Psi}_{2,m} \Vert_{H^1_0} = 1 + o(1)$. The strong convergence towards $\tilde{\Psi}_2$ then implies that $\tilde{\Psi}_2 \neq 0$. However, since $\Gamma_i = 0$ for $i=1,2,3$, \eqref{DL:inf:mkl:26} now shows that $\tilde{\Psi}_{2} \in E_{k,0}$ which is a contradiction. Therefore $\Gamma_2 \neq 0$ and by \eqref{DL:inf:mkl:27}
\begin{equation} \label{DL:inf:mkl:28}
 \begin{aligned}
\frac{\Gamma_1}{\Gamma_2}  & =  \lim_{m \to + \infty} \frac{\mu_{k,t_m} - \mkl(\Om) + (2^*-2) \mkl(\Om) \Big( \int_{\Om} u^{2^*}h \, dx\Big) t_m}{t_m^2}  \\
 & =  (2^*-2) \mkl(\Om)\Bigg[ (2^*-2) \Big( \int_{\Om} u^{2^*}h dx\Big)^2 - \frac{2^*-3}{2} \int_{\Om} u^{2^*} h^2 dx  \\
& - \int_{\Om} u^{2^*-2} \tilde{\Psi} \vp_k h \, dx\Bigg]  . 
\end{aligned} 
\end{equation}
Since $(t_m)_{m \ge 0}$ was any sequence of positive numbers converging to $0$ and since the right-hand side of \eqref{DL:inf:mkl:28} is independent of the choice of $(t_m)_{m \ge0}$ this shows that 
$$t \mapsto \frac{1}{t^2}\Bigg[ \frac{\mu_{k,t}}{\mkl(\Om)}  - 1 + (2^*-2) \Big(\int_{\Om} u^{2^*}hdx \Big) t \Bigg] $$
has a limit as $t \to 0$ which is given by the right-hand side of \eqref{DL:inf:mkl:28}. This concludes the proof of \eqref{DL:mkl:2}.
\end{proof}

As a consequence we obtain the following result: 

\begin{lemme}
Let $i \ge 0$, $\lambda \in [\Lambda_i, \Lambda_{i+1})$ and let $k = N(i)+1$ where $N$ is given by \eqref{defN}. Assume that $u \in L^{2^*}_{>}(\Om)$, with $\Vert u \Vert_{2^*} = 1$, attains $\mkl(\Om)$. Then, for any $h \in L^\infty(\Om)$, we have 
\begin{equation} \label{DL:lambda:ordre2:1}
\int_{\Om} u^{2^*} h^2 \, dx - \left( \int_{\Om} u^{2^*} h \, dx \right)^2 \ge \int_{\Om} u^{2^*-2} \tilde{\Psi} \vp_k h \, dx, 
\end{equation}
where $\tilde{\Psi}$ is as in \eqref{DL:inf:mkl:22}.
\end{lemme}

\begin{proof}
Let as before $u_t = u(1+th)$. Straightforward computations show that 
\begin{equation} \label{DL:lambda:ordre2:2}
 \begin{aligned}
\Vert u_t \Vert_{2^*}^{2^*-2} & = 1 + (2^*-2)  \int_{\Om} u^{2^*}h \, dx t \\
& + \Bigg[ \frac{(2^*-1)(2^*-2)}{2} \int_{\Om} u^{2^*} h^2 \, dx - (2^*-2) \left(\int_{\Om}u^{2^*} h \, dx \right)^2 \Bigg] t^2 + o(t^2) 
\end{aligned} 
\end{equation}
as $t \to 0$. Since $u$ attains $\mkl(\Om)$ we have 
$$ \mkl(\Om) = \mkl(u) \Vert u \Vert_{2^*}^{2^*-2} \le \mkl(u_t)\Vert u_t \Vert_{2^*}^{2^*-2}  $$
for any $t$ small enough. We may now expand the right-hand side of the latter inequality using \eqref{DL:mkl:2} and \eqref{DL:lambda:ordre2:2}: straightforward computations show that the $0$-th and first-order terms in $t$ vanish. Dividing by $t^2$ and letting $t \to 0$ yields \eqref{DL:lambda:ordre2:1}. 
\end{proof}

The following result establishes a stability inequality of sorts for the functional $u \in L^{2^*}_{>}(\Om) \mapsto \mkl(u) \Vert u \Vert_{2^*}^{2^*-2}$ at a minimiser:

\begin{prop} \label{prop:DL:lambda:ordre:2}
Let $i \ge 0$,  $\lambda \in [\Lambda_i, \Lambda_{i+1})$ and $k = N(i)+1$ where $N$ is given by \eqref{defN}. Assume that $u \in L^{2^*}_{>}(\Om)$, with $\Vert u \Vert_{2^*} = 1$, attains $\mkl(\Om)$. Then, for any $h \in C^\infty(\overline{\Om})$, we have 
\begin{equation} \label{DL:lambda:ordre2:3}
\begin{aligned}
\frac{1}{\mkl(\Om)} \int_{\Om} u^2 |\nabla h |^2 \, dx & + (2^*-2) \left( \int_{\Om} u^{2^*} h \, dx \right)^2 \\ 
& \ge (2^*-2) \int_{\Om} u^{2^*} h^2 \, dx.
\end{aligned} 
\end{equation}
\end{prop}

Proposition \ref{prop:DL:lambda:ordre:2} is, to the best of our knowledge, the first instance of a stability inequality for second variations of generalised eigenvalue functionals $u \mapsto \mkl(u)$. 

\begin{proof}
Assume that $u \in L^{2^*}_{>}(\Om)$, $\Vert u \Vert_{2^*} = 1$, attains $\mkl(\Om)$. Let $h \in C^\infty(\overline{\Om})$. For $|t|< \Vert h \Vert_{\infty}^{-1}$ we let as before $u_t = u(1+ th)$. Then $u_t \in L^{2^*}_{>}(\Om)$,  $\mu_p^{\lambda}(u_t)$ is well-defined for any $p \ge 1$ and by Proposition \ref{prop:signvp} we have $\mkl(u_t)>0$. We again adopt the notation $\mu_{p,t} = \mpl(u_t)$ and $E_{p,t} = E_p^{\lambda}(u_t)$ for any $p \ge1$. For $1 \le p \le k-1$ we will denote by $\Pi_{p,t}$ the orthogonal projection over $E_{p,t}$ in $L^2_{u_t}(\Om)$, and we will denote by  $\Pi_t$ the orthogonal projection over $\oplus_{i=1}^{k-1} E_{i,t}$ in $L^2_{u_t}(\Om)$, so that $\Pi_t = \sum_{p=1}^{k-1} \Pi_{p,t}$. Using \eqref{eq:der3} we have 
\begin{equation*}
\Vert \Pi_t(\vp_k) \Vert_{L^2_{u_t}} = O(t) 
\end{equation*}
as  $t \to 0$. 
We let, for any $p \ge 1$,  $\vp_p = \vp_p^{\lambda}(u)$ be the $L^2_u(\Om)$-orthonormal family of generalised eigenfunctions associated to $u$ and given by Proposition \ref{prop:vp}. Recall that $\mpl(u) \le 0$ for $p \le k-1$. By Proposition \ref{prop:extremales}, $E_{k,0} = E_k^{\lambda}(u)$ is one-dimensional, is spanned by $\vp_k$ and $u = |\vp_k|$. In what follows we let, for $t$ small enough,
$$ \Psi_{k,t} = \frac{(1+th) \vp_k}{\Vert (1+th) \vp_k \Vert_{L^2_{u_t}}} .$$
We define inductively, for $2 \le i \le k$ and $t$ small enough,
$$ \Psi_{i-1,t} = \vp_{i-1} - \sum_{i-1 \le \ell \le k-1} \big( \vp_{\ell}, \Psi_{\ell+1,t} \big)_{L^2_{u_t}} \Psi_{\ell,t}. $$
The family $(\Psi_{1,t}, \dots, \Psi_{k,t})$ obtained in this way is therefore $L^2_{u_t}$-orthonormal and 
$$ V_t = \text{Span} \big( \Psi_{1,t}, \dots \Psi_{k,t} \big) $$
satisfies $\dim_{u_t} V_t = k$. By construction of the family $(\Psi_{i,t})_{1 \le i \le k}$ it is easily proven that, for all $1 \le i \le k$, we have $\Vert \Psi_{i,t} - \vp_i \Vert_{H^1_0} = O(t)$ as $t \to 0$. As a consequence
$$ Q_{u_t}^{\lambda}( \Psi_{i,t}) = Q_u^{\lambda}(\vp_i) + O(t) = \mu_i^{\lambda}(u) + O(t)\le O(t) \quad \text{ as } t \to 0$$
for $1 \le i \le k-1$, where $Q_{u_t}^{\lambda}$ is as in \eqref{defQu}, and 
$$ Q_{u_t}^{\lambda}( \Psi_{k,t}) = Q_u^{\lambda}(\vp_k) + O(t) = \mu_k^{\lambda}(\Om) + O(t) \quad \text{ as } t \to 0.$$
Since the family $(\Psi_{1,t}, \dots, \Psi_{k,t})$ is $L^2_{u_t}$-orthonormal we have 
$$ \max_{v \in V_t} Q_{u_t}^{\lambda}(v) = \max_{1 \le i \le k} Q_{u_t}^{\lambda}(\Psi_{i,t}) =Q_{u_t}^{\lambda}(\Psi_{k,t}),$$
and by the definition \eqref{muk} of $\mu_{k,t}$ and the expression of $\Psi_{k,t}$ we finally obtain
\begin{equation} \label{dersconde:001}
\mu_{k,t} \le \frac{ \int_{\Om} \big|\nabla \big( (1+th) \vp_k \big)  \big|^2 - \lambda (1+th)^2 \vp_k^2 \, dx}{\int_{\Om} (1+th)^{2^*} u^{2^*} \, dx},
\end{equation}
where we again used that $u = |\vp_k|$. We easily have
\begin{equation} \label{dersconde:002}
  \int_{\Om} (1+th)^{2^*} u^{2^*} \, dx = 1 + 2^* \int_{\Om} u^{2^*}h\, dx \cdot t   + \frac{2^*(2^*-1)}{2} \int_{\Om} u^{2^*} h^2 \, dx \cdot  t^2 + o(t^2)
  \end{equation}
 as $t \to 0$. Straightforward computations show that 
   \begin{equation} \label{dersconde:008}
\begin{aligned}
& \int_{\Om}  \big|\nabla \big( (1+th) \vp_k \big)  \big|^2 - \lambda (1+th)^2 \vp_k^2 \, dx \\
 & = \mkl(\Om) + \Bigg[ 2 \int_{\Om} h |\nabla \vp_k|^2 \, dx - 2 \lambda \int_{\Om} h \vp_k^2 \, dx + 2 \int_{\Om} \vp_k \langle \nabla h, \nabla \vp_k \rangle \, dx \Bigg] \cdot t \\
 & + \Bigg[ \int_{\Om} u^2 |\nabla h|^2 \, dx + 2 \int_\Om \vp_k h \langle \nabla h, \nabla \vp_k \rangle \, dx + \int_{\Om} h^2 |\nabla \vp_k|^2 \, dx \\
 & - \lambda \int_{\Om} h^2 \vp_k^2 \, dx \Bigg] \cdot t^2.
 \end{aligned} 
 \end{equation}
 Integrating the equation $- \Delta \vp_k - \lambda \vp_k = \mkl(\Om) u^{2^*-2} \vp_k$ against $h \vp_k$ and $h^2 \vp_k$ gives, respectively,
  $$ \begin{aligned}
 \mkl(\Om) \int_{\Om} u^{2^*} h \, dx = \int_{\Om} h |\nabla \vp_k|^2 \, dx - \lambda \int_{\Om} h \vp_k^2 \, dx + \int_{\Om} \vp_k \langle \nabla h, \nabla \vp_k \rangle \, dx
 \end{aligned} $$
 and 
  $$ \begin{aligned}
 \mkl(\Om) \int_{\Om} u^{2^*} h^2 \, dx = \int_{\Om} h^2 |\nabla \vp_k|^2 \, dx + 2 \int_\Om \vp_k h \langle \nabla h, \nabla \vp_k \rangle \, dx - \lambda \int_{\Om} h^2 \vp_k^2 \, dx . \end{aligned} $$
Plugging the latter in \eqref{dersconde:008} finally gives 
   \begin{equation} \label{dersconde:009}
\begin{aligned}
& \int_{\Om}  \big|\nabla \big( (1+th) \vp_k \big)  \big|^2 - \lambda (1+th)^2 \vp_k^2 \, dx = \mkl(\Om) \\
&+2 \mkl(\Om) \int_\Om u^{2^*} h \, dx \cdot t  + \Bigg[ \int_{\Om} u^2 |\nabla h|^2 \, dx + \mkl(\Om)  \int_{\Om} u^{2^*} h^2 \, dx  \Bigg] t^2.
 \end{aligned} 
 \end{equation}
We now combine \eqref{DL:mkl:2}, \eqref{dersconde:001}, \eqref{dersconde:002} and \eqref{dersconde:009}. It is easily seen that the $0$-th and first-order terms in $t$ in \eqref{dersconde:001} cancel. By dividing the result by $t^2$ and letting $t \to 0$ we find, after direct computations, that 
   \begin{equation} \label{dersconde:10}
\begin{aligned}
& \int_{\Om} u^2 |\nabla h|^2 \, dx + 2(2^*-2) \mkl(\Om) \left( \int_\Om u^{2^*} h \, dx \right)^2 \\
& \ge 2 (2^*-2) \mkl(\Om) \int_{\Om} u^{2^*} h^2 \, dx - (2^*-2) \mkl(\Om) \int_{\Om} u^{2^*-2} \tilde{\Psi} h \vp_k \, dx \\
 \end{aligned} 
 \end{equation}
holds true. Using \eqref{DL:lambda:ordre2:1} finally concludes the proof of \eqref{DL:lambda:ordre2:3}. 
 \end{proof}

As a consequence of Proposition \ref{prop:DL:lambda:ordre:2} we finally prove that $\mkl(\Om)$ cannot be simultaneously attained and equal to $K_n^{-2}$ or, in other words, that \eqref{eq:test1} is a necessary condition for the attainment of $\mkl(\Om)$:

\begin{prop} \label{prop:optimalite:Kn}
Let $i \ge 0$, $\lambda \in [\Lambda_i, \Lambda_{i+1})$ and let $k = N(i)+1$ where $N$ is given by \eqref{defN}. Then $\mkl(\Om)$ is attained if an only if $\mkl(\Om) < K_n^{-2}$. 
\end{prop}

\begin{proof}
The right implication is given by Theorem \ref{prop:minmkl1} and we only prove the left implication. Since $\mkl(\Om) \le K_n^{-2}$ by Proposition \ref{prop:ineq:large:00} we only have to prove that if $\mkl(\Om) = K_n^{-2}$ then $\mkl(\Om)$ is not attained. We proceed by contradiction and assume that there exists $u \in L^{2^*}_{>}(\Om)$, with $\Vert u \Vert_{2^*} = 1$, which attains $\mkl(\Om) = K_n^{-2}$. Proposition \ref{prop:DL:lambda:ordre:2} then applies and \eqref{DL:lambda:ordre2:3} shows that, for any $h \in C^\infty(\overline{\Om})$, we have 
\begin{equation}  \label{DL:lambda:ordre2:3bis}
\begin{aligned}
K_n^2 \int_{\Om} u^2 |\nabla h |^2 \, dx & + (2^*-2) \left( \int_{\Om} u^{2^*} h \, dx \right)^2 \ge (2^*-2) \int_{\Om} u^{2^*} h^2 \, dx.
\end{aligned} 
\end{equation}
To obtain a contradiction we adapt an argument of  \cite{Druetdim3}. For $(t,y) \in (0, + \infty) \times \R^n$ we define 
$$ Z^0_{t,y}(x) = \frac{1 -t^2|x-y|^2}{1+t^2|x-y|^2} \quad \text{ and } \quad Z^i_{t,y}(x)=  \frac{2t (x_i-y_i)}{1+t^2|x-y|^2}, i=1 \dots n.$$
It is easily seen that, for any $x \in \Om$, we have 
$$\sum_{i=0}^n (Z^i_{t,y}(x))^2 = 1 \quad \text{ and  }  \quad \sum_{i=0}^n \left| \nabla Z^i_{t,y}(x) \right|^2 = \frac{4nt^2}{(1+t^2|x-y|^2)^2}.$$
The arguments in \cite{Druetdim3} show that  there exist $(t,y) \in (0, + \infty) \times \R^n$ such that 
\begin{equation} \label{ineq.contra.druet}
\int_{\Om} Z^i_{t,y} u^{2^*} \, dx = 0 \quad \text{ for all } i=0 \dots n.   
\end{equation}
Applying \eqref{DL:lambda:ordre2:3bis} and \eqref{ineq.contra.druet} to each $Z^i_{t,y}$, $i=0 \dots n$, and summing over $i$ then yields 
$$  (2^*-2) \int_{\Om} u^{2^*} \, dx \le K_n^2 \int_{\Om} \frac{4nt^2}{(1+t^2|x-y|^2)^2} u^2 \, dx . $$
By H\"older's inequality and since $\Vert u \Vert_{2^*} = 1$ we obtain that 
\begin{equation} \label{eq:contra:finale:druet}
\begin{aligned} 
1 & \le n(n-2) K_n^2 \left(  \int_{\Om} \frac{t^n}{(1+t^2|x-y|^2)^{n}} \, dx \right)^{\frac{2}{n}} \\
& < n(n-2) K_n^2 \left(  \int_{\R^n} \frac{t^n}{(1+t^2|x-y|^2)^{n}} \, dx \right)^{\frac{2}{n}},
\end{aligned} 
 \end{equation}
 where the strict inequality follows since $\Om$ is bounded. A simple change of variables using \eqref{energieU0} finally shows that 
$$  \int_{\R^n} \frac{t^n}{(1+t^2|x-y|^2)^{n}} \, dx = \big( n(n-2) \big)^{-\frac{n}{2}} K_n^{-n}, $$
which contradicts \eqref{eq:contra:finale:druet}. 
\end{proof}

\section{Proof of the main results} \label{sec:proof:main:results}

We prove in this section Theorems  \ref{theoreme:n:3} and \ref{theoreme:n:4plus}. We start with the higher-dimensional case:

\begin{proof}[Proof of Theorem \ref{theoreme:n:4plus}]
Let $\lambda >0$ and $i \ge 0$ be such that $\lambda \in [\Lambda_i, \Lambda_{i+1})$. If $n=4$ we assume in addition that $\lambda \in (\Lambda_i, \Lambda_{i+1})$. Under these assumptions, Theorem \ref{prop:test:func} shows that $\mkl(\Om) < K_n^{-2}$ and, as a consequence, Theorem \ref{prop:minmkl1} shows that $\mkl(\Om)$ is attained at a least-energy sign-changing solution of \eqref{eq:critlambda} of energy $\mkl(\Om)^{\frac{n}{2}}$. Proposition \ref{prop:liens}  then implies that  $\mkl(\Om)^{\frac{n}{2}} = \mathcal{E}_{\lambda}(\Om)$. With the latter, Theorem  \ref{theoreme:n:4plus} simply follows from Proposition \ref{prop:C0:mkl}. 
\end{proof}

We now consider the three-dimensional case and prove Theorem \ref{theoreme:n:3}.

\begin{proof}[Proof of Theorem \ref{theoreme:n:3}] Let $i \ge 1$. Throughout the proof we will let $k = N(i) +1$ where $N$ is as in \eqref{defN}. We recall that $\lambda_{i,*}$ and $\bar{\lambda}_i$ are as in \eqref{deflambdaistar} and  \eqref{def:lambdai:min} and we recall that Proposition \ref{prop:massepositive} below ensures that $\lambda_{i,*} < \Lambda_{i+1}$. Let $\lambda \in (\lambda_{i,*}, \Lambda_{i+1})$. By definition of $\lambda_{i,*}$ we have $\max_{\Om} m_\lambda >0$: Theorem \ref{prop:test:func} shows that $\mu_k^{\lambda}(\Om) < K_3^{-2}$, and Theorem \ref{prop:minmkl1} then shows that $\mu_k^{\lambda}(\Om)$ is attained. Proposition \ref{prop:liens}  then implies that $\mkl(\Om)^{\frac32} = \mathcal{E}_{\lambda}(\Om)$ and that $\mathcal{E}_\lambda(\Om)$ is also attained, hence that $\bar{\lambda}_i \le \lambda$. Taking the infimum over $\lambda$ yields $\bar{\lambda}_i \le \lambda_{i,*}$, which proves $(1)$.

We assume that $\bar{\lambda}_i = \Lambda_i$. Let $\lambda \in (\Lambda_i, \Lambda_{i+1})$. By the definition of $\bar{\lambda}_i$ there is $\nu \in (\Lambda_i, \lambda)$ such that $- \infty < \mathcal{E}_\nu(\Om) < K_3^{-3}$. By definition of $\mathcal{E}_\nu(\Om)$ there is a non-zero solution $u$ of \eqref{eq:critlambda} (with $\lambda = \nu$) such that $\int_{\Om} u^6 \,dx < K_3^{-3}$, and by Proposition \ref{prop:liens} we then have $\mu_k^{\nu}(\Om)^{\frac32} =  \mathcal{E}_\nu(\Om)< K_3^{-2}$, and $\mu_k^{\nu}(\Om)$ (and hence $\mathcal{E}_\nu(\Om)$) is attained. By Proposition \ref{prop:C0:mkl} we then have $\mu_k^{\nu'}(\Om) < K_3^{-2}$ for every $\nu' \in [\nu, \Lambda_{i+1})$ and hence Theorem \ref{prop:minmkl1} and then Proposition \ref{prop:liens} again show that $\mathcal{E}_{\nu'}(\Om)$ is attained and satisfies $\mathcal{E}_{\nu'}(\Om) < K_3^{-3}$. In particular $\mathcal{E}_\lambda(\Om) <K_3^{-3}$ and it is attained, and this proves (2).

We now prove (3) and (4). We assume from now on and until the end of this proof that $\bar{\lambda}_i \in (\Lambda_i, \Lambda_{i+1})$. We first claim that the following holds: 
\ben \label{argument:final:00}
\bar{\lambda}_i = \lambda_{i,*}.
\een
In view of the previous paragraph, we only have to prove that $\bar{\lambda}_i \ge \lambda_{i,*}$. By Proposition \ref{prop:liens} and by definition of $\bar{\lambda}_i$, for any $\lambda \in (\bar{\lambda}_i, \Lambda_{i+1})$ we have $\mkl(\Om) < K_3^{-2}$ and  $\mu_k^{\lambda}(\Om)$ is attained  at a least-energy solution of \eqref{eq:critlambda}, denoted by $u_\lambda$, of energy $\int_{\Om} u_\lambda^6 \, dx  = \mkl(\Om)^{\frac{3}{2}} < K_3^{-3}$. By definition of $\bar{\lambda}_i$, and since we assumed $\bar{\lambda_i} > \Lambda_i$ we have, by Proposition \ref{prop:C0:mkl}, 
\ben \label{argument:final:0}
\int_{\Om} u_\lambda^6 \, dx = \mkl(\Om)^{\frac32} \to \mu_k^{\bar{\lambda}_i}(\Om)^{\frac32} =  K_3^{-3} \quad  \text{ from below as } \quad \lambda \underset{> }{\to} \bar{\lambda}_{i}.
\een 
 We first claim that 
\ben \label{argument:final:1}
 \Vert u_\lambda \Vert_{\infty} \to + \infty 
 \een
as $\lambda \underset{> }{\to} \bar{\lambda}_{i}$ from above. Indeed, if $u_\lambda$ were bounded in $L^\infty(\Om)$ \eqref{eq:critlambda} and standard elliptic theory would show that $u_\lambda$ converges in $C^2(\overline{\Om})$, up to a subsequence as $\lambda \to \bar{\lambda}_i$, to a solution $u_0$ of \eqref{eq:critlambda}. By \eqref{argument:final:0}, $u_0$ would attain $\mu_k^{\bar{\lambda}_i}(\Om)$ and we would have $\mu_k^{\bar{\lambda}_i}(\Om) =  K_3^{-2}$, which would contradict Proposition \ref{prop:optimalite:Kn}. Therefore \eqref{argument:final:1} is true, and with  \eqref{argument:final:0} and \eqref{argument:final:1} we may now apply Proposition \ref{analyse:asympto:onebubble} below, which shows the existence of $x_0 \in \Om$ such that $m_{\bar{\lambda}_i}(x_0) = 0$. Since $\lambda \in (\Lambda_i, \Lambda_{i+1}) \mapsto m_\lambda(x_0)$ is increasing by \eqref{der:mass:lambda} below we have $\max_{\Om} m_{\bar{\lambda}_i + \ve} >0$ for $\ve >0$ small enough. By the definition of $\lambda_{i,*}$ we have $\lambda_{i,*} \le \bar{\lambda}_i + \ve$ and thus $\lambda_{i,*} \le \bar{\lambda}_i$. This proves \eqref{argument:final:00}. Since $\lambda \in [\Lambda_i, \Lambda_{i+1}) \mapsto \mkl(\Om)$ is continuous by Proposition \ref{prop:C0:mkl}, the definition of $\bar{\lambda}_i$ shows that $\mkl(\Om) = K_3^{-2}$ for all $\lambda \in [\Lambda_i, \bar{\lambda}_i]$, and Proposition \ref{prop:optimalite:Kn} therefore shows that $\mkl(\Om)$ is not attained for $\lambda$ in this range. We have thus shown that  $\mkl(\Om)$ is attained if and only if $ \lambda \in (\lambda_{i,*}, \Lambda_{i+1})$. For these values of $\lambda$, Proposition \ref{prop:liens} implies that  $\mkl(\Om)^{\frac32} = \mathcal{E}_{\lambda}(\Om)$ and that $\mathcal{E}_\lambda(\Om)$ is attained at a least-energy sign-changing solution of \eqref{eq:critlambda}. The regularity of $\lambda \mapsto \mathcal{E}_\lambda(\Om)$ again follows from Proposition \ref{prop:C0:mkl}. This proves (3). 

To conclude the proof of Theorem \ref{theoreme:n:3} we finally assume that $\bar{\lambda}_i > \Lambda_i$ and assume that there is $\lambda \in [\Lambda_i, \lambda_{i,*}]$ for which \eqref{eq:critlambda} has a non-zero solution. We proved in point (3) of Theorem \ref{theoreme:n:3} that $\mkl(\Om)$ is not attained, hence $\mkl(\Om) = K_3^{-2}$ by Theorem \ref{prop:test:func}. As a consequence, Proposition \ref{prop:liens} implies that $\mathcal{E}_\lambda(\Om) \ge K_3^{-3}$. We assume that $\mathcal{E}_\lambda(\Om) = K_3^{-3}$ and we then claim that $\mathcal{E}_\lambda(\Om)$ is not attained. We proceed by contradiction and assume that $\mathcal{E}_\lambda(\Om)$ is attained at a solution $u$ of \eqref{eq:critlambda}. We thus have $\int_{\Om} u^6 \, dx = K_3^{-3}$ and Proposition \ref{prop:structure:vp} shows that $K_3^{-2} = \mkl(\Om) \le \mkl(\frac{u}{\Vert u \Vert_{2^*}}) \le \big( \int_{\Om} u^6 \, dx \big)^{\frac23} = K_3^{-2}$. This shows that  $u$ attains $\mkl(\Om) = K_3^{-2}$, which contradicts Proposition \ref{prop:optimalite:Kn}. Thus, $\mathcal{E}_\lambda(\Om) = K_3^{-3}$ is not attained. By definition of  $\mathcal{E}_\lambda(\Om)$ there is a sequence $(u_k)_{k \ge0}$ of solutions of \eqref{eq:critlambda} such that $\int_{\Om} u_k^6\, dx \to \mathcal{E}_\lambda(\Om)$. This sequence cannot have a converging subsequence in $H^1_0(\Om)$, otherwise its limit would attain $\mathcal{E}_\lambda(\Om)$. Thus $(u_k)_{k \ge0}$ blows-up and the set of solutions of \eqref{eq:critlambda} is not compact in $H^1_0(\Om)$. This concludes the proof of (4). 
\end{proof}

Let us conclude with a few remarks on Theorem \ref{theoreme:n:3} and on the values of $\lambda_{i,*}$ and $\bar{\lambda}_i$. Estimating them for $i \ge 1$ turns out to be rather difficult. For $\lambda_{i,*}$, this is due to the behavior of the mass function $m_\lambda$: in the non-coercive case $\lambda \in (\Lambda_i, \Lambda_{i+1})$ for $i \ge 1$,  $m_\lambda$ displays richer properties than in the coercive case. We prove in Proposition \ref{prop:comportement:masse} below, for instance, that the asymptotic behavior of $m_\lambda$ as $\lambda \to \Lambda_i$ from the right is related to the geometry of the nodal sets of eigenfunctions of $E_{\Lambda_i}(\Om)$ and to the mass of the unique Green's function of $- \Delta - \Lambda_i$, both of which are not easily understood. In particular, Proposition \ref{prop:comportement:masse} below hints at the possibility of having $\lambda_{i,*} = \Lambda_i$ for some $i \ge 1$ and some smooth bounded domain $\Om$, something that would be in striking contrast with the coercive case. Concerning $\bar{\lambda}_i$, the rearrangement argument used in \cite{BN} to show that $\bar{\lambda}_0 >0$ turns out to be ineffective when $i \ge 1$, and proving that $\bar{\lambda}_i > \Lambda_i$ would require new techniques to obtain non-existence results for sign-changing solutions of \eqref{eq:critlambda} when $\lambda \ge \Lambda_1$. In Proposition \ref{prop:massenegative} below we could only prove that in the specific case $i=1$ we have $\lambda_{1,*} \in (\Lambda_1, \Lambda_2)$ (see \eqref{lambda:star:1} below), but without being able to estimate $\bar{\lambda}_1$.

\medskip

We conclude this section by proving the following result which establishes that, as $\lambda \to \Sp(-\Delta)$ from the left, the sign-changing least-energy solutions of \eqref{eq:critlambda} bifurcate from $E_{\Lambda_{i+1}}(\Om)$:

\begin{prop} \label{prop:bifurcation}
Let $\Om$ be a smooth bounded domain of $\R^n$, $n \ge 3$. Let $i \ge 1$ be fixed and let $u_\lambda$, for $\lambda$ close enough to $\Lambda_{i+1}$, be a solution of \eqref{eq:critlambda} attaining $\mathcal{E}_{\lambda}(\Om)$. Then $\frac{u_\lambda }{\Vert u_\lambda \Vert_{\infty}}$ converges, up to a subsequence, to a non-zero element of $E_{\Lambda_{i+1}}(\Om)$ as $\lambda \to \Lambda_{i+1}$ from the left. 
\end{prop}

\begin{proof}
By Theorems \ref{theoreme:n:3} and \ref{theoreme:n:4plus}, $\int_{\Om} |u_\lambda|^{2^*}dx \to 0$ as $\lambda \to \Lambda_{i+1}$ from the left. Since $u_\lambda$ solves \eqref{eq:critlambda}, we also obtain that $\Vert u_\lambda \Vert_{H^1_0} \to 0$ as $\lambda \to \Lambda_{i+1}$. The $H^1_0(\Om)$-compactness result of Struwe \cite{Struwe} then shows that $\Vert u_\lambda \Vert_{\infty} \to 0$ as $\lambda \to \Lambda_{i+1}$ from the left. The convergence of  $\frac{u_\lambda }{\Vert u_\lambda \Vert_{\infty}}$ towards a non-zero element of $E_{\Lambda_{i+1}}(\Om)$ then follows from standard elliptic theory.
\end{proof}
Proposition \ref{prop:bifurcation} thus hints that least-energy solutions of \eqref{eq:critlambda} belong to one of the bifurcating families in constructed in  Cerami-Fortunato-Struwe \cite{CeramiFortunatoStruwe}.

\appendix

\section{Minimisation of $\mu_1^{\lambda}(\Om)$ in the coercive case } \label{app:m1}

Let $\Om \subset \R^n, n \ge 3$ be an open bounded set. For $\lambda \in [0, \Lambda_1]$ we define 
\begin{equation} \label{defIl}
I_\lambda(\Om) = \inf_{v \in H^1_0(\Om)\backslash \{0\}} \frac{\int_{\Om} \big( |\nabla v|^2 - \lambda v^2 \big)dx}{\big(\int_{\Om}|v|^{2^*}dx \big)^{\frac{n-2}{n}}}.
\end{equation}
In the celebrated paper \cite{BN}, based on previous work by Aubin \cite{AubinYamabe}, Br\'ezis and Nirenberg investigated when $I_\lambda(\Om)$ is attained. We consider the following minimisation problem, which is exactly \eqref{eq:minmuk} in the coercive case $ \lambda \in [0, \Lambda_1)$, that is when $k=1$:
 \begin{equation} \label{eq:minmu11}
\mu_1^{\lambda}(\Om) = \inf_{u \in L^{2^*}_>(\Om)} \mu_1^{\lambda}(u) \Vert u \Vert_{2^*}^{2^*-2} 
\end{equation}
where $L^{2^*}_{>}(\Om)$ is as in \eqref{defLu} and $\mu_1^{\lambda}(u)$ is as in \eqref{muk}. Since $- \Delta -\lambda$ is coercive for $\lambda \in [0, \Lambda_1)$ we have $\mu_1^{\lambda}(\Om), I_\lambda(\Om) \ge 0$. We claim that \eqref{defIl} and \eqref{eq:minmu11}  are equivalent:

\begin{lemme} \label{lemme:muI}
Let $\Om$ be a smooth open bounded set of $\R^n, n \ge 3$. For any $\lambda \in [0, \Lambda_1)$
\begin{equation*}
\mu_1^{\lambda}(\Om) = I_\lambda(\Om)
\end{equation*}
holds. In addition $I_\lambda(\Om)$ is attained if and only if $\mu_1^{\lambda}(\Om)$ is. 
\end{lemme}
\begin{proof}
We assume that $\lambda \in [0, \Lambda_1)$. Let $u \in L^{2^*}_{>}(\Om)$, with $\Vert u \Vert_{2^*} = 1$. We use the equivalent caracterisation \eqref{mu1eq} for $\mu_1^{\lambda}(u)$. By H\"older's inequality, for any $v \in H^1_0(\Om)$, we have
$$ \int_{\Om} u^{2^*-2}v^2dx \le \Big(\int_{\Om} |v|^{2^*}dx \Big)^{\frac{2}{2^*}},$$
and thus $\mu_1^{\lambda}(u) \ge I_\lambda(\Om)$, which shows that $\mu_1^{\lambda}(\Om) \ge I_\lambda(\Om)$. Independently, by choosing $v = u$ in \eqref{mu1eq} we obtain $\mu_1^\lambda(u) \Vert u \Vert_{2^*}^{2^*-2} \le I_\lambda(\Om)$ and hence $\mu_1^{\lambda}(\Om) \le I_\lambda(\Om)$.

Assume first that $I_\lambda(\Om)$ is attained at some function $u \in H^1_0(\Om) \cap C^\infty(\overline{\Om})$, with $u >0$ in $\Om$ and $\Vert u \Vert_{2^*}=1$. Standard variational arguments show that 
$$ - \triangle u - \lambda u = I_\lambda(\Om) u^{2^*-1} \quad \text{ in } \Om. $$
Choosing $u$ and $v = u$ in \eqref{mu1eq} shows that $\mu_1^{\lambda}(u) \le I_\lambda(\Om)= \mu_1^{\lambda}(\Om)$, and hence $u$ attains $\mu_1^{\lambda}(\Om)$. Assume now that $\mu_1^\lambda(\Om)$ is attained at some function $u \in L^{2^*}_{>}(\Om)$ with $\Vert u\Vert_{2^*}=1$. We let $\vp_1^{\lambda}(u)$ be the first normalised eigenvector given by Proposition \ref{prop:vp1}. Choosing $u = v = \vp_1^{\lambda}(u)$ in the definition of $\mu_1^\lambda(\Om)$ gives
$$ \bal
\mu_1^{\lambda}(\Om) & \le \frac{\int_\Om |\nabla \vp_1^{\lambda}(u)|^2 - \lambda (\vp_1^{\lambda}(u))^2 dx }{\int_{\Om} ( \vp_1^{\lambda}(u))^{2^*}dx} \Vert \vp_1^{\lambda}(u) \Vert_{2^*}^{2^*-2} \\
& = \mu_1^{\lambda}(\Om) \int_{\Om} u^{2^*-2}  \vp_1^{\lambda}(u)^2 \, dx \le  \mu_1^{\lambda}(\Om)
\eal $$
where the last inequality follows from H\"older's inequality. The equality case in H\"older's inequality then shows that $u_0 = c \vp_1^{\lambda}(u_0)$ a. e. for some $c >0$. The normalisation of $\vp_1^{\lambda}(u_0)$ then implies $c=\pm1$, and hence $ \vp = \vp_{1}^\lambda(u_0) \neq 0$ solves
$$ - \triangle \vp - \lambda \vp  = \mu_1^{\lambda}(\Om) |\vp|^{2^*-2} \vp \quad \text{ in } \Om. $$
Since $I_\lambda(\Om) = \mu_1^{\lambda}(\Om)$, $\vp_1^{\lambda}(u_0)$ then attains $I_\lambda(\Om)$. 
\end{proof}

\section{Asymptotic analysis for sign-changing solutions of \eqref{eq:critlambda} at the least energy level} \label{app:m2}

In this section we prove an asymptotic analysis result that was used in the proof of Theorem \ref{theoreme:n:3}. Let $n \ge 3$. For $z \in \Om$ and $\mu >0$ we let, for $x \in \Om$, 
$$ B_{z,\mu}(x) = \frac{\mu^{\frac{n-2}{2}}}{\left( \mu^2 + \frac{|x-z|^2}{n(n-2)} \right)^{\frac{n-2}{2}}} = \mu^{- \frac{n-2}{2}} U_0 \Big( \frac{x-z}{\mu}\Big), $$
where $U_0$ is as in \eqref{defU0}. For $z \in \Om$ and $\mu >0$ we let $P B_{z,\mu}$ be the unique solution in $H^1_0(\Om)$ of 
\begin{equation} \label{projbulle:l}
\left \{ \begin{aligned}
-\Delta P B_{z,\mu}  & = B_{z,\mu}^{2^*-1} \quad \text{ in } \Om \\
P B_{z,\mu} & = 0 \quad \text{ on } \partial \Om.
\end{aligned} \right. 
\end{equation}
By the maximum principe we have 
\begin{equation} \label{projbulle:l:2}
0 < P B_{z,\mu} (x) \le B_{z,\mu}(x)
\end{equation} for every $x \in \Om$. We will let in the following 
\begin{equation} \label{kernels}
\begin{aligned}
  T_{z,\mu} & = \text{Span} \Big \{  P B_{z,\mu}, \frac{\partial}{\partial \mu} P B_{z,\mu},  \frac{\partial}{\partial z_i} P B_{z,\mu} , 1 = i \cdots n\Big \} 
   \end{aligned} 
 \end{equation}
and we will let $T_{z,\mu}^{\perp}$ be the orthogonal complement with respect to the $H^1_0(\Om)$ scalar product. We then have the following result:

\begin{prop} \label{analyse:asympto:onebubble}
Let $\Om \subset \R^3$ be smooth open bounded. Let $i \ge 0$, let $\bar{\lambda} \in (\Lambda_i, \Lambda_{i+1})$ and let $(u_\lambda)_{\lambda}$, for $\lambda$ close to $\bar{\lambda}$, be a family of solutions of \eqref{eq:critlambda} that satisfies $E(u_\lambda) = \int_{\Om} |u_\lambda|^{6}\, dx \le  K_3^{-3}$ for all $\lambda$ and $\Vert u_\lambda \Vert_{\infty} \to + \infty$ as $\lambda \to \bar{\lambda}$. Then, up to a subsequence as $\lambda \to \bar{\lambda}$, there exist families $(x_\lambda)_{\lambda}$ of points of $\Om$, $(\mu_\lambda)_{\lambda}$ of positive real numbers and $(\alpha_\lambda)_{\lambda}$ of real numbers such that $\mu_\lambda \to 0$, $\alpha_\lambda \to 1$, $u_\lambda = \alpha_\lambda PB_{x_\lambda, \mu_\lambda} + o(1)$ in $H^1_0(\Om)$ and such that $x_\lambda \to x_0 \in \Om$ with $m_{\bar{\lambda}}(x_0) = 0$.
\end{prop}

This result is a variation, for sign-changing solutions, of similar and by now classical results that have been proven for positive solutions (see for instance \cite{Druetdim3, FrankKonigKovarik3, Rey}) or for almost-minimisers of \eqref{eq:critlambda} in the coercive case $\lambda \in (0, \Lambda_1)$ (see e.g. \cite{FrankKonigKovarik3autre}). A generalisation of Proposition \ref{analyse:asympto:onebubble}, for sign-changing solutions that may have energy higher than $K_3^{-3}$, was recently given in \cite{CheikhAliPremoselli}. 

\begin{proof}
We let $u_\lambda$ be as in the statement of Proposition \ref{analyse:asympto:onebubble}. Since $u_\lambda$ solves \eqref{eq:critlambda}, Struwe's celebrated $H^1_0(\Om)$-compactness result \cite[Chap. III, Theorem 3.1]{StruweVariationalMethods} shows that there exist families $(\tilde{x}_\lambda)_{\lambda}$ of points of $\Om$ and $(\tilde{\mu}_\lambda)_{\lambda}$ of positive real numbers converging to $0$ such that $\frac{d(\tilde{x}_\lambda, \partial \Om)}{\tilde{\mu}_\lambda} \to + \infty$ and 
\begin{equation} \label{assumption:l}
 u_\lambda = \tilde{B}_\lambda + o(1) \quad \text{  in } H^1_0(\Om)
 \end{equation}
 as $\lambda \to \bar{\lambda}$, where we have let $\tilde{B}_\lambda = B_{\tilde{x}_\lambda, \tilde{\mu}_\lambda}$. By \eqref{assumption:l} and by assumption we thus have $\lim_{\lambda \to \bar{\lambda}} E(u_\lambda) = K_3^{-3}$ from below. Throughout this proof all the limits are taken up to a subsequence a $\lambda \to \bar{\lambda}$. We let $x_0 = \lim_{\lambda \to \bar{\lambda}} \tilde{x}_\lambda \in \overline{\Om}$.  
 \medskip
 
 \textbf{Step $1$:} we first claim that the following holds:
\begin{equation} \label{x0:interieur}
x_0 \in \Om. 
\end{equation}
We adapt a simple argument in \cite[Proposition 4.1]{FrankKonigKovarik3autre}, initially written for positive solutions, that immediately extends to sign-changing solutions. We only sketch it here. First, a classical argument \`a la \cite[Proposition 7]{BahriCoron} shows, up to changing $u_\lambda$ in $-u_\lambda$, that  there exist families $(x_\lambda)_{\lambda}$ of points of $\Om$, $(\mu_\lambda)_{\lambda}$ of positive real numbers converging to $0$ and $(\alpha_\lambda)_{\lambda}$ of real numbers such that 
 \begin{equation} \label{bahricoron}
  \begin{aligned} 
 u_\lambda & = \alpha_\lambda \big( P B_{x_\lambda, \mu_\lambda} + s_\lambda \big) ,\quad \frac{d(x_\lambda, \partial \Om)}{\mu_\lambda}  \to + \infty, \\
\alpha _\lambda & \to 1, \quad s_\lambda  \in T_{x_\lambda, \mu_\lambda}^{\perp} \quad \text{ and } \quad  \Vert s_\lambda \Vert_{H_0^1} = o(1)
 \end{aligned} 
\end{equation}
as $\lambda \to \bar{\lambda}$, and where $T_{x_\lambda, \mu_\lambda}$ is as in \eqref{kernels}. As explained in \cite[Proposition 7]{BahriCoron} we have $|x_\lambda - \tilde{x}_\lambda| \to 0$ as $\lambda \to \bar{\lambda}$, so that we still have $\lim_{\lambda \to \bar{\lambda}} x_\lambda = x_0$. Using \eqref{bahricoron}, we obtain that 
$$ \begin{aligned} 
 \frac{1}{\alpha_\lambda^2} \int_{\Om}\big( |\nabla u_\lambda|^2 - \lambda u_\lambda^2 \, \big) dx =  K_3^{-3} + C_0 m_0(x_\lambda) \mu_\lambda +  \int_{\Om}\big( |\nabla s_\lambda|^2 - \lambda u_\lambda^2 \big)\, dx \\
 +  o(\mu_\lambda) + O (\mu_\lambda^{\frac12} \Vert s_\lambda \Vert_{H^1_0}) + o \Big( \frac{\mu_\lambda}{d(x_\lambda, \partial \Om)} \Big)+ o(\Vert s_\lambda \Vert_{H^1_0}^2 )
\end{aligned} $$ 
for some positive constant $C_0$ independent of $\lambda$, and that 
$$  \begin{aligned} 
 \frac{1}{\alpha_\lambda^6}  \int_{\Om} u_\lambda^6 \, dx =   K_3^{-3} + 2 C_0 m_0(x_\lambda) \mu_\lambda + 15 \int_{\Om} PB_{x_\lambda, \mu_\lambda}^4 s_\lambda^2 \, dx \\
 + o \Big( \frac{\mu_\lambda}{d(x_\lambda, \partial \Om)} \Big)+ o(\Vert s_\lambda \Vert_{H^1_0}^2 ),
\end{aligned} $$ 
where $m_0(x_\lambda)$ is defined as in \eqref{eq:asymG}. For the details, see \cite[Proposition 4.1]{FrankKonigKovarik3autre}. Since $u_\lambda$ solves \eqref{eq:critlambda}, we have
$$ \frac{\int_{\Om} \big( |\nabla u_\lambda|^2 - \lambda u_\lambda^2\big) \, dx}{\left( \int_{\Om} |u_\lambda|^{6}\, dx\right)^{\frac13}} = \left( \int_{\Om} |u_\lambda|^{6}\, dx\right)^{\frac23} \le K_3^{-2}.$$
Using the previous expansions to expand the left-hand side of the latter yields 
\begin{equation} \label{x0:interieur:2}
 \begin{aligned} 
 \int_{\Om} \Big[ |\nabla s_\lambda|^2 - \lambda u_\lambda^2 \,  - 15 PB_{x_\lambda, \mu_\lambda}^4 s_\lambda^2\Big] \, dx - C_0 m_0(x_\lambda) \mu_\lambda  \\
 \le o(\Vert s_\lambda \Vert_{H^1_0}^2 ) +  o(\mu_\lambda) + O (\mu_\lambda^{\frac12} \Vert s_\lambda \Vert_{H^1_0}) + o \Big( \frac{\mu_\lambda}{d(x_\lambda, \partial \Om)} \Big)
  \end{aligned}
\end{equation} 
as $\lambda \to \bar{\lambda}$. It is well-known that there exists $\eta >0$ such that for any $\lambda$ small enough,
$$  \int_{\Om} \Big[ |\nabla s|^2 - \lambda s^2 \,  - 15 PB_{x_\lambda, \mu_\lambda}^4 s^2\Big] \, dx \ge \eta \Vert s \Vert^2 \quad \text{ holds for all } s \in H^1_0(\Om) $$
(see e.g.  \cite[Equation (D.1)]{Rey}). It was also proven in \cite[Equation (2.8)]{Rey} that $\sup_{y \in \Om} d(y, \partial \Om) m_0(y)  <0$. Combining these observations in \eqref{x0:interieur:2} then shows that there exists $C>0$ independent of $\lambda$ such that 
$$ \Vert s_\lambda \Vert_{H^1_0}^2 + \frac{\mu_\lambda}{d(x_\lambda, \partial \Om)} \le C \mu_\lambda^{\frac12} \Vert s_\lambda \Vert_{H^1_0} + o(\mu_\lambda) \quad \text{ as } \lambda \to \bar{\lambda}.$$
The latter first implies that $\Vert s_\lambda \Vert_{H^1_0} = O(\mu_\lambda^{\frac12})$ and then that $d(x_\lambda, \partial \Om) \not \to 0$ as $\lambda \to \bar{\lambda}$, which proves \eqref{x0:interieur}.

\medskip

\textbf{Step $2$: improved pointwise estimates.} For simplicity we will let $B_\lambda = B_{x_\lambda, \mu_\lambda}$. We claim that there is a constant $C >0$ such that for every $\lambda$ we have
\begin{equation} \label{theorie:C0:l}
|u_\lambda(x) | \le C B_\lambda(x)\quad \text{ for every } x \in \Om.	
\end{equation}
Pointwise estimates like \eqref{theorie:C0:l}, originally developed for positive solutions of critical elliptic equations (see e.g. \cite{Druetdim3, DruetHebeyRobert}), have recently been generalised to sign-changing solutions that exhibit bubbling profiles that may not be non-degenerate: see e.g. \cite{Premoselli13, PremoselliRobert, PremoselliVetois4}. The proof of \eqref{theorie:C0:l} is essentially contained in the proof of \cite[Theorem 2.1]{CheikhAliPremoselli} (up to assuming $v_\infty = 0$, with the notations of \cite{CheikhAliPremoselli}). We simplify the proof of \cite[Theorem 2.1]{CheikhAliPremoselli} here using \eqref{x0:interieur}. First, an adaptation of the arguments in \cite[Proposition 2.2]{CheikhAliPremoselli} shows that
\begin{equation} \label{theorie:C0:l:1}
\max_{x \in \overline{\Om}} (\mu_\lambda + |x-x_\lambda|)^{\frac{n-2}{2}} | u_\lambda(x) - P B_\lambda(x) | \to 0 
\end{equation}
as $\lambda \to \bar{\lambda}$. For $\rho>0$ small enough, we define
\begin{equation*}
	\eta_{\lambda}(\rho):=\sup_{\Om\backslash B(x_0, \rho)} \big| u_\lambda(x)\big|.
\end{equation*} 
Thanks to \eqref{assumption:l} and standard elliptic theory we obtain that, for any fixed $\rho >0$,
\begin{equation}\label{theorie:C0:l:2}
	\lim_{\lambda \to \bar{\lambda}} \eta_{\lambda}(\rho)=0.
\end{equation} 
We now claim that for any $\theta \in (0,\frac{1}{2})$ there exists $R>0$, $\rho>0$, and $C>0$ such that 
\begin{equation}\label{theorie:C0:l:3}
\big|u_{\la}(x)\big|\leq C\left( \frac{\mu_\lambda^{\frac{n-2}{2}-\theta(n-2)}}{|x-\xl|^{(n-2)(1-\theta)}}+\frac{\eta_{\lambda}(\rho)}{|x-\xl|^{(n-2)\theta}}\right) 
\end{equation} 
holds for all $\lambda$ and for all $ x\in \Om \backslash B(\xl, R \mu_\lambda )$. We prove \eqref{theorie:C0:l:3}: we let $G_0$ be the Green's function of $-\Delta$ in $\Om$ with Dirichlet boundary conditions. Since $x_0 \in \Om$, it follows from \cite{RobDirichlet} that there exists $c,C>0$ such that 
\begin{equation} \label{prop:greeen:l:1}  
	c  |\xl - x|^{-1} \le G_0 (\xl, x) \leq C  |\xl - x|^{-1} \bb{ for all } x\in \overline{\Om} \backslash \{\xl \}
\end{equation}
and 
\begin{equation} \label{prop:greeen:l:2} 
\frac{|\nabla G_{0}(\xl, x)|}{G_{0}(\xl, x)}\geq c |\xl - x|^{-1},
\end{equation}
for all $x\in B(\xl, \rho) \backslash \{ \xl \}\subset \Om$. We define
 \begin{equation*}
	\psi_{\theta,\la}(x):=\mu_\lambda^{\frac{n-2}{2}-\theta(n-2)}G_{0}(\xl,x)^{1-\theta}+\eta_{\la}(\rho)G_{0}(\xl, x)^{\theta},
\end{equation*}
which is a smooth positive function in $\Om \backslash \{\xl\}$, and we let $L_{\lambda}=-\Delta - \lambda -u_{\la}^{4}$. Straightforward computations show that 
$$ \frac{L_\lambda \psi_{\theta,\la}}{\psi_{\theta,\la}}  = - \lambda - u_\lambda^4 + \theta(1 - \theta)  \frac{|\nabla G_{0}(\xl, x)|^2}{G_{0}(\xl, x)^2}$$
for any $x \in B(\xl, \rho) \backslash B(\xl, R \mu_\lambda )$. Using \eqref{projbulle:l:2} and \eqref{theorie:C0:l:1} we obtain that 
$$ \lim_{R \to + \infty} \limsup_{\lambda \to \bar{\lambda}} \sup_{\Om \backslash B(\xl , R \mu_\lambda )} |\xl - x|^2 u_\lambda(x)^4 = 0. $$
Combining the latter with \eqref{prop:greeen:l:2} now shows that for $R$ large enough and $\rho$ small enough (depending on $\theta$, but both fixed) we have, for any $\lambda$ close enough to $\bar{\lambda}$, $L_\lambda \psi_{\theta,\la} >0$ in $B(\xl, \rho) \backslash B(\xl, R \mu_\lambda )$. By \eqref{theorie:C0:l:1} and the definition of $\eta_\lambda(\rho)$ there is a constant $C >0$ such that $u_\lambda \le C \psi_{\theta, \lambda}$ in $\partial \big( B(\xl, \rho) \backslash B(\xl, R \mu_\lambda )\big)$. Since independently we have $L_\lambda u_\lambda = 0$ by \eqref{eq:critlambda} the maximum principle shows that $u_\lambda \le C \psi_{\lambda, \theta}$ in $B(\xl, \rho) \backslash B(\xl, R \mu_\lambda )$. Together with the definition of $\eta_\lambda(\rho)$ this proves \eqref{theorie:C0:l:3}. 

\medskip

We now prove \eqref{theorie:C0:l}. First, arguing as in the proof of \cite[Proposition 2.4]{CheikhAliPremoselli}, a representation formula with \eqref{theorie:C0:l:3} shows that there exists $\rho >0$ such that 
\begin{equation} \label{theorie:C0:l:4}
 |u_\lambda(x)| \le C \big( B_\lambda(x) + \eta_\lambda(\rho) \big) 
 \end{equation}
for every $x \in \Om$ and $\lambda$ close to $\bar{\lambda}$, for some $C >0$. We proceed by contradiction and assume that, up to a subsequence as $\lambda \to \bar{\lambda}$, we have  $\eta_\lambda(\rho) >> \mu_\lambda^{\frac12}$. We let  $U_\lambda = \frac{u_\lambda}{\eta_\lambda(\rho)}$. For any $\lambda$ we let $z_\lambda \in \Om\backslash B(\xl , \rho)$ be such that $|u_\lambda(z_\lambda)| = \eta_\lambda(\rho)$. By \eqref{theorie:C0:l:4} we see that for any $\delta >0$ fixed we have $|U_\lambda(z_\lambda)|=1$ and 
\begin{equation*}
|U_\lambda(x)| \le C + o(1) \quad \text{ for } x \in \Om \backslash B(\xl, \delta). 
\end{equation*}
in $\Om$. Since $\eta_\lambda(\rho) \to 0$ by  \eqref{theorie:C0:l:2}, standard elliptic theory shows that $U_\lambda \to U_\infty$ in $C^2_{loc}(\overline{\Om} \backslash \{x_0\}$ as $\lambda \to \bar{\lambda}$, where $U_\infty$ satisfies $|U_\infty(x)| \le C$ for any $x \neq x_0$, $(- \Delta - \bar{\lambda})U_\infty = 0$ and $ \Om \backslash \{x_0\}$.  In particular, the singularity of $U_\infty$ at $x_0$ is removable and $U_\infty$ satisfies weakly $-\Delta U_\infty - \bar{\lambda} U_\infty  =0$ in $\Om$. Since $- \Delta - \bar{\lambda}$ has no kernel by the assumption $\bar{\lambda} \in (\Lambda_i, \Lambda_{i+1})$ this shows that $U_\infty \equiv 0$. Independently, if we let $z_\infty = \lim_{\lambda \to \bar{\lambda}} z_\lambda$, the $C^2_{loc}$ convergence shows that $|U_\infty(z_\infty)| = 1$, hence $U_\infty \not \equiv 0$. This is a contradiction, and thus $\eta_\lambda(\rho) = O( \mu_\lambda^{\frac12})$ as $\lambda \to \bar{\lambda}$. Plugging the latter in \eqref{theorie:C0:l:4} then shows that $ |u_\lambda(x)| \le C B_\lambda(x) $ in $\Om$ and concludes the proof of \eqref{theorie:C0:l}.

\medskip

\textbf{Step $3$: conclusion.} Once \eqref{theorie:C0:l} is proven the fact that $m_{\bar{\lambda}}(x_0) = 0$ follows from a simple Pohozaev identity. We refer for instance to \cite[Proposition 3.1]{CheikhAliPremoselli} for a proof (the statement of \cite[Proposition 3.1]{CheikhAliPremoselli} assumes coercivity of $-\Delta - \lambda$ but this is not used in the proof).
\end{proof}

The analysis of blowing-up solutions of \eqref{eq:critlambda} has originated a vast literature in the last decades. As $\lambda \to 0$ (when $n \ge 4$) or $\lambda \to \lambda_{*,0}$ (when $n=3$) from the right, for instance, minimal solutions constructed in \cite{BN, Druetdim3} may blow-up and their behavior has been investigated e.g. in \cite{Druetdim3, HanBN, Rey}. A priori descriptions of blowing-up solutions of \eqref{eq:critlambda} were obtained in  \cite{AGGP, BAEMP1, BAEMP2, FrankKonigKovarik3, HebeyBN, IacopettiPacella, IacopettiVaira2, LaurainKonig1, LaurainKonig2, Rey}. Constructive results for blowing-up solutions of \eqref{eq:critlambda} were obtained in \cite{LiVairaWeiWu, MussoPistoia2, MussoSalazar, PistoiaRagoVaira, PistoiaRocci, PistoiaVaira2, Premoselli12, SunWeiYang, SunWeiYang2} and the references therein.

\section{Properties of the mass functions} \label{annexe:masse}

In this last section we assume that $n=3$. We let $\lambda \not \in \Sp(- \Delta)$ and $G_\lambda$ be defined as in  \eqref{green:coercif}. For any $x \in \Om$ and $y \neq x$ we let
$$ H_\lambda(x,y) = G_\lambda(x,y) - \frac{1}{4\pi |x-y|}. $$
By  \eqref{green:coercif} it is easily see that $H_\lambda(x,\cdot)$ satisfies 
\begin{equation} \label{eq:Hlambda}
\left \{ \begin{aligned}
- \Delta H_\lambda(x, \cdot) - \lambda H_\lambda(x, \cdot) &=  \frac{\lambda}{4\pi |x-y|} \quad \text{ in } \Om, \\
  H_\lambda(x, \cdot) & = - \frac{1}{4\pi |x-y|} \quad \text{ in } \partial \Om.
\end{aligned} \right. 
\end{equation}
Standard elliptic theory and \eqref{eq:Hlambda} then show that, for any $x \in \Om$ fixed, $H_\lambda(x, \cdot) \in C^{0,1}(\overline{\Om})$. Using the expansion \eqref{eq:asymG} it is then easily seen that 
$$ m_\lambda(x) = H_\lambda(x,x). $$
Since $G_\lambda$ is symmetric in $x,y$ then so is $H_\lambda$. In this section we investigate several properties of the mass function. This function has been thoroughly studied in the coercive case $\lambda \in (0, \Lambda_1)$ (see e.g.\cite{DelPinoDolbeaultMusso}) but little is known in the non-coercive case $\lambda \ge \Lambda_1$. We prove several properties of $m_\lambda$ and relate its asymptotic behavior as $\lambda$ approaches $\Sp(-\Delta)$ to the nodal sets of eigenfunctions of $- \Delta$ in $\Om$. We first prove that the mass function is increasing as $\lambda$ varies between two consecutive eigenvalues: 

\begin{lemme}
Let $i \ge 0$ and let $\lambda \in (\Lambda_i, \Lambda_{i+1})$. Let $x \in \Om$ be fixed. Then 
\begin{equation} \label{der:mass:lambda}
\frac{d}{d \lambda} m_{\lambda}(x) = \int_{\Om} G_\lambda(x, y)^2 \, dy >0.
\end{equation}
\end{lemme}

\begin{proof}
Let $\lambda \in (\Lambda_i, \Lambda_{i+1})$ and $x \in \Om$ be fixed and let $\ve >0$ be small enough. By definition of $G_\lambda$ we have 
$$  - \Delta \big(  G_{\lambda+\ve}(x, \cdot) - G_\lambda(x, \cdot) \big) - (\lambda + \ve)(G_{\lambda+\ve}(x, \cdot)- G_\lambda(x, \cdot)) = \ve G_\lambda(x,\cdot). $$
Since $G_{\lambda+\ve}(x, \cdot) - G_\lambda(x, \cdot)$ vanishes in $\partial \Om$ and by \eqref{eq:estimateGp}, standard elliptic theory shows that there exists $C >0$ independent of $x$ such that  
\begin{equation} \label{der:mass:lambda:1}
\Vert  G_{\lambda+\ve}(x, \cdot) - G_\lambda(x, \cdot) \Vert_{L^\infty(\Om)} \le C\ve.
\end{equation}
A representation formula also gives, for $y \in \Om$,  
$$ H_{\lambda+\ve}(x,y) - H_\lambda(x,y) = G_{\lambda+\ve}(x,y) - G_\lambda(x, y) = \ve \int_{\Om}G_{\lambda+\ve}(y, z)G_\lambda(x,z) \, dz .$$
Expanding the integral in the right-hand side with \eqref{der:mass:lambda:1} then yields 
$$ H_{\lambda+\ve}(x,y) - H_\lambda(x,y) = \ve \int_{\Om}G_{\lambda}(y, z)G_\lambda(x,z) \, dz + O(\ve^2), $$
and applying the latter at $y=x$ yields \eqref{der:mass:lambda}.
\end{proof}
Note that in the coercive case $\lambda \in (0, \Lambda_1)$ the monotony of $\lambda \mapsto m_\lambda(x)$ easily follows from the maximum principle (see e.g. \cite{DelPinoDolbeaultMusso}). We now investigate the behavior of the mass function as $\lambda$ approaches $\Sp(-\Delta)$. We first prove that interior points of positive mass always exist provided $\lambda$ approaches $\Sp(-\Delta)$ from the left:

\begin{prop} \label{prop:massepositive}
Let $i \ge 0$ and let $\lambda \in (\Lambda_i, \Lambda_{i+1})$. Let $K \subset \subset \Om$ be a compact set of positive measure. Then  
$$ \max_{K} m_\lambda \to + \infty \quad \text{ as } \lambda \underset{< }{\to} \Lambda_{i+1} . $$
\end{prop}

\begin{proof}
Let $E_{\Lambda_{i+1}}(\Om)$ denote the eigenspace associated to $\Lambda_{i+1}$ and let $d = \dim E_{\Lambda_{i+1}}(\Om)$, with the convention that $d=0$ if $i=0$. We let $(\vp_1, \cdots, \vp_d)$ be a $L^2(\Om)$-orthonormal basis of $E_{\Lambda_{i+1}}(\Om)$. Let $K$ be a compact subspace of $\Om$ of positive measure and let $x \in K$. We decompose $H_\lambda(x, \cdot)$ as 
$$ H_{\lambda}(x, y) = S_\lambda(x,y) + \sum_{j=1}^d a_j^{\lambda}(x) \vp_j(y), $$
where $S_\lambda(x,\cdot)$ satisfies $\int_{\Om} S_\lambda(x,\cdot) \vp_j \, dy = 0$ for all $1 \le j \le d$ and where $a_j^{\lambda}(x) = \int_{\Om} H_\lambda(x, \cdot) \vp_j dy$. First, a representation formula for $\vp_j$ at $x$ shows that 
\begin{equation} \label{eq:ajl:2}
 \begin{aligned}
\vp_j(x) & =  (\Lambda_{i+1} - \lambda)  \int_{\Om}  \frac{1}{4\pi |x-y|} \vp_j(y) dy +  (\Lambda_{i+1} - \lambda) a_j^{\lambda}(x).
\end{aligned} 
\end{equation}
Using \eqref{eq:Hlambda} it is easily seen that $S_\lambda(x, \cdot)$ satisfies 
\begin{equation*} 
\left \{\begin{aligned}
- \Delta S_\lambda(x, \cdot) - \lambda S_\lambda(x, \cdot) &=  \frac{\lambda}{4\pi |x-y|} - (\Lambda_{i+1} - \lambda) \sum_{j=1}^d a_j^{\lambda}(x) \vp_j(y)  \quad \text{ in } \Om, \\
  S_\lambda(x, \cdot) & = - \frac{1}{4\pi |x-y|} \text{ in } \partial \Om.
\end{aligned} \right.
\end{equation*}
Since $x \in K$, standard elliptic theory then shows that $\Vert S_\lambda(x, \cdot) \Vert_{C^{0,1}(\overline{\Om})} = O(1)$ as $\lambda \underset{< }{\to} \Lambda_{i+1}$. Thus, using \eqref{eq:ajl:2}, we have
\begin{equation} \label{masse:coeur}
\begin{aligned}
 m_\lambda(x)  = S_\lambda(x,x)  +  \sum_{j=1}^d a_j^{\lambda}(x) \vp_j(x) 
  =  \frac{1}{\Lambda_{i+1} - \lambda}  \sum_{j=1}^d \vp_j(x)^2 + O(1) 
  \end{aligned} 
  \end{equation} 
  as $\lambda \underset{< }{\to} \Lambda_{i+1}$. By the unique continuation results in \cite{HardtSimon}, each $\vp_j$ vanishes on a set of measure zero in $\Om$. Since $|K| >0$, we may thus assume that $x \in K$ is chosen so that $ \sum_{j=1}^d \vp_j(x)^2 >0$, and \eqref{masse:coeur} shows that $m_\lambda(x) \to + \infty$ as $\lambda \underset{< }{\to} \Lambda_{i+1}$. 
  \end{proof}

We now investigate the behavior of $m_{\lambda}$ as $\lambda$ approaches $\Sp(-\Delta)$ from the right. The situation in this case turns out to be surprisingly more involved and the nodal sets of the eigenfunctions associated to $\Lambda_i$ arise in the expansion of the mass. We first prove that if the nodal sets of the eigenfunctions of $\Lambda_i$ have no common point in $\Om$ the mass goes to $-\infty$ everywhere in $\Om$ as $\lambda$ approaches $\Sp(-\Delta)$ from the right. As before we let $E_{\Lambda_{i}}(\Om)$ be the eigenspace associated to $\Lambda_{i}$, we let $d_i = \dim E_{\Lambda_{i}}(\Om)$ and we let $(\vp_1, \cdots, \vp_{d_i})$ be a $L^2(\Om)$-orthonormal basis of $E_{\Lambda_{i}}(\Om)$. 
  
  \begin{prop} \label{prop:massenegative}
Let $i \ge 1$. We assume that $\cap_{j=1}^{d_i} \{\vp_j = 0 \} = \emptyset$. Then
$$ \max_{\Om} m_\lambda \to - \infty \quad \text{ as } \lambda \underset{> }{\to} \Lambda_{i} . $$
\end{prop}

\begin{proof}
Let $\lambda \in (\Lambda_i, \Lambda_{i+1})$. For any $x \in \Om$ fixed we decompose $G_\lambda(x, \cdot)$ as 
$$ G_\lambda(x, \cdot) = \bar{G}_\lambda(x, \cdot) + \sum_{j=1}^{d_i} a_j^{\lambda}(x) \vp_j(y), $$
where $\bar{G}_\lambda(x,\cdot)$ satisfies $\int_{\Om} \bar{G}_\lambda(x,\cdot) \vp_j \, dy = 0$ for all $1 \le j \le d_i$. Since, for all $1 \le j \le d_i$, we have $\vp_j = \frac{1}{\Lambda_i - \lambda} (- \Delta - \lambda)\vp_j$, a representation formula for $\vp_j$ yields
\begin{equation} \label{eq:ajl:0}
 \begin{aligned}
a_j^{\lambda}(x) = -\frac{1}{\lambda - \Lambda_i} \vp_j(x). \end{aligned} 
\end{equation}
By definition of $G_\lambda(x, \cdot)$ and with \eqref{eq:ajl:0} it is easily seen that $\bar{G}_\lambda(x, \cdot)$ satisfies 
\begin{equation} \label{eq:barGlambda}
\left \{\begin{aligned}
- \Delta \bar{G}_\lambda(x, \cdot) - \lambda \bar{G}_\lambda(x, \cdot) &= \delta_x -  \sum_{j=1}^{d_i} \vp_j(x) \vp_j(y)  \quad \text{ in } \Om, \\
  \bar{G}_\lambda(x, \cdot) & = 0 \quad \text{ in } \partial \Om,
\end{aligned} \right.
\end{equation}
for any $\lambda \in (\Lambda_i, \Lambda_{i+1})$. We observe that $\delta_x = \text{div} F_x$ in the distributional sense, where $F_x(y) = \frac{x-y}{|x-y|^3}$ for $y \in \Om$. For a fixed $1 \le p < \frac32$, $F_x \in L^p(\Om)$. Since, for every $\lambda \in (\Lambda_i, \Lambda_{i+1})$ and $x\in \Om$, $\bar{G}_\lambda(x, \cdot)$ is orthogonal to $E_{\Lambda_i}(\Om)$ for the $L^2(\Om)$-scalar product, classical arguments from standard elliptic theory (see for instance \cite[Theorem 7.1, Theorem 7.4]{GiaquintaMartinazzi}) show that for any $1 \le p < \frac32$ fixed there exists $C = C(\Om,p) >0$ such that, for every $\lambda \in (\Lambda_i, \Lambda_{i+1})$ and $x\in \Om$, we have 
$$ \Vert \bar{G}_\lambda(x,\cdot) \Vert_{W^{1,p}_0(\Om)} \le C . $$
If we choose $p \ge \frac65$, Sobolev's embeddings then show that there exists $C >0$ such that, for every $\lambda \in (\Lambda_i, \Lambda_{i+1})$ and $x\in \Om$
\begin{equation} \label{estL2uniforme}
\Vert \bar{G}_\lambda(x,\cdot) \Vert_{L^2(\Om)} \le C. 
\end{equation}
We now let, for $x \neq y$ in $\Om$, 
$$\bar{H}_\lambda(x,y) = \bar{G}_\lambda(x,y) - \frac{1}{4 \pi |x-y|}.$$ The definition of $m_\lambda$ and \eqref{eq:ajl:0} then yield, for any $x \in \Om$ and $\lambda \in (\Lambda_i, \Lambda_{i+1})$,
\begin{equation} \label{masse:explose}
\begin{aligned}
 m_\lambda(x) & = \bar{H}_\lambda(x,x)  - \frac{1}{\lambda-\Lambda_i}  \sum_{j=1}^{d_i} \vp_j(x)^2. 
  \end{aligned} 
  \end{equation}
Independently, using \eqref{eq:barGlambda} it is easily seen that $\bar{H}_\lambda(x, \cdot)$ satisfies 
\begin{equation} \label{eq:Slambda}
\left \{\begin{aligned}
- \Delta \bar{H}_\lambda(x, \cdot) - \lambda \bar{H}_\lambda(x, \cdot) &=  \frac{\lambda}{4\pi |x-y|} -  \sum_{j=1}^{d_i} \vp_j(x) \vp_j(y)  \quad \text{ in } \Om, \\
  \bar{H}_\lambda(x, \cdot) & = - \frac{1}{4\pi |x-y|} \quad \text{ in } \partial \Om,
\end{aligned} \right.
\end{equation}
and standard elliptic theory shows that, for any $x \in \Om$ fixed,  $\bar{H}_\lambda(x, \cdot) \in C^{0,1}(\overline{\Om})$. Let now $G_0$ denote the Green's function of $- \Delta$ in $\Om$ with Dirichlet boundary conditions, which solves  \eqref{green:coercif} with $\lambda = 0$. Standard estimates (see again \cite{RobDirichlet}) show that $0 \le G_0(x,y) \le C|x-y|^{-1}$ for $x\neq y$ in $\Om$. Using \eqref{estL2uniforme} we have $\int_{\Om} G_0(x,z) \bar{H}_\lambda(x,z) dx = O(1)$ as $\lambda \to \bar{\lambda}$. With the latter, a representation formula for $\bar{H}_\lambda(x, \cdot)$ shows that there exists $C >0$ such that, for every $\lambda \in (\Lambda_i, \Lambda_{i+1})$ and $x\in \Om$, we have 
\begin{equation} \label{eq:Slambda:2}
 \bar{H}_\lambda(x, x) \le C + \int_{\partial \Om} \partial_\nu G_0(x,y) \frac{1}{4 \pi |x-y|} d \sigma(y). 
\end{equation}
We claim that 
\begin{equation} \label{masse:moins:infini}
\limsup_{d(x, \partial \Om) \to 0} \int_{\partial \Om} \partial_\nu G_0(x,y) \frac{1}{4 \pi |x-y|} d \sigma(y)= - \infty.
\end{equation}
Assume for the moment that \eqref{masse:moins:infini} holds true. For $\delta >0$ we let in the following $\Omega_\delta = \{ x \in \Om, d(x, \partial \Om) < \delta \}$. For any $L > 0$ fixed, by \eqref{masse:explose}, \eqref{eq:Slambda:2} and \eqref{masse:moins:infini} we may choose $\delta$ small enough such that for every $\lambda \in (\Lambda_i, \Lambda_{i+1})$,
\begin{equation} \label{masse:moins:infini:1}
\max_{\Om_\delta} m_\lambda(x) \le - L .
\end{equation}
Independently, since every $\vp_j$, $1 \le j \le d_i$ is smooth in $\overline{\Om}$, the assumption $\cap_{j=1}^{d_i} \{\vp_j = 0 \} = \emptyset$ implies that $\min_{x \in \Om \backslash \Om_\delta} \sum_{j=1}^{d_i} \vp_j(x)^2 >0$. Using \eqref{masse:explose} we thus have 
\begin{equation} \label{masse:moins:infini:2}
\lim_{ \lambda \underset{> }{\to} \Lambda_{i}} \max_{\Om \backslash \Om_\delta} m_\lambda(x)  = - \infty. 
\end{equation}
Combining \eqref{masse:moins:infini:1} and \eqref{masse:moins:infini:2} finally shows that for every $L >0$ fixed,
$$\limsup_{\lambda \underset{> }{\to} \Lambda_{i} } \max_{\Om} m_{\lambda} \le - L$$
holds. Letting $L \to + \infty$ concludes the proof of Proposition \ref{prop:massenegative}. We thus only have to prove \eqref{masse:moins:infini}. Let, for $z \in \Om$, $u(z) = \int_{\partial \Om} \partial_\nu G_0(z,y) \frac{1}{4 \pi |x-y|} d \sigma(y)$. It is easily seen that $u$ is the unique solution of $-\Delta u =0$ in $\Om$ with $u = - \frac{1}{4\pi |x-\cdot|}$ in $\partial \Om$. As a consequence, and by \eqref{eq:Hlambda}, we thus have $u(z) = H_0(x,z)$ and thus
$$ \int_{\partial \Om} \partial_\nu G_0(x,y) \frac{1}{4 \pi |x-y|} d \sigma(y) = H_0(x,x) = m_0(x). $$
It is well-known that $m_0(x) \sim \frac{-1}{8\pi d(x, \partial \Om)} \to - \infty$ as $d(x, \partial \Om) \to 0$ (see e.g. \cite[Equation (2.8)]{Rey}), and this concludes the proof of \eqref{masse:moins:infini}.
\end{proof}

For $i \ge 0$ we recall the definition of $\lambda_{i,*}$ introduced in \eqref{deflambdastar}:
$$ \lambda_{i,*}= \inf \big \{ \lambda \in (\Lambda_i, \Lambda_{i+1}), \, \,  \max_{\Om} m_\lambda >0 \big \}, $$
where we used the convention $\Lambda_0 = 0$. Proposition \ref{prop:massepositive} shows that $\lambda_{i,*} < \Lambda_{i+1}$. It is well-known that $\lambda_{0,*} \in (0, \Lambda_1)$ (see \cite[Equation (1.51)]{BN}). When $i=1$, since $E_{\Lambda_1}(\Om)$ is spanned by a positive eigenvector $\vp_1$, Proposition \ref{prop:massenegative} shows that we similarly have 
\begin{equation} \label{lambda:star:1} \lambda_{1,*} \in (\Lambda_1, \Lambda_2). \end{equation}
When $i \ge 2$ and on a general smooth bounded domain $\Om$, however, we only have $\lambda_{i,*} \ge \Lambda_i$. By Proposition \ref{prop:massenegative} the strict inequality holds when the elements in $E_{\Lambda_i}(\Om)$ have no common zero, but this condition is complicated to verify. For a general $i \ge 2$, suprisingly, the asymptotic behavior of the mass function $x \mapsto m_\lambda(x)$ as $\lambda \underset{> }{\to} \Lambda_{i}$ is related to the behavior of the nodal lines of the eigenvectors in $E_{\Lambda_i}(\Om)$. To illustrate this fact we conclude this section by proving a generalisation of Proposition \ref{prop:massenegative}. If $i \ge 2$ we let $\bar{G}_{\Lambda_i}$ be the Green's function of $- \Delta - \Lambda_i$, that is the unique function $\bar{G}_{\Lambda_i} \in L^1(\Om \times \Om)$ such that for all $x \in \Om$, $\bar{G}_{\Lambda_i} (x, \cdot)$ is $L^2(\Om)$-orthogonal to $E_{\Lambda_i}(\Om)$ and satisfies
\begin{equation} \label{eq:Glambda:;noyau}
\left \{\begin{aligned}
- \Delta \bar{G}_{\Lambda_i}(x, \cdot) - \Lambda_i \bar{G}_{\Lambda_i}(x, \cdot) &= \delta_x -  \sum_{j=1}^{d_i} \vp_j(x) \vp_j(y)  \quad \text{ in } \Om, \\
  \bar{G}_{\Lambda_i}(x, \cdot) & = 0 \quad \text{ in } \partial \Om.
\end{aligned} \right.
\end{equation}
Again by standard elliptic theory, for any fixed $x \in \Om$ we may expand $\bar{G}_{\Lambda_i}(x, \cdot)$ as 
$$ \bar{G}_{\Lambda_i}(x, y) = \frac{1}{4 \pi |x-y|} + \bar{m}_{\Lambda_i}(x) + O(|x-y|) \quad \text{ as } y \to x. $$
We then have the following result:

  \begin{prop} \label{prop:comportement:masse}
Let $i \ge 2$ and $x \in \Om$ be fixed. 

\begin{itemize}
\item Assume that $ x \not \in \cap_{j=1}^{d_i} \{\vp_j = 0 \} $. Then
$$ m_\lambda(x) \to - \infty \quad \text{ as } \lambda \underset{> }{\to} \Lambda_{i} . $$
\item Assume that $ x \in \cap_{j=1}^{d_i} \{\vp_j = 0 \} $. Then 
$$ m_{\lambda}(x) \to  \bar{m}_{\Lambda_i}(x) \quad \text{ as } \lambda \underset{> }{\to} \Lambda_{i}. $$
\end{itemize}
\end{prop}

\begin{proof}
We use the notations in the proof of Proposition~\ref{prop:massenegative}. Estimate \eqref{estL2uniforme} together with \eqref{eq:barGlambda} shows that, for every fixed $x \in \Om$, $\bar{G}_\lambda(x, \cdot)$ weakly converges to $\bar{G}_{\Lambda_i}(x, \cdot)$ in $W^{1,p}_0(\Om)$ for some $ \frac{6}{5} \le p < \frac32$,  up to a subsequence as $\lambda\underset{> }{\to} \Lambda_{i}$. Again \eqref{estL2uniforme} together with \eqref{eq:Slambda} and standard elliptic theory shows that there is $0 < \alpha < 1$ such that $\bar{H}_\lambda(x, \cdot)$ strongly converges in $C^{0,\alpha}(\overline{\Om})$ to $\bar{G}_{\Lambda_i}(x, \cdot) - \frac{1}{4 \pi |x - \cdot|}$. As a consequence, $\bar{H}_\lambda(x,x) \to \bar{m}_{\Lambda_i}(x)$ as $\lambda\underset{> }{\to} \Lambda_{i}$. If $x \in \cap_{j=1}^{d_i} \{\vp_j = 0 \}$, \eqref{masse:explose} shows that $ \bar{H}_\lambda(x,x) =m_\lambda(x) $ and the conclusion follows. If $x \not \in \cap_{j=1}^{d_i} \{\vp_j = 0 \}$ then $\sum_{j=1}^{d_i} \vp_j(x)^2 >0$ and 
again by \eqref{masse:explose} we have $m_\lambda(x) \to - \infty $ as $\lambda\underset{> }{\to} \Lambda_{i}$. 
\end{proof}
Proposition \ref{prop:comportement:masse} thus shows that determining whether the strict inequality $\lambda_{i,*} > \Lambda_i$ holds reduces to investigating the sign of $x\in \Om \mapsto \bar{m}_{\Lambda_i}(x)$. This seems to be a complicated question in general and it may happen that $\lambda_{i,*} = \Lambda_i$ for some $i \ge 2$, although no examples of domains $\Om$ where this holds are known.

\bibliographystyle{amsplain}
\bibliography{biblio}

\end{document}